\documentclass[11pt]{amsart}
\usepackage{amscd}
\usepackage{amssymb}
\usepackage{graphicx}
\newcounter{mtheorem}

\newtheorem{theorem}{Theorem}[section]
\newtheorem{lemma}[theorem]{Lemma}
\newtheorem{prop}[theorem]{Proposition}
\newtheorem{corollary}[theorem]{Corollary}

\theoremstyle{definition}
\newtheorem{definition}[theorem]{Definition}
\newtheorem{example}[theorem]{Example}

\theoremstyle{remark}
\newtheorem{remark}[theorem]{Remark}

\numberwithin{equation}{section}
\newcommand{\C}{\mathbb{C}}
\newcommand{\R}{\mathbb{R}}
\newcommand{\Z}{\mathbb{Z}}
\newcommand{\N}{\mathbb{N}}

\newcommand{\tnabla}{{\widetilde{\nabla}}}
\newcommand{\tg}{{\tilde{g}}}

\newcommand{\bt}{\boldsymbol{t}}
\title[Desingularizing isolated conical singularities]{Desingularizing
isolated
conical singularities: uniform estimates via weighted Sobolev spaces}
\author[T.~Pacini]{Tommaso~Pacini}
\address{Scuola Normale Superiore, Pisa} \email{tommaso.pacini@sns.it}
\subjclass[2000]{53C21, 58Dxx, 58J05}
\date{\today}
\begin{document}

\begin{abstract}
We define a very general ``parametric connect sum'' construction which can be
used to eliminate isolated conical singularities of Riemannian manifolds.
We then show that various important
analytic and elliptic estimates, formulated in terms of weighted Sobolev spaces,
can be obtained independently of the parameters used in the construction.
Specifically, we prove uniform estimates related to (i) Sobolev Embedding
Theorems, (ii) the invertibility of the Laplace operator and (iii)
Poincar\'{e} and Gagliardo-Nirenberg-Sobolev type inequalities.  

Our main tools are the well-known theories of weighted Sobolev spaces and elliptic operators on ``conifolds''. We  provide an overview of both, together with an extension of the former to general Riemannian manifolds. 

For a geometric application of our results we refer the reader to our paper \cite{pacini:slgluing} concerning desingularizations of special Lagrangian
conifolds in $\C^m$.
\end{abstract}

\maketitle
\tableofcontents


\section{Introduction}\label{s:intro}

It is a common problem in Differential Geometry to produce examples of
(possibly immersed) Riemannian manifolds $(L,g)$ satisfying a given geometric
constraint, usually a nonlinear PDE, on the metric (Einstein, constant scalar curvature,
\textit{etc.}) or on the immersion (constant mean curvature, minimal,
\textit{etc.}). If $L$ (or the immersion) happens to be singular, one then faces
the problem of ``desingularizing'' it to produce a new, smooth,
Riemannian manifold satisfying the same constraint. Often, one actually hopes to
produce a family $(L_t,g_t)$ of manifolds satisfying the constraint and which
converges in some sense to $(L,g)$ as $t\rightarrow 0$.
One typical way to solve this problem is via ``gluing''. We outline this construction as follows, focusing for simplicity on the situation where $L$ has only isolated point singularities and the
constraint is on the metric.

\ 

\textbf{Step 1:} For each singular point $x\in L$, we look for an explicit
smooth ``local model'': \textit{i.e.}, a manifold $(\hat{L},\hat{g})$ which
satisfies a related, scale-invariant, constraint and which, outside of some compact region, is
topologically and metrically similar to an annulus $B(x,\epsilon_1)\setminus
B(x,\epsilon_2)$ in $L$, centered in the singularity. We can then glue $\hat{L}$
onto the manifold $L\setminus B(x,\epsilon_2)$, using the ``neck region''
$B(x,\epsilon_1)\setminus B(x,\epsilon_2)$ to interpolate between the two
metrics. The fact that the neck region is ``small'' is usually not a problem: one can simply
rescale $\hat{g}$ to $t^2\hat{g}$ so that now $(\hat{L},t^2\hat{g})$ is of
similar size. The resulting manifold, which we denote
$(\hat{L}\#L, \hat{g}\#g)$, satisfies the constraints outside of the neck region
simply by construction. If the interpolation is done carefully we also get very good control over what happens on the neck. We think of $(\hat{L}\#L, \hat{g}\#g)$ as an ``approximate solution'' to the gluing problem. Rescaling also gives a way to build families: the idea is
to glue $(\hat{L},t^2\hat{g})$ into $B(x,\epsilon_1)\setminus
B(x,t\epsilon_2)$, producing a family $(L_t,g_t)$; intuitively, as $t\rightarrow
0$ the compact region in $\hat{L}$ collapses to the singular point $x$ and $L_t$
converges to $L$.

\ 

\textbf{Step 2:} We now need to perturb each $(L_t, g_t)$ so
that the resulting family satisfies the constraint globally. Thanks to a
linearization process, the perturbation process often boils down to studying a
linear elliptic system on $g_t$. One of the main problems is to verify that this
system satisfies estimates which are uniform in $t$. This is the key to
obtaining the desired perturbation for all sufficiently small $t$. Roughly
speaking, there is often a delicate balance to be found as $t\rightarrow 0$:
on the one hand, if $L_t$ was built properly, as $t\rightarrow 0$ it will get
closer to solving the constraint; on the other hand, it becomes more
singular. Uniform estimates are important in proving that this balance can be
reached. 

\ 

The geometric problem defines the differential operator to be studied. However, this operator is often fairly intrinsic, and can be defined independently of the geometric specifics. The necessary estimates may likewise be of a much more general nature. Filtering out the geometric ``super-structure'' and concentrating on the analysis of the appropriate category of abstract Riemannian manifolds will then enhance the understanding of the problem, leading to improved results and clarity. The first goal of this paper is thus to set up an abstract framework for dealing with gluing constructions and the corresponding uniform estimates. Here, ``abstract'' means: independent of any specific geometric problem. We focus on gluing constructions concerning Riemannian manifolds with isolated
conical singularities. These are perhaps the simplest singularities possible,
but in the gluing literature they often appear as an interesting and
important case. Our framework involves two steps, parallel to those outlined above. 

\ 

\textbf{Step A:} In Section \ref{s:sums_sobolev} we define a general \textit{connect sum} construction between Riemannian manifolds, extrapolating from standard desingularization procedures. 

\ 

\textbf{Step B:} We show how to produce uniform estimates on these connect sum manifolds, by presenting a detailed analysis of three important problems:
(i) Sobolev Embedding Theorems, (ii) invertibility of the
Laplace operator, (iii) Poincar\'{e} and
Gagliardo-Nirenberg-Sobolev type inequalities. The main results are Theorems
\ref{thm:normstequivalent},
\ref{thm:sum_injective}, \ref{thm:cpt_sum_injective}, \ref{thm:d_invertible} and Corollary \ref{cor:improved_sob}.

\ 

Our Step A is actually
much more general than Step 1, as described above: it is specifically
designed to deal with both compact and non-compact manifolds and it allows us to replace
the given singularity not only with smooth compact regions but also with
non-compact ``asymptotically conical ends'' or even with new singular regions.
It also allows for
different ``neck-sizes'' around each singularity. In this sense it offers a very
broad and flexible framework to work with. 

The range of possible estimates covered by our framework is clearly much wider than the set of Problems (i)-(iii) listed in Step B. Indeed, the underlying, well-known, theory of elliptic operators on conifolds is extremely general. Within this paper, this choice is to be intended as fairly arbitrary: amoungst the many possible, we choose 3 estimates of general interest but differing one from the other in flavour: Problem (i) is of a mostly
local nature, Problems (ii) and (iii) are global. 
In reality, however, our choice of Problems (i)-(iii) is  based on the very specific geometric problems we happen to be interested in. The second goal of this paper is thus to lay down the analytic foundations for our papers \cite{pacini:sldefs}, \cite{pacini:slgluing} concerning deformations and desingularizations of submanifolds whose
immersion map satisfies the \textit{special Lagrangian} constraint. The starting point for this work was a collection of gluing results concerning special Lagrangian submanifolds due to Arezzo-Pacard \cite{arezzopacard}, Butscher \cite{butscher}, Lee \cite{lee} and Joyce \cite{joyce:III}, \cite{joyce:IV}, and parallel results concerning \textit{coassociative} submanifolds due to Lotay \cite{lotay}. It slowly became apparent, thanks also to many conversations with some of these authors, that several parts of these papers could be simplified, improved or generalized: related work is currently still in progress. In particular, building approximate solutions and setting up the perturbation problem requires making several choices which then influence the analysis rather drastically. A third goal of the paper is thus to present a set of choices which leads to very clean, simple and general results. 
One such choice concerns the parametrization of the approximate solutions: parametrizing the necks so that they depend explicitly on the parameter $t$ is one ingredient in obtaining uniform estimates. A second ingredient is the consistent use, even when dealing with compact manifolds, of weighted rather than standard Sobolev spaces. 
Although such choices may seem obvious to some members of the ``gluing community'', it still seems useful to emphasize this point. 

For expository purposes we found it useful to split the paper into three separate parts. Part I is devoted to weighted Sobolev spaces and the corresponding Sobolev Embedding Theorems. The main example we are interested in is the case of ``conifolds''; in this special case the Sobolev Embedding Theorems, cf. Corollary \ref{cor:embedding}, are well-known. However, Problem (i) requires keeping close track of how the corresponding Sobolev constants depend on the conifolds
and on the other data used in the connect sum construction. It is thus useful to step back and investigate exactly which properties of Sobolev spaces are crucial to the validity of Embedding Theorems. In the standard, \textit{i.e.} non-weighted, case, the book by Hebey \cite{hebey} provides an excellent introduction to this problem. Given the lack of an analogous reference for weighted Sobolev spaces, we devote a fair amount of attention to their definition and
properties. Our main result in Part I is Theorem
\ref{thm:weighted_ok}, which proves the validity of the Sobolev Embedding Theorems under fairly general hypotheses on the ``scale'' and ``weight'' functions with which we define these spaces. 

Part II is devoted to the Fredholm theory of
elliptic operators on conifolds. This theory
is well-known but, for the reader's convenience, we review it (together with its
asymptotically cylindrical counterpart) in Sections \ref{s:acyl_analysis} and
\ref{s:accs_analysis}. Sections \ref{s:weightcrossing} and \ref{s:accs_harmonic}
contain instead some useful consequences of the Fredholm theory.

Part III contains the main results of this paper, corresponding to Steps A and B, above: the definition of ``conifold connect sums'' and the uniform estimates, Problems (i)-(iii).

We conclude with one last comment. Depending on the details, the connect sum construction can have two outcomes: compact or non-compact manifolds. In the context of weighted spaces, Problem (i) does not notice the difference. Problems (ii) and (iii) require instead that
the kernels of the operators in question vanish. On non-compact manifolds this
can be achieved very simply, via an a-priori choice of weights: roughly
speaking, we require that there exist non-compact ``ends'', then put weights on them which kill the kernel. This topological assumption is perfectly compatible
with the geometric applications described in \cite{pacini:slgluing}. On compact manifolds it is instead necessary to work transversally to the kernel; uniform estimates depend on allowing the subspace itself to depend on the parameter $t$. We refer to Section \ref{s:sums_laplace} for details.


\ 

{\bf Acknowledgments.} I would like to thank D. Joyce for many useful
suggestions and discussions concerning the material of this paper. I also thank
M. Haskins and J. Lotay for several conversations. Part of this work was carried out while I
was a Marie Curie EIF Fellow at the University of Oxford. It has also been supported by a Marie Curie ERG grant at the Scuola Normale Superiore in Pisa.


\section{Preliminaries}\label{s:prelim}
Let $(L,g)$ be an oriented $m$-dimensional Riemannian manifold. We can identify
its tangent and cotangent bundles via the maps
\begin{equation}\label{eq:sharp}
T_xL\rightarrow T_x^*L,\ \ v\mapsto v^{\#}:=g(v,\cdot), \ \ \mbox{with inverse }
T_x^*L\rightarrow T_xL,\ \ \alpha\mapsto\alpha^\flat.
\end{equation}
There are induced isomorphisms on all higher-order tensor bundles over $L$. In
particular the metric tensor $g$, as a section of $(T^*L)^2$, corresponds to a
tensor $g^\flat$, section of $(TL)^2$. This tensor defines a natural metric on
$T^*L$ with respect to which the map of Equation \ref{eq:sharp} is an isometry.
In local coordinates, if $g=g_{ij}dx^i\otimes dx^j$ then
$g^\flat=g^{ij}\partial_i\otimes \partial_j$, where $(g^{ij})$ denotes the
inverse matrix of $(g_{ij})$. 

Given any $x\in L$ we denote by $i_x(g)$ the \textit{injectivity radius at $x$},
\textit{i.e.} the radius of the largest ball in $T_xL$ on which the exponential
map is a diffeomorphism. We then define the \textit{injectivity radius of $L$}
to be the number $i(g):=\mbox{inf}_{x\in L} i_x(g)$. We denote by $Ric(g)$ the
\textit{Ricci curvature tensor} of $L$: for each $x\in L$, this gives an element
$Ric_x(g)\in T_x^*L\otimes T_x^*L$.

Let $E$ be a vector bundle over $L$. We denote by $C^\infty(E)$ (respectively,
$C^\infty_c(E)$) the corresponding space of smooth sections (respectively, with
compact support). If $E$ is a metric bundle we can define the notion of a
\textit{metric connection} on $E$: namely, a connection $\nabla$ satisfying
\begin{equation*}
\nabla(\sigma,\tau)=(\nabla\sigma,\tau)+(\sigma,\nabla\tau),
\end{equation*}
where $(\cdot,\cdot)$ is the appropriate metric. We then say that $(E,\nabla)$
is a \textit{metric pair}.

Recall that coupling the Levi-Civita connection on $TL$ with a given connection
on $E$ produces induced connections on all tensor products of these bundles and
of their duals. The induced connections depend linearly on the initial
connections. Our notation will usually not distinguish between the initial
connections and the induced connections: this is apparent when we write, for
example, $\nabla^2\sigma$ (short for $\nabla\nabla\sigma$). Recall also that the
difference between two connections $\nabla$, $\hat{\nabla}$ defines a tensor
$A:=\nabla-\hat{\nabla}$. For example, if the connections are on $E$ then $A$ is
a
tensor in $T^*L\otimes E^*\otimes E$. Once again, we will not distinguish
between this $A$ and the $A$ defined by any induced connections.

Let $E$, $F$ be vector bundles over $L$. Let $P:C^\infty(E)\rightarrow
C^\infty(F)$ be a linear differential operator with smooth coefficients, of
order $n$. We can then write $P=\sum_{i=0}^n A_i\cdot\nabla^i$, where $A_i$ is a
global section of $(TL)^i\otimes E^*\otimes F$ and $\cdot$ denotes an
appropriate contraction. Notice that since $P$ is a local operator it is
completely defined by its behaviour on compactly-supported sections.

\begin{remark} \label{rem:A}
Assume $P=\sum_{i=0}^n A_i\cdot\nabla^i$. Choose a second connection
$\hat{\nabla}$ on $E$ and set $A:=\nabla-\hat{\nabla}$. Substituting
$\nabla=\nabla-\hat{\nabla}+\hat{\nabla}=A+\hat{\nabla}$ allows us to write $P$
in terms of $\hat{\nabla}$. Notice that the new coefficient tensors $\hat{A}_i$
will depend on $A$ and on its derivatives $\hat{\nabla}^k A$.
\end{remark}

Now assume $E$ and $F$ are metric bundles. Then $P$ admits a \textit{formal
adjoint} $P^*:C^\infty(F)\rightarrow C^\infty(E)$, uniquely defined by imposing
\begin{equation}
\int_L(P\sigma,\tau)_F \,\mbox{vol}_g=\int_L(\sigma,P^*\tau)_E \,\mbox{vol}_g, \
\ \forall \sigma\in C^\infty_c(E), \ \tau\in C^\infty_c(F).
\end{equation} 
$P^*$ is also a linear differential operator, of the same order as $P$. 

\begin{example} \label{e:nablalaplace}
The operator $\nabla:C^\infty(E)\rightarrow C^\infty(T^*L\otimes E)$ has a
formal adjoint $\nabla^*: C^\infty(T^*L\otimes E)\rightarrow C^\infty(E)$. Given
$P=\sum_{i=0}^n A_i\cdot\nabla^i$, we can write $P^*$ in terms of $\nabla^*$.
For example, choose a smooth vector field $X$ on $L$ and consider the operator
$P:=\nabla_X=X\cdot\nabla:C^\infty(E)\rightarrow C^\infty(E)$. Then
$(\nabla_X)^*\sigma=\nabla^*(X^{\#}\otimes\sigma)$. 
\end{example}

The \textit{$\nabla$-Laplace} operator on $E$ is defined as
$\Delta:=\nabla^*\nabla:C^\infty(E)\rightarrow C^\infty(E)$. When $E$ is the
trivial $\R$-bundle over $L$ and we use the Levi-Civita connection, this
coincides with the standard positive Laplace operator acting on functions
\begin{equation}\label{eq:laplace}
\Delta_g:=-\mbox{tr}_g(\nabla^2)=-g^\flat\cdot\nabla^2:C^\infty(L)\rightarrow
C^\infty(L).
\end{equation}
Furthermore $\nabla=d$ and $\nabla^*=d^*$ so this Laplacian also coincides with
the Hodge Laplacian $d^*d$.
On differential $k$-forms the Levi-Civita $\nabla$-Laplacian and the Hodge
Laplacian coincide only up to curvature terms.

\ 

To conclude, let us recall a few elements of Functional Analysis. We now let $E$
denote a Banach space. Then $E^*$ denotes its dual space and
$\langle\cdot,\cdot\rangle$ denotes the duality map 
$E^*\times E\rightarrow\R$.

Let $P:E\rightarrow F$ be a continuous linear map between Banach spaces. Recall
that the \textit{norm} of $P$ is defined as
$\|P\|:=\mbox{sup}_{|e|=1}|P(e)|=\mbox{sup}_{e\neq 0} (|P(e)|/|e|)$. This
implies that, $\forall e\neq 0$, $|P(e)|\leq\|P\|\cdot |e|$. If $P$ is injective
and surjective then it follows from the Open Mapping Theorem that its inverse
$P^{-1}$ is also continuous. In this case $\mbox{inf}_{|e|=1}|P(e)|>0$ and we
can calculate the norm of $P^{-1}$ as follows:
\begin{equation}\label{eq:normofinverse}
\|P^{-1}\|=\mbox{sup}_{f\neq 0}\frac{|P^{-1}(f)|}{|f|}=\mbox{sup}_{e\neq
0}\frac{|e|}{|P(e)|}=\mbox{sup}_{|e|=1}\frac{1}{|P(e)|}=\frac{1}{\mbox{inf}_{
|e|=1}|P(e)|}.
\end{equation}
Recall that, given any subspace $Z\leq F$, the \textit{annihilator} of $Z$ is
defined as
$$\mbox{Ann}(Z):=\{\phi\in F^*:\langle\phi,z\rangle=0, \ \forall z\in Z\}.$$ 
Notice that $\mbox{Ann}(\overline{Z})=\mbox{Ann}(Z)$. Let $P^*:F^*\rightarrow
E^*$ be the dual map, defined by $\langle
P^*(\phi),e\rangle:=\langle\phi,P(e)\rangle$. It is simple to check that
$\mbox{Ann(Im$(P)$)}=\mbox{Ker}(P^*)$. 

Recall that the \textit{cokernel} of $P$ is defined to be the quotient space
$\mbox{Coker}(P):=F/\mbox{Im}(P)$. Assume the image $\mbox{Im}(P)$ of $P$ is a
closed subspace of $F$, so that $\mbox{Coker}(P)$ has an induced Banach space
structure. The projection $\pi:F\rightarrow \mbox{Coker}(P)$ is surjective so
its dual map $\pi^*:(\mbox{Coker}(P))^*\rightarrow F^*$ is injective. The image
of $\pi^*$ coincides with the space $\mbox{Ann(Im$(P)$)}$ so $\pi^*$ defines an
isomorphism between $(\mbox{Coker}(P))^*$ and $\mbox{Ann(Im$(P)$)}$. We conclude
that there exists a natural isomorphism
$(\mbox{Coker}(P))^*\simeq\mbox{Ker}(P^*)$.

\begin{remark} \label{rem:characterizations}
It is clear that $\mbox{Ker}(P^*)$ can be characterized as follows:
$$\phi\in\mbox{Ker}(P^*)\Leftrightarrow\langle\phi,P(e)\rangle=0,\ \ \forall
e\in E.$$
On the other hand, the Hahn-Banach Theorem shows that $f\in \overline{Z}$ iff
$\langle\phi,f\rangle=0$, $\forall \phi \in \mbox{Ann}(Z)$. Applying this to
$Z:=\mbox{Im}(P)$, we find the following characterization of
$\overline{\mbox{Im}(P)}$: 
$$f\in\overline{\mbox{Im}(P)}\Leftrightarrow \langle\phi,f\rangle=0,\ \ \forall
\phi\in\mbox{Ker}(P^*).$$
\end{remark}

We say that $P$ is \textit{Fredholm} if its image $\mbox{Im}(P)$ is closed in
$F$ and both $\mbox{Ker}(P)$ and $\mbox{Coker}(P)$ are finite-dimensional. We
then define the \textit{index} of $P$ to be 
$$i(P):=\mbox{dim(Ker$(P)$)}-\mbox{dim(Coker$(P)$)}=\mbox{dim(Ker$(P)$)}-\mbox{
dim(Ker$(P^*)$)}.$$

\ 

\textbf{Important remarks}: Throughout this paper we will often encounter chains
of inequalities of the form
\begin{equation*}
|e_0|\leq C_1 |e_1|\leq C_2 |e_2|\leq\dots 
\end{equation*}
The constants $C_i$ will often depend on factors that are irrelevant within the
given context. In this case we will sometimes simplify such expressions by
omitting the subscripts of the constants $C_i$, \textit{i.e.} by using a single
constant $C$. 

We assume all manifolds are oriented. In Part 2 of the paper we will work under
the assumption $m\geq 3$.


\part{Sobolev Embedding Theorems}

The goal of this part is to provide a self-contained overview of certain aspects of the theory of weighted Sobolev spaces on Riemannian manifolds. Aside from the special case of ``conifolds'', discussed in Section \ref{s:manifolds} and which is well-known, the point of view we present here applies to manifolds in general and we would not know where to find it in the literature. In Sections \ref{s:scaled} and \ref{s:weighted} we find it useful to separate the ``scaling factor'' $\rho$ from the ``weight'' $w$: distinguishing them in this way appears not to be a standard choice in the literature, but we find it useful so to emphasize their different roles in the theory.

\section{Review of the theory of standard Sobolev spaces}\label{s:std}

We now introduce and discuss Sobolev spaces on manifolds. A good reference,
which at times we follow closely, is Hebey \cite{hebey}.

Let $(E,\nabla)$ be a metric pair over $(L,g)$. The \textit{standard Sobolev
spaces}
are defined by
\begin{equation}\label{eq:std_sob}
W^p_k(E):=\mbox{Banach space completion of the space }\{\sigma\in
C^\infty(E):\|\sigma\|_{W^p_k}<\infty\},
\end{equation}
where $p\in [1,\infty)$, $k\geq 0$ and we use the norm
$\|\sigma\|_{W^p_k}:=\left(\Sigma_{j=0}^k\int_L|\nabla^j\sigma|^p\,\mbox{vol}
_g\right)^ { 1/p}$. We will sometimes use $L^p$ to denote the space $W^p_0$.

\begin{remark} 
At times we will want to emphasize the metric $g$ rather than the specific
Sobolev spaces. In these cases we will use the notation $\|\cdot\|_g$.
\end{remark}

It is important to find conditions ensuring that two metrics $g$, $\hat{g}$ on
$L$ (corresponding to Levi-Civita connections $\nabla$, $\hat{\nabla}$), define
\textit{equivalent} Sobolev norms, \textit{i.e.} such that there exists $C>0$
with $(1/C)\|\cdot\|_g\leq \|\cdot\|_{\hat{g}}\leq C\|\cdot\|_g$. In
this case the corresponding two completions, \textit{i.e.} the two spaces
$W^p_k$, coincide.

\begin{definition}\label{def:equivalentmetrics}
We say that two Riemannian metrics $g$, $\hat{g}$ on a manifold $L$ are
\textit{equivalent} if they satisfy the following assumptions:
\begin{description}
\item[A1] There exists $C_0>0$ such that 
\begin{equation*}
(1/C_0)g\leq \hat{g}\leq C_0 g.
\end{equation*}
\item[A2] For all $j\geq 1$ there exists $C_j>0$
such that
\begin{equation*}
|\nabla^j\hat{g}|_g\leq C_j.
\end{equation*}
\end{description}
\end{definition}

\begin{remark}\label{rem:equivalence}
It may be useful to emphasize that the conditions of Definition
\ref{def:equivalentmetrics} are symmetric in $g$ and
$\hat{g}$. Assumption A1 is obviously
symmetric. Assumption A2 is also symmetric. For
$j=1$, for example, this follows from the following calculation which uses the
fact that the connections are metric:
\begin{equation}\label{eq:Abounded}
|\nabla\hat{g}|_g=|\nabla\hat{g}-\hat{\nabla}\hat{g}|_g=|A(\hat{g})|_g\simeq
|A(g)|_{\hat{g}}=|\hat{\nabla} g|_{\hat{g}},
\end{equation} 
where $\simeq$ replaces multiplicative constants.
Notice that in Equation \ref{eq:Abounded} $A$ is the difference of the induced
connections on $T^*L\otimes T^*L$. This tensor depends linearly on the tensor
defined as the difference of the connections on $TL$. It is simple to see that
these two tensors have equivalent norms so that Assumption A2 provides a
pointwise bound on the norms of either one. From here we easily obtain bounds on
the norms of the tensor defined as the difference of the induced connections on
any tensor product of $TL$ and $T^*L$. Similar statements hold for
bounds on the derivatives of $A$.

Assumptions 1 and 2 can be unified as follows. Assume that, for
all $j\geq 0$, there exists $C_j>0$ such that
$$|\nabla^j(\hat{g}-g)|_g\leq C_j.$$
As long as $C_0$ is sufficiently small, for $j=0$ this condition implies
Assumption 1. Since $\nabla^j g=0$, it is clear that for $j>0$ it is equivalent
to Assumption 2.
\end{remark}

\begin{lemma}\label{l:equivalentstdnorms}
Assume $g$, $\hat{g}$ are equivalent.
Then the Sobolev norms defined by $g$ and $\hat{g}$ are equivalent.
\end{lemma}
\begin{proof}
Consider the Sobolev spaces of functions on $L$. Recall that $\nabla u=du$.
This implies that the $W^p_1$ norms depend only pointwise on the metrics. In
this case Assumption A1 is sufficient to ensure equivalence. In general,
however, the $W^p_k$ norms use the induced connections on tensor bundles. For
example, assume $j=2$. Then
\begin{equation*}
|\nabla^2u|=|(A+\hat{\nabla})(A+\hat{\nabla})u|\leq |A^2u|+|A\cdot\hat{\nabla}
u|+|\hat{\nabla} (Au)|+|\hat{\nabla}^2u|,
\end{equation*}
where $A:=\nabla-\hat{\nabla}$ is the difference of the appropriate connections.
It is clearly sufficient to obtain pointwise bounds on $A$ and its derivative
$\hat{\nabla} A$. As mentioned in Remark \ref{rem:equivalence}, these
follow from Assumption A2. The same is true
for Sobolev spaces of sections of tensor bundles over $L$. 

Now consider the Sobolev spaces of sections of $E$. Since we are not changing
the connection on $E$, Assumption A1 ensures equivalence of the $W^p_1$ norms.
The equivalence of the $W^p_k$ norms is proved as above.
\end{proof}

For $p>1$ we define $p'$ via 
\begin{equation}\label{eq:p'}
\frac{1}{p}+\frac{1}{p'}=1, \ \ \mbox{\textit{i.e. }} p'=\frac{p}{p-1}.
\end{equation}
For $p\geq 1$ we define $p^*$ via
\begin{equation}\label{eq:p*}
\frac{1}{p^*}=\frac{1}{p}-\frac{1}{m},\ \ \mbox{\textit{i.e. }}
p^*=\frac{mp}{m-p}.
\end{equation}
It is simple to check that
\begin{equation}\label{eq:sobconstantsfundrelation}
\frac{1}{p^*}+\frac{1}{p'}=\frac{m-1}{m}.
\end{equation}
More generally, for $p\geq 1$ and $l=\{1,2,\dots\}$ we define $p^*_l$ via
\begin{equation}\label{eq:p*l}
\frac{1}{p^*_l}=\frac{1}{p}-\frac{l}{m},\ \ \mbox{\textit{i.e. }}
p^*_l=\frac{mp}{m-lp},
\end{equation}
so that $p^*=p^*_1$. Notice that $p^*_l$ is obtained by $l$ iterations of the
operation
\begin{equation*}
 p\mapsto p^*
\end{equation*}
and that $\frac{1}{p^*_l}<\frac{1}{p^*_{l-1}}<\frac{1}{p}$, so if $p^*_l>0$
(equivalently, $lp<m$) then $p^*_l>p^*_{l-1}>p$. In other words, under
appropriate conditions $p^*_l$ increases with $l$.

The Sobolev Embedding Theorems come in two basic forms, depending on the product
$lp$.
The \textit{Sobolev Embedding Theorems, Part I} concern the existence of
continuous embeddings of the form 
\begin{equation}\label{eq:std_partI}
 W^p_{k+l}(E)\hookrightarrow W^{p^*_l}_k(E)\ \ (\mbox{for }lp<m),
\end{equation}
\textit{i.e.} the existence of some constant $C>0$ such that, $\forall \sigma\in
W^p_{k+l}(E)$,
\begin{equation}\label{eq:std_partIbis}
 \|\sigma\|_{W^{p^*_l}_{k}(E)}\leq C\|\sigma\|_{W^p_{k+l}(E)}.
\end{equation}
A standard argument based on H\"older's inequality then shows that 
$ W^p_{k+l}(E)\hookrightarrow W^q_k(E)$, for all $q\in [p, p^*_l]$.
We call $C$ the \textit{Sobolev constant}. In words, bounds on the higher derivatives of $\sigma$ enhance the integrability
of $\sigma$. Otherwise said, one can sacrifice derivatives to improve
integrability; the more derivatives one sacrifices, the larger the integrability range $[p,p^*_l]$. 

The \textit{exceptional case} of Part I concerns the existence of continuous
embeddings of the form
\begin{equation}\label{eq:except_partI}
 W^p_{k+l}(E)\hookrightarrow W^{q}_k(E)\ \ (\mbox{for }lp=m),\ \ \forall q\in
[p,\infty).
\end{equation}
The \textit{Sobolev Embedding Theorems, Part II} concern the existence of
continuous embeddings of the form 
\begin{equation}\label{eq:std_partII}
 W^p_{k+l}(E)\hookrightarrow C^k(E)\ \ (\mbox{for }lp>m).
\end{equation}
Roughly speaking, this means that one can sacrifice derivatives to improve
regularity.

\ 

The validity of these theorems for a given manifold $(L,g)$ depends on its
Riemannian properties. It is a useful fact that the properties of $(E,\nabla)$
play no extra role: more precisely, if an Embedding Theorem holds for functions
on $L$ it then holds for sections of any metric pair $(E,\nabla)$. This is a
consequence of the following result.

\begin{lemma}[Kato's inequality]\label{l:kato} Let $(E,\nabla)$ be a metric
pair. Let $\sigma$ be a smooth section of $E$. Then, away from the zero set of
$\sigma$,
\begin{equation}\label{eq:kato}
|d|\sigma||\leq |\nabla\sigma|.
\end{equation}
\end{lemma}
\begin{proof}
\begin{equation*}
2|\sigma||d|\sigma||=|d|\sigma|^2|=2(\nabla\sigma,\sigma)\leq
2|\nabla\sigma||\sigma|.
\end{equation*}
\end{proof}
The next result shows that if Part I holds in the simplest cases it then holds
in all cases. Likewise, the general case of Part II follows from combining the
simplest cases of Part II with the general case of Part I.

\begin{prop}\label{prop:complexfromsimple}
\ 

\begin{enumerate} 
\item Assume Part I, Equation \ref{eq:std_partI}, holds for all $p<m$ with $l=1$
and $k=0$. Then Part I holds for all $p$ and $l$ satisfying $lp<m$ and for all
$k\geq 0$.
\item Assume Part I, Equation \ref{eq:std_partI}, holds in all cases and that
the exceptional case, Equation \ref{eq:except_partI}, holds for $l=1$ and $k=0$.
Then the exceptional case holds for all $p$ and $l$ satisfying $lp=m$ and for
all $k\geq 0$.
\item Assume Part I, Equation \ref{eq:std_partI}, and the exceptional case,
Equation \ref{eq:except_partI}, hold in all cases and that Part II, Equation
\ref{eq:std_partII}, holds for all $p>m$ with $l=1$ and $k=0$. Then Part II
holds for all $p$ and $l$ satisfying $lp>m$ and for all $k\geq 0$.
\end{enumerate}
\end{prop}
\begin{proof} 
As discussed above, it is sufficient to prove that the result holds for
functions: as a result of Kato's inequality it will then hold for arbitrary
metric pairs $(E,\nabla)$.

(1) Assume $l=1$. Given $u\in W^p_{k+1}$, Kato's inequality shows that
$|u|,\dots\, |\nabla^k u|\in W^p_1$. Applying Part I to each of these then shows
that $W^p_{k+1}\hookrightarrow W^{p^*}_k$. The general case follows from the
composition of the embeddings 
\begin{equation*}
W^p_{k+l}\hookrightarrow W^{p^*}_{k+l-1}\hookrightarrow
W^{p^*_2}_{k+l-2}\hookrightarrow\dots
\end{equation*}

(2) For $l=1$ we can prove $W^p_{k+1}\hookrightarrow W^q_k$ as in (1) above. Now
assume $lp=m$ for $l\geq 2$. Then Part I yields $W^p_l\hookrightarrow
W^{p^*_{l-1}}_1$. Since $p^*_{l-1}=m$ we can now apply the exceptional case in
its simplest form.

(3) Let us consider, for example, the case $l=2$ and $k=0$. We are then assuming
that $p>m/2$. Let us distinguish three subcases, as follows. Assume $p\in
(m/2,m)$. Then Part 1 implies that $W^p_2\hookrightarrow W^{p^*}_1$. Since
$p^*>m$ we can now use the embedding $W^{p^*}_1\hookrightarrow C^0$ to conclude.
Now assume $p=m$. Then $W^p_2\hookrightarrow W^q_1$ for any $q>m$ and we can
conclude as above. Finally, assume $p>m$. Then $W^p_2\hookrightarrow
W^p_1\hookrightarrow C^0$. The other cases are similar.
\end{proof}

\begin{corollary}
 Assume the Sobolev Embedding Theorems hold for $(L,g)$. Let $\hat{g}$ be a
second Riemannian metric on $L$ such that, for some $C_0>0$, $(1/C_0)g\leq
\hat{g}\leq C_0 g$. Then the Sobolev Embedding Theorems hold also for
$(L,\hat{g})$.  
\end{corollary}
\begin{proof}
 According to Proposition \ref{prop:complexfromsimple} it is sufficient to
verify the Sobolev Embedding Theorems in the case $l=1$ and $k=0$. These
involve only $C^0$-information on the metric. The conclusion is thus
straight-forward.
\end{proof}

\begin{remark} \label{rem:complexfromsimple}
Under a certain density condition, Proposition \ref{prop:complexfromsimple} can
be enhanced as follows.

Assume Part I, Equation \ref{eq:std_partI}, holds for $p=1$, $l=1$ and $k=0$,
\textit{i.e.} $W^1_1\hookrightarrow L^\frac{m}{m-1}$. Assume also that, for all
$p<m$, the space $C^\infty_c(L)$ is dense in $W^p_1$. Then Part I holds for all
$p<m$ with $l=1$ and $k=0$, \textit{i.e.} $W^p_1\hookrightarrow L^{p^*}$. The
proof is as follows. 

Choose $u\in C^\infty_c(L)$. One can check that, for all $s>1$, $|u|^s\in
W^p_1$, cf. \textit{e.g.} \cite{hebey}. Then, using Part I and H\"{o}lder's
inequality,
\begin{align*}
\||u|^s\|_{L^{\frac{m}{m-1}}} &\leq C \int_L (|u|^s+|\nabla
|u|^s|)\,\mbox{vol}_g\\
&\leq C \int_L (|u|^{s-1}|u|+|u|^{s-1}|\nabla u|)\,\mbox{vol}_g\\
&\leq C \,\| |u|^{s-1}\|_{L^{p'}} \left(\|u\|_{L^p}+\|\nabla u\|_{L^p}\right).
\end{align*}
Let us now choose $s$ so that $(s-1)p'=sm/(m-1)$, \textit{i.e.} $s=p^*(m-1)/m$.
Substituting, we find
\begin{equation*}
\left(\int_L |u|^{p^*}\right)^{\frac{m-1}{m}}\leq C \left(\int_L
|u|^{p^*}\right)^{\frac{1}{p'}}\|u\|_{W^p_1}.
\end{equation*}
This leads to $\|u\|_{L^{p^*}}\leq C \|u\|_{W^p_1}$, for all $u\in
C^\infty_c(L)$. By density, the same is true for all $u\in W^p_1$.

To conclude, we mention that if $(L,g)$ is complete then $C^\infty_c(L)$ is
known to be dense in $W^p_1$ for all $p\geq 1$, cf. \cite{hebey} Theorem 3.1.
\end{remark}

The most basic setting in which all parts of the Sobolev Embedding Theorems hold
is when $L$ is a smooth bounded domain in $\R^m$ endowed with the standard
metric $\tilde{g}$. Another important class of examples is the following.

\begin{theorem}\label{thm:std_ok}
 Assume $(L,g)$ satisfies the following assumptions: there exists $R_1>0$ and
$R_2\in\R$ such that
\begin{equation*}
i(g)\geq R_1,\ \ Ric(g)\geq R_2\,g.
\end{equation*}
Then:
\begin{enumerate}
\item The Sobolev embeddings Part I, Equation \ref{eq:std_partI}, hold for all
$p$ and $l$ satisfying $lp<m$ and for all $k\geq 0$.
\item The exceptional case of Part I, Equation \ref{eq:except_partI}, holds for
all $p$ and $l$ satisfying $lp=m$ and for all $k\geq 0$.
\item The Sobolev embeddings Part II, Equation \ref{eq:std_partII}, hold for all
$p$ and $l$ satisfying $lp>m$ and for all $k\geq 0$.
\end{enumerate}
Furthermore, when $kp>m$, $W^p_k$ is a Banach algebra.
Specifically, there exists $C>0$ such that, for all $u,v\in W^p_k$, the
product $uv$ belongs to $W^p_k$ and satisfies
\begin{equation*}
 \|uv\|_{W^p_k}\leq C \|u\|_{W^p_k}\cdot\|v\|_{W^p_k}.
\end{equation*}
\end{theorem}

We will prove Theorem \ref{thm:std_ok} below. Roughly speaking, the reason it
holds is the following. Given any coordinate system on $L$, the embeddings hold
on every chart endowed with the flat metric $\tilde{g}$. Now recall that, given
any $(L,g)$ and any $x\in L$, it is always possible to find coordinates
$\phi_x:B\subset\R^m\rightarrow L$ in which the metric $g$ is a small
perturbation of the flat metric: this implies that the embeddings hold locally
also with respect to $g$. The problem is that, in general, the size of the ball
$B$, thus the corresponding Sobolev constants, will depend on $x$. Our
assumptions on $L$, however, can be used to build a special coordinate system
whose charts admit uniform bounds. One can then show that this implies that the
embeddings hold globally. The main technical step in the proof of Theorem
\ref{thm:std_ok} is thus the following result concerning the
existence and properties of \textit{harmonic coordinate systems}.

\begin{theorem}\label{thm:harmonic_coords}
Assume $(L,g)$ satisfies the assumptions of Theorem \ref{thm:std_ok}.
Then for all small $\epsilon>0$ there exists $r>0$ such that, for each $x\in L$,
there exist coordinates $\phi_x:B_r\subset\R^m\rightarrow L$ satisfying
\begin{enumerate}
\item $\phi_x^{-1}$ (seen as a map into $\R^m$) is harmonic.
\item $\|\phi_x^*g-\tilde{g}\|_{C^0}\leq\epsilon$.
\end{enumerate}
\end{theorem}

\begin{remark} \label{rem:harmonic_coords}
Theorem \ref{thm:harmonic_coords} can be heavily improved, cf. \cite{hebey}
Theorem 1.2. Firstly, it is actually a local result, \textit{i.e.} one can get
similar results for any open subset of $L$ by imposing similar assumptions on a
slightly larger subset. Secondly, these same assumptions actually yield certain
$C^{0,\alpha}$ bounds. Thirdly, assumptions on the higher derivatives of the
Ricci tensor yield certain bounds on the higher derivatives of
$\phi_x^*g-\tilde{g}$, see Remark \ref{rem:harmonic_coordsbis} for details.

To conclude, it may be useful to emphasize that imposing a global lower bound on
the injectivity radius of $(L,g)$ implies completeness.
\end{remark}

\begin{proof}[Proof of Theorem \ref{thm:std_ok}]
As seen in Proposition \ref{prop:complexfromsimple}, it is sufficient to prove
the Sobolev Embedding Theorems in the simplest cases. Concerning Part I, let us
choose $u\in W^p_1(L)$. Using the coordinates of Theorem
\ref{thm:harmonic_coords}, $\phi_x^*u\in W^p_1(B_r)$. All Sobolev Embedding
Theorems hold on $B_r$ with its standard metric $\tilde{g}$. Thus there exists a
constant $C$ such that, with respect to $\tilde{g}$,
\begin{equation}\label{eq:sob_local}
\|\phi_x^*u\|_{L^{p^*}(B_r)}\leq C \|\phi_x^*u\|_{W^p_1(B_r)}.
\end{equation}
The fact that $\nabla u=du$ implies that Equation \ref{eq:sob_local} involves
only $C^0$ information on the metric. Since $\phi_x^*g$ is $C^0$-close to
$\tilde{g}$, up to a small change of the constant $C$ the same inequality holds
with respect to $\phi_x^*g$. Let $B_x(r)$ denote the ball in $(L,g)$ with center
$x$ and radius $r$. Then $B_x(r/2)\subset\phi_x(B_r)\subset B_x(2r)$ so
\begin{align*}
\int_{B_x(r/2)}|u|^{p^*} \mbox{vol}_g&\leq
\int_{\phi_x(B_r)}|u|^{p^*}\mbox{vol}_g\\
&\leq
C\left(\int_{\phi_x(B_r)}(|u|^p+|du|^p)\,\mbox{vol}_g\right)^{\frac{p^*-p+p}{p}}
\\
&\leq C\left(\int_L
(|u|^p+|du|^p)\,\mbox{vol}_g\right)^{\frac{p^*-p}{p}}\left(\int_{B_x(2r)}
(|u|^p+|du|^p)\,\mbox{vol}_g\right).
\end{align*}
Let us now integrate both sides of the above equation with respect to $x\in L$.
We can then change the order of integration according to the formula
\begin{equation*}
\int_{x\in L}\left(\int_{y\in B_x(r)}
f(y)\,\mbox{vol}_g\right)\,\mbox{vol}_g=\int_{y\in L}f(y)\left(\int_{x\in
B_y(r)} \mbox{vol}_g\right)\,\mbox{vol}_g.
\end{equation*}
Reducing $r$ if necessary, the $C^0$ estimate on $g$ yields uniform
bounds (with respect to $x$) on $\mbox{vol}_g(B_x(r/2))$ and
$\mbox{vol}_g(B_x(2r))$ because analogous bounds hold for $\tilde{g}$. This
allows us to substitute the inner integrals with appropriate constants. We
conclude that
\begin{align*}
\int_L |u|^{p^*}\,\mbox{vol}_g &\leq C\left(\int_L
(|u|^p+|du|^p)\,\mbox{vol}_g\right)^{\frac{p^*-p}{p}}\left(\int_L
(|u|^p+|du|^p)\,\mbox{vol}_g\right)\\
&= C\left(\int_L (|u|^p+|du|^p)\,\mbox{vol}_g\right)^{\frac{p^*}{p}}.
\end{align*}
We conclude by raising both sides of the above equation to the power $1/{p^*}$.
Notice that the final constant $C$ can be estimated in terms of the volume of
balls in $L$ and of the constant $C$ appearing in Equation \ref{eq:sob_local}.

The exceptional case of Part I is similar: it is sufficient to replace $p^*$
with any $q>m$. Part II is also similar, though slightly simpler. Specifically,
one finds as above that
\begin{equation*}
\|u\|_{C^0(\phi_x(B_r))} \leq C \|u\|_{W^p_1(\phi_x(B_r))}\leq C
\|u\|_{W^p_1(L)}
\end{equation*}
Since this holds for all $x\in L$, we conclude that $\|u\|_{C^0(L)}\leq C
\|u\|_{W^p_1(L)}$.

The proof that $W^p_k$ is a Banach algebra relies on the Sobolev Embedding Theorems and some simple algebraic manipulations. For brevity we present only the case $W^p_2$ with $2p>m$, which already contains all the main ideas; \cite{adams}, Theorem 5.23, gives the general proof for domains in $\R^m$. 

Recall the Leibniz rule
\begin{equation*}
\nabla^j(uv)=\sum_{k=0}^j\binom{j}{k}(\nabla^ku)\otimes(\nabla^{j-k}v).
\end{equation*}
It thus suffices to estimate each term on the right hand side, for $j=0,1,2$. The embedding $W^p_2\hookrightarrow C^0$ implies that 
\begin{equation*}
 \int_L|uv|^p\,\mbox{vol}_g\leq \|u\|^p_{C^0}\cdot\int_L|v|^p\,\mbox{vol}_g\leq C\|u\|^p_{W^p_2}\cdot\|v\|^p_{W^p_2}.
\end{equation*}
We can analogously estimate all other terms except perhaps $\int|\nabla u|^p|\nabla v|^p$. If $p>m$ we can use the stronger embedding $W^p_2\hookrightarrow C^1$ to estimate this term as above. Otherwise we use the following fact.

\textit{Fact}: Assume $m/2<p\leq m$. Then there exist $r,r'$ such that $1/r+1/r'=1$ and $pr<p^*$, $pr'<p^*$. 

This fact is obvious if $p=m$ (using the convention $p^*=\infty$). For $p<m$ it suffices to choose $r$ such that $m/p<r<m/(p-m)$ and $r'=r/(r-1)$. 

The Sobolev Embedding Theorem, Part I, then yields $W^p_1\hookrightarrow L^{pr}$ so $|\nabla u|^p\in L^r$. Likewise, $|\nabla v|^p\in L^{r'}$ so, using H\"older's inequality,
\begin{align*}
 \int_L|\nabla u|^p|\nabla v|^p\,\mbox{vol}_g&\leq \||\nabla u|^p\|_{L^r}\cdot\||\nabla v|^p\|_{L^{r'}}=\|\nabla u\|^p_{L^{pr}}\cdot\|\nabla v\|^p_{L^{pr'}}\\
&\leq C\|\nabla u\|^p_{W^p_1}\cdot\|\nabla v\|^p_{W^p_1}\leq C\|u\|^p_{W^p_2}\cdot\|v\|^p_{W^p_2}.
\end{align*}
Combining all these estimates proves that $\|uv\|_{W^p_2}\leq C\|u\|_{W^p_2}\cdot\|v\|_{W^p_2}$, as claimed.
\end{proof}

\begin{example}\label{e:std_ok}
Any compact oriented Riemannian manifold $(L,g)$ satisfies the assumptions of
Theorem \ref{thm:std_ok}. Thus the Sobolev Embedding Theorems hold in full
generality for such manifolds. The same is true for the non-compact manifold
$\R^m$, endowed with the standard metric $\tilde{g}$.

Let $(\Sigma,g')$ be a compact oriented Riemannian manifold. Consider
$L:=\Sigma\times\R$ endowed with the metric $\tilde{h}:=dz^2+g'$. It is clear
that $(L,\tilde{h})$ satisfies the assumptions of Theorem \ref{thm:std_ok} so
again the Sobolev Embedding Theorems hold in full generality for these
manifolds. More generally they hold for the asymptotically cylindrical manifolds
of Section \ref{s:manifolds}. Notice however that here we are using
the Sobolev spaces defined in Equation \ref{eq:std_sob}. In Section
\ref{s:manifolds} we will verify the Sobolev Embedding Theorems for a different
class of Sobolev spaces, cf. Definition \ref{def:acyl_sectionspaces}. 
\end{example}


\section{Scaled Sobolev spaces}\label{s:scaled}

In applications standard Sobolev spaces are often not satisfactory for various
reasons. Firstly, they do not have good properties with respect to rescalings of
the sort $(L,t^2g)$. Secondly, uniform geometric
bounds of the sort seen in Theorem \ref{thm:std_ok} are too strong. Thirdly, the
finiteness condition in Equation
\ref{eq:std_sob} is very rigid and restrictive.

For all the above reasons it is often useful to modify the Sobolev norms. A
simple way of addressing the first two problems is to introduce an extra
piece of data, as follows.

Let $(L,g,\rho)$ be an oriented Riemannian manifold endowed with a \textit{scale
factor} $\rho>0$ or a \textit{scale function} $\rho=\rho(x)>0$. Given any metric
pair $(E,\nabla)$, the \textit{scaled Sobolev spaces} are defined by
\begin{equation}\label{eq:inv_sob}
{W^p_{k;sc}(E)}:=\mbox{Banach space completion of the space }
\left\{\sigma\in
C^\infty(E):\|\sigma\|_{W^p_{k;sc}}<\infty\right\},
\end{equation}
where we use the norm $\|\sigma\|_{W^p_{k;sc}}:=\left(\Sigma_{j=0}^k \int_L
|\rho^j\nabla^j\sigma|_g^p\rho^{-m}\,\mbox{vol}_g\right)^{1/p}$. 

Notice that at the scale $\rho\equiv 1$ these norms coincide with the standard
norms. 

\begin{remark}\label{rem:scalednorms}
Let us slightly change notation, using $g_L$ (respectively, $g_E$) to denote the
metric on $L$ (respectively, on $E$). 
The metric $g$ used in the above norms to measure $\nabla^j\sigma$ is obtained
by tensoring $g_L$ (applied to $\nabla^j$) with $g_E$ (applied to $\sigma$): let
us write $g=g_L\otimes g_E$. We then find
\begin{equation*}
|\rho^j\nabla^j\sigma|_{g_L\otimes g_E}\rho^{-m}\,\mbox{vol}_{g_L\otimes
g_E}=|\nabla^j\sigma|_{(\rho^{-2}g_L)\otimes
g_E}\mbox{vol}_{(\rho^{-2}g_L)\otimes g_E}.
\end{equation*}
Roughly speaking, the scaled norms thus coincide with the standard norms
obtained via the conformally equivalent metric $\rho^{-2}g_L$ on $L$. It is
important to emphasize, however, that we are conformally rescaling only part of
the metric.
This can be confusing when $E$ is a tensor bundle over $L$, endowed with the
induced metric: it would then be natural to also rescale the metric of $E$.
We are also not changing the connections $\nabla$. In general these
connections are not metric connections with respect to $(\rho^{-2}g_L)\otimes
g_E$. This has important consequences regarding the Sobolev Embedding Theorems
for scaled Sobolev spaces, as follows.

Naively, one might hope that such theorems hold under the assumptions:
\begin{equation*}
i(\rho^{-2}g)\geq R_1,\ \ Ric(\rho^{-2}g)\geq R_2 \rho^{-2}g.
\end{equation*}
Indeed, these assumptions do suffice to prove the Sobolev Embedding Theorems in
the simplest case, \textit{i.e.} $l=1$ and $k=0$. However, the general case
requires Kato's inequality, Lemma \ref{l:kato}, which in turn requires metric
connections. To prove these theorems we will thus need further assumptions on
$\rho$, cf. Theorem \ref{thm:scaled_ok}.
\end{remark}

We now define \textit{rescaling} to be an action of $\R^+$ on the triple
$(L,g,\rho)$, via $t\cdot(L,g,\rho):=(L,t^2g, t\rho)$. Recall that the
Levi-Civita connection $\nabla$ on $L$ does not change under rescaling. Using
this fact it is simple to check that $\|\sigma\|_{W^p_{k;sc}}$, calculated with
respect to $t\cdot(L,g,\rho)$, coincides with $\|\sigma\|_{W^p_{k;sc}}$,
calculated with respect to $(L,g,\rho)$: in this sense the scaled norm is
invariant under rescaling.

\begin{remark}\label{rem:scalednormsbis}
As in Remark \ref{rem:scalednorms}, our definition of rescaling requires some
care. To explain this let us adopt the same notation as in Remark
\ref{rem:scalednorms}. Our notion of rescaling affects only the metric on $L$,
not the metric on $E$. As before, this can be confusing when $E$ is a tensor
bundle over $L$, endowed with the induced metric.
\end{remark}

As in Section \ref{s:std}, it is important to find conditions under which
$(L,g,\rho)$ and $(L,\hat{g},\rho)$ define equivalent norms.

\begin{definition}\label{def:equivalentscaledmetrics}
 Let $(L,\rho)$ be a manifold endowed with a scale function. We say that two
Riemannian metrics $g$, $\hat{g}$ are \textit{scaled-equivalent} if they
satisfy the following assumptions:
\begin{description}
\item[A1] There exists $C_0>0$ such that 
\begin{equation*}
(1/C_0)g\leq \hat{g}\leq C_0 g.
\end{equation*}
\item[A2] For all $j\geq 1$ there exists $C_j>0$
such that
\begin{equation*}
|\nabla^j\hat{g}|_{\rho^{-2}g\otimes g_E}\leq C_j,
\end{equation*}
where $\nabla$ is the Levi-Civita connection defined by $g$, $E=T^*L\otimes
T^*L$ and we are using the notation introduced in Remark \ref{rem:scalednorms}.
\end{description}
\end{definition}

\begin{remark}\label{rem:scaledequivalence}
 As in Remark \ref{rem:equivalence}, one can check that
\begin{equation*}
|\nabla\hat{g}|_{\rho^{-2}g\otimes g_E}\leq C_1\Rightarrow
|A(\hat{g})|_{\rho^{-2}g\otimes g_E}\leq C_1.
\end{equation*}
In turn this implies that 
$|A|_{\rho^{-2}g\otimes g_E}\leq C_1$, where now $A$ denotes the difference
$\nabla-\hat{\nabla}$ of the connections on $TL$ and $E=T^*L\otimes TL$. 

Again as in Remark \ref{rem:equivalence}, one can check that if for all
$j\geq 0$ there exists $C_j>0$ such that 
$$|\nabla^j(\hat{g}-g)|_{\rho^{-2}g\otimes g_E}\leq C_j$$
and if $C_0$ is sufficiently small then $g$, $\hat{g}$ satisfy Assumptions
A1, A2.
\end{remark}

The following
result is a simple consequence of Remark \ref{rem:scalednorms} and Lemma
\ref{l:equivalentstdnorms}. 

\begin{lemma}\label{l:equivalentscalednorms}
Assume $(L,g,\rho)$, $(L,\hat{g},\rho)$ are scaled-equivalent in the sense of
Definition \ref{def:equivalentscaledmetrics}. 
Then the scaled Sobolev norms are equivalent.
\end{lemma} 

We can also define the \textit{scaled spaces of $C^k$ sections}
\begin{equation}\label{eq:scaled_C^k}
C^k_{sc}(E):=\left\{\sigma\in C^k(E): \|\sigma\|_{C^k_{sc}}<\infty\right\},
\end{equation}
where we use the norm $\|\sigma\|_{C^k_{sc}}:=\sum_{j=0}^k \sup_{x\in
L}|\rho^j\nabla^j\sigma|_g$. Once again, these norms define Banach spaces. 

\begin{remark} \label{rem:harmonic_coordsbis}
One can analogously define $C^{k,\alpha}_{sc}$ spaces. Notice that Equation
\ref{eq:scaled_C^k} implies that $C^0_{sc}=C^0$. It is these spaces which are
relevant to the generalization to higher derivatives of Theorem
\ref{thm:harmonic_coords}. Specifically, bounds on the higher derivatives of
$Ric(g)$ yield $C^{k,\alpha}_{sc}$ bounds on $\phi_x^*g-\tilde{g}$ with respect
to the (constant) scale factor $r$ determined by the theorem.
\end{remark}

We are now ready to study the Sobolev Embedding Theorems for scaled spaces. As
mentioned in Remark \ref{rem:scalednorms}, these theorems require further
assumptions on $\rho$. 

\begin{theorem}\label{thm:scaled_ok}
Let $(L,g)$ be a Riemannian manifold and $\rho$ a positive function on
$L$. Assume there exist constants $R_1>0$, $R_2\in\R$, $R_3>1$ and $\zeta>0$
such that:
\begin{description}
\item[A1] $\forall x\in L,\ \ i_x(g)\geq R_1\rho(x)$.
\item[A2] $\forall x\in L,\ \ Ric_x(g)\geq R_2 \rho(x)^{-2}g_x$.
\item[A3] $\forall x\in L, \forall y\in B(x,\zeta\rho(x))$,
\begin{equation*}
(1/R_3)\rho(x)\leq\rho(y)\leq R_3\rho(x).
\end{equation*}
\end{description}
Then all parts of the Sobolev Embedding Theorems hold for scaled norms and for
any metric pair $(E,\nabla)$. Furthermore, when $kp>m$,
$W^p_{k;sc}$ is a Banach algebra.

Now let $\hat{g}$ be a second Riemannian metric on $L$ such that, for some
$C_0>0$, $(1/C_0)g\leq \hat{g}\leq C_0 g$. Then the scaled Sobolev Embedding
Theorems hold also for $(L,\hat{g},\rho)$ and for any metric pair $(E,\nabla)$.
The Sobolev constants of $\hat{g}$ depend only on the Sobolev constants of $g$
and on $C_0$.
\end{theorem}
\begin{proof} Let us prove Part 1 for functions, assuming $l=1$, $k=0$. Choose
$x\in L$. Set $B_x:=B(x,\zeta\rho(x))$. For $y\in B_x$, consider the rescaled
metric $h$ defined by $h_y:=\rho(x)^{-2}g_y$. Assumption A1 shows that
$i_y(g)\geq R_1\rho(y)$. Using Assumption A3 we find 
\begin{equation*}
i_y(h)=\rho(x)^{-1}i_y(g)\geq R_1\rho(y)\rho(x)^{-1}\geq R_1/R_3.
\end{equation*}
Now recall that the Ricci curvature $Ric$ is invariant under rescaling,
\textit{i.e.} $Ric(h)=Ric(g)$. Then Assumptions A2 and A3 show that 
\begin{equation*}
Ric_y(h)=Ric_y(g)\geq R_2\rho(y)^{-2}\rho(x)^2h\geq (R_2/R_3^2)h.
\end{equation*}
We have thus obtained lower bounds on the injectivity radius and Ricci curvature
of $(B_x,h)$. Notice that these bounds are independent of $x$. Recall from
Remark \ref{rem:harmonic_coords} that Theorem \ref{thm:harmonic_coords} is
essentially local. Specifically, set $B'_x:=B(x, (1/2)\zeta\rho(x))$. Then for
any $\epsilon>0$ there exists $r=r(p,R_1,R_2,R_3,\epsilon,m)$ such that, for any
$x\in L$, there exist coordinates $\phi_x:B_r\rightarrow (B'_x,h)$ satisfying
$\|\phi_x^*h-\tilde{g}\|_{C^0}\leq \epsilon$.

Exactly as in the proof of Theorem \ref{thm:std_ok}, we can now use the local
Sobolev Embedding Theorems for $B_r$ to conclude that 
\begin{equation}\label{eq:scaled_okzero}
\left(\int_{B'_x}|u|^{p^*}\mbox{vol}_h\right)^{1/p^*}\leq C
\left(\int_{B'_x}(|u|^p+|du|_h^p)\,\mbox{vol}_h\right)^{1/p}.
\end{equation}
Assumption A3 allows us, up to a change of constants, to replace the (locally)
constant quantity $\rho(x)$ with the function $\rho(y)$. Remark
\ref{rem:scalednorms} shows how replacing $\rho^{-2}g$ with $g$ leads to the
scaled norms. Proceeding as in the proof of Theorem \ref{thm:std_ok}, via double
integration, we then get
\begin{equation}\label{eq:scaled_ok}
\|u\|_{L^{p^*}_{sc}}\leq C\|u\|_{W^{p}_{1;sc}},
\end{equation}
where we are now using the metric $g$. 

Now consider the case $k=1$, \textit{i.e.} assume $u\in W^p_{2;sc}$. Then 
$\phi_x^*|\nabla u|_h\in W^p_1(B_r)$. As before, we obtain
\begin{equation}\label{eq:scaled_okbis}
\left(\int_{B'_x}|\nabla u|_h^{p^*}\mbox{vol}_h\right)^{1/p^*}\leq C
\left(\int_{B'_x}(|\nabla u|_h^p+|d(|\nabla
u|_h)|_h^p)\,\mbox{vol}_h\right)^{1/p}.
\end{equation}
Notice that the Levi-Civita connections of $g$ and $h$ coincide. We can thus
apply Kato's inequality, finding $|d|\nabla
u|_h|_h\leq|\nabla^2u|_h=|\rho(x)^2\nabla^2 u|_g$. This leads to
\begin{equation}\label{eq:scaled_okter}
\left(\int_{B'_x}|\rho(x)\nabla
u|_g^{p^*}\rho(x)^{-m}\mbox{vol}_g\right)^{1/p^*}\leq C
\left(\int_{B'_x}(|\rho(x)\nabla u|_g^p+|\rho(x)^2\nabla^2
u|_g^p)\rho(x)^{-m}\mbox{vol}_g\right)^{1/p}.
\end{equation}
We can now proceed as before, using Assumption A3, to obtain
\begin{equation*}
\|\nabla u\|_{L^{p^*}_{sc}}\leq C\|\nabla u\|_{W^{p}_{1;sc}}.
\end{equation*}
Together with Equation \ref{eq:scaled_ok}, this implies
$W^p_{2;sc}\hookrightarrow W^{p^*}_{1;sc}$.

The other cases and parts of the Sobolev Embedding Theorems can be proved
analogously. 

The claim that $W^p_{k;sc}$ is a Banach algebra can be proved as in
Theorem \ref{thm:std_ok}, using Remark
\ref{rem:scalednorms} to write the scaled norms in terms of standard norms. In
this case the fact that the connection $\nabla$ is not a metric connection with
respect to the rescaled metric $\rho^{-2}g$ is not a problem: the proof only
uses the Leibniz rule (together with H\"older's inequality for $L^p$ norms and
the Sobolev Embedding Theorems which we have just proved).

The proof of the Sobolev Embedding Theorems for $(L,\hat{g},\rho)$ is similar.
For example, to prove Part I with $l=1$ and $k=0$ we locally define
$\hat{h}_y:=\rho^{-2}(x)\hat{g}_y$. Our assumption on $\hat{g}$ allows us to
substitute
$h$ with $\hat{h}$ in Equation \ref{eq:scaled_okzero}. The proof then continues
as before. Now consider the case $k=1$, \textit{i.e.} assume $u\in W^p_{2;sc}$
with respect to $\hat{g}$. Let $\hat{\nabla}$ denote the Levi-Civita connection
defined by $\hat{g}$. We can then study $\phi_x^*|\hat{\nabla}u|_{\hat{h}}$ as
before, obtaining the analogue of Equation \ref{eq:scaled_okbis} in terms of
$(\hat{h},\hat{\nabla})$ instead of $(h,\nabla)$. Since the Levi-Civita
connections
of $\hat{g}$ and $\hat{h}$ coincide we also obtain the analogue of Equation
\ref{eq:scaled_okter}. The proof then continues as before. 
\end{proof}

\begin{remark}
Compare the proof of Theorem \ref{thm:scaled_ok} with the ideas of Remark
\ref{rem:scalednorms}. The main issue raised in Remark \ref{rem:scalednorms}
concerned Kato's inequality for the rescaled metric $\rho^{-2}g$. In the proof
of the theorem this problem is solved by Assumption A3, which essentially allows
us to locally treat $\rho$ as a constant. Assumptions A1 and A2 are then similar
to the assumptions of Remark \ref{rem:scalednorms}.
\end{remark}

\begin{example}\label{e:radius} 
We now want to present two important examples of $(L,g,\rho)$ satisfying
Assumptions A1-A3 of Theorem \ref{thm:scaled_ok}.
\begin{enumerate}
\item Let $L$ be a smooth bounded domain in $\R^m$, endowed with the standard
metric
$\tilde{g}$. Given any $x\in L$ we
can define $\rho(x):=d(x,\partial L)$. This function satisfies Assumption A1
with $R_1=1$ and Assumption A2 with $R_2=0$. The triangle inequality shows that,
for all $y\in B(x,(1/2)\rho(x))$, $(1/2)\rho(x)\leq\rho(y)\leq(3/2)\rho(x)$.
This implies that Assumption A3 is also satisfied.
\item Given a compact oriented Riemannian manifold $(\Sigma,g')$, let
$L:=\Sigma\times (0,\infty)$ and $\tilde{g}:=dr^2+r^2g'$. Let $\theta$ denote
the generic point on $\Sigma$. There is a natural action
\begin{equation*}
\R^+\times L\rightarrow L,\ \ t\cdot(\theta,r):=(\theta,tr).
\end{equation*}
Given any $t\in\R^+$, it is simple to check that $t^*\tilde{g}=t^2\tilde{g}$.
For any $x\in L$, notice that $i_{tx}(\tilde{g})=i_x(t^*\tilde{g})$. We conclude
that $i_{tx}(\tilde{g})=ti_x(\tilde{g})$.  Analogously,
$Ric_{tx}(\tilde{g})=Ric_x(\tilde{g})$. It follows that, given any strictly
positive $f=f(\theta)$, the function $\rho(\theta,r):=rf(\theta)$
satisfies A1 and A2. It is simple to check that it also satisfies Assumption A3.
The simplest example is $f(\theta)\equiv 1$, \textit{i.e.}
$\rho(\theta,r)=r$. In Section \ref{s:manifolds} we will extend this example to the
category of ``conifolds''.
\end{enumerate} 
\end{example}

\begin{remark}\label{rem:tscaled_ok}
Since the norms $\|\cdot\|_{W^p_{k;sc}}$ are scale-invariant it is clear that
if the Sobolev Embedding Theorems hold for $(L,g,\rho)$ then they also hold
for $(L,t^2g,t\rho)$ with the same Sobolev constants. This is reflected in the
fact that Assumptions A1-A3 of Theorem \ref{thm:scaled_ok} are
scale-invariant.  
\end{remark}


\section{Weighted Sobolev spaces}\label{s:weighted}

In Section \ref{s:scaled} we mentioned that the finiteness condition
determined by the standard Sobolev norms is very
restrictive. This problem can be addressed by introducing a \textit{weight
function} $w=w(x)>0$ into the integrand. Coupling weights with scale functions
then produces very general and useful spaces, as follows. 

Let $(L,g)$ be a Riemannian manifold endowed with two positive
functions $\rho$ and $w$. Given any metric pair $(E,\nabla)$, the
\textit{weighted Sobolev spaces} are defined by
\begin{equation}\label{eq:weighted_sob}
W^p_{k;w}(E):=\mbox{Banach space completion of the space }
\left\{\sigma\in
C^\infty(E):\|\sigma\|_{W^p_{k;w}}<\infty\right\},
\end{equation}
where we use the norm
$\|\sigma\|_{W^p_{k;w}}:=\left(\Sigma_{j=0}
^k\int_L|w\rho^j\nabla^j\sigma|_g^p\rho^ {
-m }
\,\mbox{vol}_g\right)^{1/p}$.

We can also define the \textit{weighted spaces of $C^k$ sections}
\begin{equation}\label{eq:weighted_C^k}
C^k_w(E):=\left\{\sigma\in C^k(E): \|\sigma\|_{C^k_w}<\infty\right\},
\end{equation}
where we use the norm $\|\sigma\|_{C^k_w}:=\sum_{j=0}^k \mbox{sup}_{x\in
L}|w\rho^j\nabla^j\sigma|_g$. Once again, these norms define Banach spaces.

\begin{theorem}\label{thm:weighted_ok}
Let $(L,g)$ be a Riemannian manifold endowed with positive functions
$\rho$ and $w$. Assume $\rho$ satisfies the assumptions of Theorem
\ref{thm:scaled_ok} with respect to constants $R_1$, $R_2$, $R_3$ and $\zeta$.
Assume also that there exists a positive constant $R_4$ such that, $\forall x\in
L, \forall y\in B(x,\zeta\rho(x))$,
\begin{equation*}
(1/R_4)w(x)\leq w(y)\leq R_4 w(x).
\end{equation*}
Then all parts of the Sobolev Embedding Theorems hold for the weighted norms
defined by $(\rho,w)$ and for any metric pair $(E,\nabla)$.

Now let $\hat{g}$ be a second Riemannian metric on $L$ such that, for some
$C_0>0$, $(1/C_0)g\leq \hat{g}\leq C_0 g$. Then the weighted Sobolev Embedding
Theorems hold also for $(L,\hat{g},\rho,w)$ and for any metric pair
$(E,\nabla)$. The Sobolev constants of $\hat{g}$ depend only on the Sobolev
constants of $g$
and on $C_0$. 
\end{theorem}
\begin{proof}
The proof is a small modification of the proof of Theorem \ref{thm:scaled_ok}:
one needs simply to take into account the weights by multiplying Equations
\ref{eq:scaled_okzero} and \ref{eq:scaled_okter} by $w(x)$. The assumption on
$w$ allows us, up to a change of constants, to replace the (locally) constant
quantity $w(x)$ with the function $w(y)$.
\end{proof}

\begin{remark}\label{rem:tweighted_ok}
 Choose any constant $\beta\in\R$. Define \textit{rescaling} to be an
action of $\R^+$ on 
$(L,g,\rho,w)$, via $t\cdot(L,g,\rho,w):=(L,t^2g, t\rho,t^\beta w)$. Then
$\|\sigma\|_{W^p_{k;w}}$, calculated with
respect to $t\cdot(L,g,\rho,w)$, coincides with
$t^\beta\|\sigma\|_{W^p_{k;w}}$,
calculated with respect to $(L,g,\rho,w)$: this shows that these weighted norms
are in general not invariant under rescaling. However, if the Sobolev Embedding
Theorems hold for $(L,g,\rho,w)$ then, multiplying by the factor $t^\beta$, we
see that they hold for $(L,t^2g, t\rho,t^\beta w)$ with the same Sobolev
constant. This is reflected in the fact that the hypotheses of Theorem
\ref{thm:weighted_ok} are $t$-invariant.
\end{remark}


\section{Application: manifolds with ends modelled on cones and cylinders}\label{s:manifolds}

We now introduce the category of ``conifolds''. These Riemannian manifolds are a well-known example for the theory of weighted Sobolev spaces. They will also provide a useful framework for our study of desingularizations. It will also be useful to define  the
analogous ``cylindrical'' category, both for its affinities to conifolds and as
a tool for studying them.

\begin{definition}\label{def:manifold_ends}
Let $L^m$ be a smooth manifold. We say $L$ is a \textit{manifold with ends} if
it satisfies the following conditions:
\begin{enumerate}
\item We are given a compact subset $K\subset L$ such that $S:=L\setminus K$ has
a finite number of connected components $S_1,\dots,S_e$, \textit{i.e.}
$S=\amalg_{i=1}^e S_i$.
\item For each $S_i$ we are given a connected ($m-1$)-dimensional compact
manifold $\Sigma_i$ without boundary. 
\item There exist diffeomorphisms $\phi_i:\Sigma_i\times [1,\infty)\rightarrow
\overline{S_i}$.
\end{enumerate}
We then call the components $S_i$ the \textit{ends} of $L$ and the manifolds
$\Sigma_i$ the \textit{links} of $L$. We denote by
$\Sigma$ the union of the links of $L$. 
\end{definition}

\begin{definition}\label{def:metrics_ends}
Let L be a manifold with ends. Let $g$ be a Riemannian metric on $L$. Choose an
end $S_i$ with corresponding link $\Sigma_i$.

We say that $S_i$ is a \textit{conically singular} (CS) end if the following
conditions hold:
\begin{enumerate}
\item $\Sigma_i$ is endowed with a Riemannian metric $g_i'$.

We then let $(\theta,r)$ denote the generic point on the product manifold
$C_i:=\Sigma_i\times (0,\infty)$ and $\tg_i:=dr^2+r^2g_i'$ denote the
corresponding \textit{conical metric} on $C_i$.
\item There exist a constant $\nu_i>0$ and a diffeomorphism
$\phi_i:\Sigma_i\times (0,\epsilon]\rightarrow \overline{S_i}$ such that, as
$r\rightarrow
0$ and for all $k\geq 0$,
$$|\tnabla^k(\phi_i^*g-\tg_i)|_{\tg_i}=O(r^{\nu_i-k}),$$
where $\tnabla$ is the Levi-Civita connection on $C_i$ defined by $\tg_i$. 
\end{enumerate}
We say that $S_i$ is an \textit{asymptotically conical} (AC) end if the
following conditions hold:
\begin{enumerate}
\item $\Sigma_i$ is endowed with a Riemannian metric $g_i'$.

We again let $(\theta,r)$ denote the generic point on the product manifold
$C_i:=\Sigma_i\times (0,\infty)$ and $\tg_i:=dr^2+r^2g_i'$ denote the
corresponding \textit{conical metric} on $C_i$.
\item There exist a constant $\nu_i<0$ and a diffeomorphism
$\phi_i:\Sigma_i\times [R,\infty)\rightarrow \overline{S_i}$ such that, as
$r\rightarrow
\infty$ and for all $k\geq 0$,
$$|\tnabla^k(\phi_i^*g-\tg_i)|_{\tg_i}=O(r^{\nu_i-k}),$$
where $\tnabla$ is the Levi-Civita connection on $C_i$ defined by $\tg_i$.
\end{enumerate}
In either of the above situations we call $\nu_i$ the \textit{convergence rate}
of
$S_i$.
\end{definition}

\begin{remark}\label{rem:metrics_ends}
Let $(L,g)$ be a manifold with ends. Assume $S_i$ is an
AC end as in Definition \ref{def:metrics_ends}. Using the
notation of Remark \ref{rem:scalednorms} we can rewrite this condition as
follows: for all $k\geq 0$,
$$|\tnabla^k(\phi_i^*g-\tg_i)|_{r^{-2}\tg_i\otimes \tg_i}=O(r^{\nu_i}).$$
In particular there exist constants $C_k>0$ such that
$$|\tnabla^k(\phi_i^*g-\tg_i)|_{r^{-2}\tg_i\otimes \tg_i}\leq C_kR^{\nu_i}.$$
By making $R$ larger if necessary, we can assume $C_0R^{\nu_i}$ is small.
This implies that $\phi_i^*g$ and $\tilde{g}_i$
are scaled-equivalent in the sense of Definition
\ref{def:equivalentscaledmetrics}, cf. Remark \ref{rem:scaledequivalence}. The above conditions are stable under duality and tensor products so one can prove 
that, for any tensor $\sigma$ on $L$ and as $r\rightarrow\infty$,
\begin{equation*}
 |\sigma|_{\phi_i^*g}=|\sigma|_{\tg_i}\left(1+O(r^{\nu_i})\right).
\end{equation*}
If $\sigma=df$ for some function $f$ on $L$, we can multiply both sides by $r$ to obtain an analogous estimate in terms of the rescaled metrics:
\begin{equation*}
 |df|_{r^{-2}\phi_i^*g}=|df|_{r^{-2}\tg_i}\left(1+O(r^{\nu_i})\right).
\end{equation*}
Furthermore, let $A:=\nabla-\tnabla$ denote the difference of the two
connections defined by
$\phi_i^*g$ and $\tg_i$. Then, as in Remark \ref{rem:equivalence},
Definition \ref{def:metrics_ends} implies that $|A|_{\tg_i}=O(r^{\nu_i-1})$.
This leads to
\begin{align*}
 |\nabla^2f|_{\phi_i^*g}&=|\tnabla^2f|_{\tg_i}\left(1+O(r^{\nu_i})\right)+
|df|_{\tg_i}O(r^{\nu_i-1}),\\
|\mbox{tr}_{\phi_i^*g}\nabla^2
f|&=|\mbox{tr}_{\tg_i}\tnabla^2f|\left(1+O(r^{\nu_i})\right)
+|df|_{\tg_i} O(r^{\nu_i-1}).
\end{align*}
Multiplying these equations by $r^2$ we can re-write them as
\begin{align*}
 |\nabla^2f|_{r^{-2}\phi_i^*g}&=|\tnabla^2f|_{r^{
-2}\tg_i}+O(r^{\nu_i})\left(|\tnabla^2f|_{r^{
-2}\tg_i}+|df|_{r^{-2}\tg_i}\right),\\
|r^2\Delta_{\phi_i^*g}f|&=|r^2\Delta_{\tg_i}f|+O(r^{\nu_i})\left(|r^2\Delta_{\tg_i}f|+|df|_{r^{-2}\tg_i}\right).
\end{align*}
Analogous comments apply to higher derivatives and to CS ends.
\end{remark}

\begin{definition}
Let $(L,g)$ be a manifold with ends endowed with a Riemannian metric. We say
that $L$ is a \textit{CS} (respectively, \textit{AC})
manifold if all ends are conically singular (respectively, asymptotically
conical). We say that $L$ is a \textit{CS/AC}
manifold if all ends are either conically singular or asymptotically conical. We
use the generic term \textit{conifold} to indicate any CS, AC or CS/AC manifold.

When working with a CS/AC manifold we will often index the CS (``small") ends
with numbers $\{1,\dots,s\}$ and the AC (``large") ends with numbers
$\{1,\dots,l\}$. Furthermore we will denote the union of the CS links
(respectively, of the CS ends) by $\Sigma_0$ (respectively, $S_0$) and those
corresponding to the AC links and ends by $\Sigma_\infty$, $S_\infty$. 
\end{definition}

\begin{remark}\label{rem:cpt_conifolds}
It is useful to include smooth compact manifolds in the category of conifolds: they are precisely those for which the set of ends is empty.
\end{remark}

We now need to choose which function spaces to work with on conifolds. It
turns out that the most useful
classes of function spaces are precisely those of Section \ref{s:weighted}. One
needs only to choose appropriate functions $\rho$ and $w$ satisfying the
assumptions of Theorem \ref{thm:weighted_ok}, as follows.

Regarding notation, given a vector
$\boldsymbol{\beta}=(\beta_1,\dots,\beta_e)\in \R^e$ and $j\in\N$ we set
$\boldsymbol{\beta}+j:=(\beta_1+j,\dots,\beta_e+j)$. We write
$\boldsymbol{\beta}\geq\hat{\boldsymbol{\beta}}$ iff $\beta_i\geq\hat{\beta_i}$ for all $i=1,\dots, e$.

\begin{definition}\label{def:csac_sectionspaces}
Let $L$ be a conifold with metric $g$. We say that a smooth function
$\rho:L\rightarrow (0,\infty)$ is a \textit{radius function} if
$\phi_i^*\rho=r$, where $\phi_i$ are the
diffeomorphisms of Definition \ref{def:metrics_ends}. Given any vector
$\boldsymbol{\beta}=(\beta_1,\dots,\beta_e)\in\R^e$, choose a function
$\boldsymbol{\beta}:L\rightarrow \R$ which, on each end $S_i$, restricts to the
constant
$\beta_i$.  Then $\rho$ and $w:=\rho^{-\beta}$ satisfy the assumptions of
Theorem \ref{thm:weighted_ok}, cf. Example \ref{e:radius}. We call
$(L,g,\rho,\boldsymbol{\beta})$ a \textit{weighted conifold}.

Given any metric pair $(E,\nabla)$ we define weighted spaces
$C^k_{\boldsymbol{\beta}}(E)$ and $W^p_{k,\boldsymbol{\beta}}(E)$ as in Section
\ref{s:weighted}. We can equivalently define the space
$C^k_{\boldsymbol{\beta}}(E)$ to be the space of sections $\sigma\in C^k(E)$
such that $|\nabla^j \sigma|=O(r^{\boldsymbol{\beta}-j})$ as $r\rightarrow 0$
(respectively, $r\rightarrow\infty$) along each CS (respectively, AC) end.

In the case of a CS/AC manifold we will often separate the CS and AC weights,
writing $\boldsymbol{\beta}=(\boldsymbol{\mu},\boldsymbol{\lambda})$ for some
$\boldsymbol{\mu}\in \R^s$ and some $\boldsymbol{\lambda}\in \R^l$. We then
write $C^k_{(\boldsymbol{\mu},\boldsymbol{\lambda})}(E)$ and
$W^p_{k,(\boldsymbol{\mu},\boldsymbol{\lambda})}(E)$.
\end{definition}

One can extend to these weighted spaces many results valid for standard Sobolev
spaces. H\"{o}lder's inequality is one example.

\begin{lemma}[Weighted H\"{o}lder's inequality] \label{l:weightedhoelder}
Let $(L,g)$ be a conifold.
Then, for all $p> 1$ and
$\boldsymbol{\beta}=\boldsymbol{\beta_1}+\boldsymbol{\beta_2}$, 
\begin{equation*}
\|uv\|_{L^1_{\boldsymbol{\beta}}}\leq
\|u\|_{L^p_{\boldsymbol{\beta_1}}}\cdot\|v\|_{L^{p'}_{\boldsymbol{\beta_2}}}.
\end{equation*}
More generally, assume $\frac{1}{q}=\frac{1}{q_1}+\frac{1}{q_2}$. Then 
\begin{equation*}
\|uv\|_{L^q_{\boldsymbol{\beta}}}\leq
\|u\|_{L^{q_1}_{\boldsymbol{\beta_1}}}\cdot\|v\|_{L^{q_2}_{\boldsymbol{\beta_2}}}.
\end{equation*}
\end{lemma}
\begin{proof}
\begin{align*}
\|uv\|_{L^1_{\boldsymbol{\beta}}}
&=\int_L(\rho^{-\boldsymbol{\beta_1}}u\rho^{-m/p})(\rho^{-\boldsymbol{\beta_2}}
v\rho^ { -m/p' } )\, \mbox {vol}_g\\
&\leq
\|\rho^{-\boldsymbol{\beta_1}}u\rho^{-m/p}\|_{L^p}\cdot\|\rho^{-\boldsymbol{\beta_2}}
v\rho^ { -m/p' } \|_ { L^ { p' } } \\
&=\|u\|_{L^p_{\boldsymbol{\beta_1}}}\cdot\|v\|_{L^{p'}_{\boldsymbol{\beta_2}}}.
\end{align*}
The general case is similar.
\end{proof}

\begin{corollary}\label{cor:embedding}
 Let $(L,g,\boldsymbol{\beta})$ be a weighted conifold. Then
all parts of the weighted Sobolev Embedding Theorems hold for any metric pair
$(E,\nabla)$.

Furthermore, assume $kp>m$. Then the corresponding weighted
Sobolev spaces are closed under multiplication, in the following sense. For
any $\boldsymbol{\beta}_1$ and $\boldsymbol{\beta_2}$ there exists $C>0$ such
that, for all $u\in W^p_{k,\boldsymbol{\beta_1}}$ and $v\in
W^p_{k,\boldsymbol{\beta_2}}$,
\begin{equation*}
\|uv\|_{W^p_{k,\boldsymbol{\beta_1}+\boldsymbol{\beta_2}}}\leq
C\|u\|_{W^p_{k,\boldsymbol{\beta_1}}}\cdot\|v\|_{W^p_{k,\boldsymbol{\beta_2}}}.
\end{equation*}
\end{corollary}
\begin{proof}
Let $(L,g)$ be a conifold. Write $L=K\cup S$ as in Definition
\ref{def:manifold_ends} and let $C_i$ denote
the cone corresponding to the end $S_i$. Example \ref{e:radius}
showed that the assumptions for the scaled Sobolev Embedding Theorems hold for
$(C_i,\tilde{g}_i,r)$. The same is true for the weighted Sobolev Embedding
Theorems. Using the compactness of $K$ we conclude that these
assumptions, thus the theorems, hold for $L$ with respect to any metric
$\hat{g}$ such that $\phi_i^*\hat{g}=\tg_i$ on each end. As in Remark
\ref{rem:metrics_ends} one can assume that
$\phi_i^*g$ and $\tilde{g}_i$
are scaled-equivalent so there exists $C_0>0$ such that $(1/C_0)\tg_i\leq
\phi_i^*g\leq C_0\tg_i$. Again using the compactness of $K$ we may thus assume
that $(1/C_0)\hat{g}\leq g\leq C_0\hat{g}$. Theorem \ref{thm:weighted_ok} now
shows that the weighted Sobolev Embedding Theorems hold for
$(L,g)$. The fact that weighted Sobolev spaces
are closed with respect to products can
be proved as in Theorem \ref{thm:scaled_ok}, using Lemma
\ref{l:weightedhoelder}.
\end{proof}

\begin{remark}\label{rem:embedding}
Let $(L,g)$ be an AC manifold. Notice that for
$\hat{\boldsymbol{\beta}}\geq\boldsymbol{\beta}$ there exist continuous
embeddings
$W^r_{k,\boldsymbol{\beta}}\hookrightarrow W^r_{k,\hat{\boldsymbol{\beta}}}$.
The analogous statement is true for the weighted $C^k$ spaces. By composition
Corollary \ref{cor:embedding} thus leads to the following statements:
\begin{enumerate}
\item If $lp<m$ then there exists a continuous embedding
$W^p_{k+l,\boldsymbol{\beta}}(E)\hookrightarrow
W^{p^*_l}_{k,\hat{\boldsymbol{\beta}}}(E)$.
\item If $lp=m$ then, for all $q\in [p,\infty)$, there exist continuous
embeddings $W^p_{k+l,\boldsymbol{\beta}}(E)\hookrightarrow
W^q_{k,\hat{\boldsymbol{\beta}}}(E)$.
\item If $lp>m$ then there exists a continuous embedding
$W^p_{k+l,\boldsymbol{\beta}}(E)\hookrightarrow
C^k_{\hat{\boldsymbol{\beta}}}(E)$. 
\end{enumerate}
Notice that if $(L,g)$ is a CS
manifold then the behaviour on the ends is studied in terms of $r\rightarrow 0$
rather than $r\rightarrow \infty$. In this case the same conclusions hold for
the opposite situation $\hat{\boldsymbol{\beta}}\leq\boldsymbol{\beta}$.
Finally, let $(L,g)$ be a CS/AC manifold with
$\boldsymbol{\beta}=(\boldsymbol{\mu},\boldsymbol{\lambda})$. Then the same
conclusions hold for all
$\hat{\boldsymbol{\beta}}=(\hat{\boldsymbol{\mu}},\hat{\boldsymbol{\lambda}})$
with
$\hat{\boldsymbol{\mu}}\leq\boldsymbol{\mu}$,
$\hat{\boldsymbol{\lambda}}\geq\boldsymbol{\lambda}$. 
\end{remark}

We now want to show that all the above notions and results are
scale-independent, as long as we rescale the weight function
correctly to take into account the possibility of variable weights. We start by examining the properties of $(L,t^2g)$.

\begin{lemma}\label{l:rescaledconifold}
 Let $(L,g)$ be a conifold. For each AC end $S_i$ let
$\phi_i:\Sigma_i\times [R,\infty)\rightarrow \overline{S_i}$ denote the
diffeomorphism of Definition \ref{def:metrics_ends}. In particular, for all
$k\geq 0$ there exist $C_k>0$ such that, for $r\geq R$,
$$|\tnabla^k(\phi_i^*g-\tg_i)|_{r^{-2}\tg_i\otimes\tg_i}\leq C_kr^{\nu_i}\leq
C_kR^{\nu_i}.$$
As seen in Remark \ref{rem:metrics_ends}, we can thus assume that
$\phi_i^*g$, $\tg_i$ are scaled-equivalent.

Choose any $t>0$. Define the diffeomorphism
\begin{equation*}
 \phi_{t,i}:\Sigma_i\times [tR,\infty)\rightarrow \overline{S_i},\ \
\phi_{t,i}(\theta,r):=\phi_i(\theta,r/t).
\end{equation*}
Then, for $r\geq tR$ and with respect to the same $C_k$,
there are $t$-uniform estimates
$$|\tnabla^k(\phi_{t,i}^*(t^2g)-\tg_i)|_{r^{-2}\tg_i\otimes \tg_i}\leq
C_k(r/t)^{\nu_i}\leq
C_kR^{\nu_i}.$$
Analogously, for each CS end $S_i$ let $\phi_i$ denote the
diffeomorphism of Definition \ref{def:metrics_ends}. Define the diffeomorphism
\begin{equation*}
 \phi_{t,i}:\Sigma_i\times (0,t\epsilon]\rightarrow \overline{S_i},\ \
\phi_{t,i}(\theta,r):=\phi_i(\theta,r/t).
\end{equation*}
Then there are $t$-uniform estimates as above.

In particular, with respect to these diffeomorphisms, $(L,t^2g)$ is again a
conifold. If $\rho$ is a radius function for $(L,g)$ then $t\rho$ is a radius
function for $(L,t^2g)$.
\end{lemma}
\begin{proof}
Define the map
$$\delta_t:\Sigma_i\times \R^+\rightarrow \Sigma_i\times\R^+,\ \
(\theta,r)\mapsto (\theta,tr).$$
Since $\delta_t$ is simply a rescaling it preserves the Levi-Civita connection
$\tilde{\nabla}$. Notice that
$\phi_{t,i}=\phi_i\circ\delta_{1/t}$. It is simple to check that
$\delta_{1/t}^*(t^2\tg_i)=\tg_i$. Thus, for $r\geq tR$,
\begin{align*}
|\tnabla^k(\phi_{t,i}^*(t^2g)-\tg_i)|_{\tg_i\otimes\tg_i}
&=|\tnabla^k(\delta_{1/t}^*\phi_i^*(t^2g)-\tg_i)|_{\tg_i\otimes\tg_i}\\
&=\delta_{1/t}^*\left(|\tnabla^k(\phi_i^*(t^2g)-t^2\tg_i)|_{t^2\tg_i\otimes t^2\tg_i}\right)\\
&=\delta_{1/t}^*\left(|\tnabla^k(\phi_i^*g-\tg_i)|_{t^2\tg_i\otimes \tg_i}\right)\\
&\leq t^{-k}C_k(r/t)^{\nu_i-k} = C_k(r/t)^{\nu_i}r^{-k},
\end{align*}
where in the last line the factor $t^{-k}$ comes from measuring $\tilde{\nabla}^k$ using $t^2\tilde{g}_i$, cf. Remark \ref{rem:scalednorms}.
These inequalities can be rescaled as in Remark \ref{rem:metrics_ends} to
obtain the desired $t$-uniform estimates.

Now notice that 
$$\phi_{t,i}^*(t\rho)_{|(\theta,r)}=t\rho\circ\phi_{t,i}(\theta,
r)=t\rho\circ\phi_i(\theta, r/t)=tr/t=r, $$
so $t\rho$ is a radius function in the sense of Definition
\ref{def:csac_sectionspaces}.
CS ends can be studied analogously.
\end{proof}

The following result is a direct consequence of Theorem \ref{thm:weighted_ok}
and Remark \ref{rem:tweighted_ok}.
\begin{corollary}\label{cor:rescaledconifold}
Let $(L,g)$ be a conifold. Then,
for all $t>0$: 
\begin{enumerate}
\item Choose a constant weight
$\boldsymbol{\beta}$. Define weighted Sobolev spaces
$W^p_{k,\boldsymbol{\beta}}$ as in Section \ref{s:weighted} using the metric
$t^2g$, the scale function $t\rho$ and the weight function
$w:=(t\rho)^{-\boldsymbol{\beta}}$. Then 
all forms of the weighted Sobolev
Theorems hold for
$(L,t^2g,t\rho,(t\rho)^{-\boldsymbol{\beta}})$ with $t$-independent Sobolev
constants. 
\item More generally, let $\boldsymbol{\beta}$ be a function as
in Definition \ref{def:csac_sectionspaces}. Choose a constant ``reference''
weight $\boldsymbol{\beta}'$ and define weighted Sobolev spaces
$W^p_{k,\boldsymbol{\beta}}$ as in Section \ref{s:weighted} using the metric
$t^2g$, the scale function $t\rho$ and the weight function
$w_t:=(t^\frac{\boldsymbol{\beta}'-\boldsymbol{\beta}}{\boldsymbol{\beta}}t\rho)^{-\boldsymbol{\beta
} } $. Then the weighted norms
$\|\cdot\|_{W^p_{k,\boldsymbol{\beta}}}$,
calculated with respect to these choices, coincide with 
$t^{-\boldsymbol{\beta}'}\|\cdot\|_{W^p_{k,\boldsymbol{\beta}}}$,
calculated with
respect to $(L,g,\rho,w:=\rho^{-\boldsymbol{\beta}})$. In particular,
all forms of the weighted Sobolev Embedding Theorems hold for
$(L,t^2g,t\rho,
w_t:=(t^\frac{\boldsymbol{\beta}'-\boldsymbol{\beta}}{\boldsymbol{\beta}}t\rho)^{-\boldsymbol{\beta
} } )$
with $t$-independent Sobolev constants. 
\end{enumerate}
\end{corollary}
\begin{remark}\label{rem:corrective_term}
 Compare the weights used in parts (1) and (2) above. Basically, to deal with variable weights we introduce a corrective factor of the form $t^{\boldsymbol{\beta}-\boldsymbol{\beta}'}$: since the exponent is bounded, for fixed $t$ this doesn't affect the decay/growth condition on the ends. Its effect is simply to yield estimates which are uniform with respect to $t$.
\end{remark}

We conclude this section by summarizing the main definitions and properties of
a second class of manifolds with ends, modelled on cylinders. We will see that
the corresponding theory is closely related to that of conifolds.

\begin{definition}\label{def:metrics_endsbis}
Let $L$ be a manifold with ends. Let $g$ be a Riemannian metric on $L$. Choose an
end $S_i$ with corresponding link $\Sigma_i$. We say that $S_i$ is an
\textit{asymptotically cylindrical} (A.Cyl.) end if the
following conditions hold:
\begin{enumerate}
\item $\Sigma_i$ is endowed with a Riemannian metric $g_i'$.

We then let $(\theta,z)$ denote the generic point on the product manifold
$C_i:=\Sigma_i\times (-\infty,\infty)$ and $\tilde{h}_i:=dz^2+g_i'$ denote the
corresponding \textit{cylindrical metric} on $C_i$.
\item There exist a constant $\nu_i<0$ and a diffeomorphism
$\phi_i:\Sigma_i\times [R',\infty)\rightarrow \overline{S_i}$ such that, as
$z\rightarrow
\infty$ and for all $k\geq 0$,
$$|\tnabla^k(\phi_i^*g-\tilde{h}_i)|_{\tilde{h}_i}=O(e^{\nu_i z}),$$
where $\tnabla$ is the Levi-Civita connection on $C_i$ defined by $\tilde{h}_i$.
\end{enumerate}
We say
that $L$ is a \textit{A.Cyl.}
manifold if all ends are asymptotically
cylindrical. 
\end{definition}

For the purposes of
this paper the function spaces of most interest on A.Cyl. manifolds are not the
ones already
encountered, cf. Section \ref{s:std} and Example \ref{e:std_ok}. Instead, we
use the following.

\begin{definition}\label{def:acyl_sectionspaces}
Let $(L,h)$ be a A.Cyl. manifold. We say that a smooth function
$\zeta:L\rightarrow [1,\infty)$ is a \textit{radius function} if
$\phi_i^*\zeta= z$, where $\phi_i$ are the
diffeomorphisms of Definition \ref{def:metrics_ends}. Given any vector
$\boldsymbol{\beta}=(\beta_1,\dots,\beta_e)\in\R^e$, choose a function
$\boldsymbol{\beta}$ on $L$ which, on each end $S_i$, restricts to the constant
$\beta_i$. We call $(L,h,\zeta,\boldsymbol{\beta})$ a \textit{weighted A.Cyl.
manifold}. Given any metric pair $(E,\nabla)$ we define Banach spaces of
sections of $E$ in the following two ways.

The \textit{weighted spaces of $C^k$ sections} of $E$ are defined by
\begin{equation}\label{eq:weighted_C^kcyl}
C^k_{\boldsymbol{\beta}}(E):=\left\{\sigma\in C^k(E):
\|\sigma\|_{C^k_{\boldsymbol{\beta}}}<\infty\right\},
\end{equation}
where we use the norm $\|\sigma\|_{C^k_{\boldsymbol{\beta}}}:=\sum_{j=0}^k
\mbox{sup}_{x\in L}|e^{-\boldsymbol{\beta}(x)\zeta(x)}\nabla^j\sigma|$. 

The \textit{weighted Sobolev spaces} are defined by
\begin{equation}\label{eq:weighted_cyl}
W^p_{k,\boldsymbol{\beta}}(E):=\mbox{Banach space completion of the space
}\left\{\sigma\in
C^\infty(E):\|\sigma\|_{W^p_{k,\boldsymbol{\beta}}}<\infty\right\},
\end{equation}
where $p\in [1,\infty)$, $k\geq 0$ and we use the norm
$\|\sigma\|_{W^p_{k,\boldsymbol{\beta}}}:=\left(\sum_{j=0}^k \int_L
|e^{-\boldsymbol{\beta} \zeta}\nabla^j\sigma|^p\,\mbox{vol}_h\right)^{1/p}$.

Both types of spaces are independent of the particular choices made.
\end{definition}

\begin{remark}\label{rem:acyl_equivalent}
It is simple to see that the norm $\|\sigma\|_{W^p_{k,\boldsymbol{\beta}}}$ is
equivalent to the norm defined by $\sum_{j=0}^k(\int_L
|\nabla^j(e^{-\boldsymbol{\beta} \zeta}\sigma)|^p\,\mbox{vol}_h)^{1/p}$. This
leads to the following fact. 

Let $W^p_k(E)$ denote the standard Sobolev spaces for $(L,h)$ introduced in
Section \ref{s:std}. Let $e^{\boldsymbol{\beta}\zeta}\cdot W^p_k$ denote the
space of all sections of $E$ of the form
$\sigma=e^{\boldsymbol{\beta}\zeta}\tau$ for some $\tau\in W^p_k(E)$, endowed
with the norm $\|\sigma\|:=\|\tau\|$. Then
$W^p_{k,\boldsymbol{\beta}}(E)=e^{\boldsymbol{\beta}\zeta}\cdot W^p_k(E)$ as
sets and the norms are equivalent. Analogously, the spaces
$C^k_{\boldsymbol{\beta}}(E)$ are equivalent to the spaces
$e^{\boldsymbol{\beta}\zeta}\cdot C^k(E)$, where $C^k(E)$ are the standard
spaces of $C^k$ sections used in Section \ref{s:std}.
\end{remark}

As before, weighted spaces defined with respect to A.Cyl. metrics and
cylindrical metrics are equivalent. Remark \ref{rem:acyl_equivalent} allows us
to reduce
the weighted Sobolev Embedding Theorems for A.Cyl. manifolds to the standard
Sobolev Embedding Theorems, obtaining results analogous to Corollary
\ref{cor:embedding} and Remark \ref{rem:embedding}. According to \cite{hebey}
Theorem 3.1 and Proposition 3.2 the spaces $C^\infty_c$
are dense in the standard Sobolev spaces defined for manifolds whose
ends are exactly cylindrical. The same is then true for weighted Sobolev
spaces on A.Cyl. manifolds.

\begin{remark}\label{rem:spacescoincide}
It is interesting to compare Definitions \ref{def:acyl_sectionspaces} and
\ref{def:csac_sectionspaces}. Assume $(L,h)$ is an A.Cyl. manifold with respect
to certain diffeomorphisms $\phi_i=\phi_i(\theta,z)$ as in Definition
\ref{def:metrics_ends}. Since the corresponding weighted Sobolev spaces are
equivalent we may assume that $h$ is exactly cylindrical on each end,
\textit{i.e.} using the notation of Definition \ref{def:metrics_ends} it can be
written $h=dz^2+g_i'$. Consider the conformally rescaled metric
$g:=e^{2\zeta}h$. Using the change of variables $r=e^z$ it is simple to check
that $g=dr^2+r^2g_i'$. This implies that $(L,g)$ is an AC manifold with respect
to the diffeomorphisms $\phi_i(\theta,\log z)$. Viceversa, any AC metric on $L$
defines a conformally equivalent A.Cyl. metric. Notice that if $z\in
(R',\infty)$ then $r\in (R,\infty)$ with $R:=e^{R'}$ and that
$r^{-m}\mbox{vol}_g=\mbox{vol}_h$. Thus, by change of variables, 
\begin{equation}
\int_R^\infty\int_{\Sigma}|r^{-\boldsymbol{\beta}}\sigma|^pr^{-m}\,\mbox{vol}
_g=\int_{R'}^\infty\int_{\Sigma}|e^{-\boldsymbol{\beta}
z}\sigma|^p\,\mbox{vol}_h.
\end{equation}
This shows that the spaces $L^p_{\boldsymbol{\beta}}(E)$ of sections of $E$
coincide for $(L,g)$ and $(L,h)$, while the corresponding norms are
equivalent (but again, as in Remark \ref{rem:scalednorms}, one may need to take
into account which metric is being used on $E$ in the two cases).

The same is true also for Sobolev spaces of higher order. Specifically, an
explicit calculation shows that the Levi-Civita connections defined by $h$ and
$g$ are equivalent, \textit{i.e.} the corresponding Christoffel symbols coincide
up to constant multiplicative factors. It thus makes no difference which metric
is used to define $\nabla$. On the other hand, the norm inside the integral does
depend on the choice of metric. For example,
\begin{equation}
\int_R^\infty\int_{\Sigma}|r^{-\boldsymbol{\beta}+j}\nabla^j
f\sigma|_g^pr^{-m}\,\mbox{vol}_g=\int_{R'}^\infty\int_{\Sigma}|e^{-\boldsymbol{
\beta} z}\nabla^j \sigma|_h^p\,\mbox{vol}_h.
\end{equation}
This proves that the spaces $W^p_{k,\boldsymbol{\beta}}(E)$ are equivalent.

Analogous results hold for CS manifolds: if $h$ is A.Cyl. then $g:=e^{-2\zeta}h$
is CS. In this case
\begin{equation}
\int_0^{\epsilon}\int_{\Sigma}|r^{-\boldsymbol{\beta}}f|^pr^{-m}\,\mbox{vol}
_g=\int_{-\log \epsilon}^\infty\int_{\Sigma}|e^{\boldsymbol{\beta}
z}f|^p\,\mbox{vol}_h,
\end{equation}
so the space $L^p_{\boldsymbol{\beta}}$ for $(L,g)$ coincides with the space
$L^p_{-\boldsymbol{\beta}}$ for $(L,h)$.

These facts show, for example, that the Sobolev Embedding Theorems for
conifolds and A.Cyl. manifolds are simply two different points of view on the
same result. They also show that $C^\infty_c$ is dense in all weighted Sobolev
spaces on conifolds because, as already seen, this is true on A.Cyl.
manifolds. Finally, they show that in Remark \ref{rem:metrics_ends} we are
really using the cylindrical metric $r^{-2}\tg=\tilde{h}$ to ``measure''
$\tnabla^k$ (in the sense of Remark \ref{rem:scalednorms}).
\end{remark}


\part{Elliptic estimates}

We now turn to the theory of elliptic operators via weighted Sobolev spaces,
focusing on Fredholm and index results for the manifolds discussed in Section \ref{s:manifolds}. Results of this
kind have been proved by various authors, \textit{e.g.} Lockhart-McOwen
\cite{lockhartmcowen}, Lockhart \cite{lockhart} and Melrose \cite{melrose}. We
will follow the point of view of Lockhart and McOwen to which we refer for
details, see also Joyce-Salur \cite{joycesalur}.

\section{Fredholm results for elliptic operators on
A.Cyl. manifolds}\label{s:acyl_analysis}

We start with the case of A.Cyl.
manifolds. The theory requires appropriate assumptions on the asymptotic
behaviour of the operators, which we roughly summarize as follows.
\begin{definition}\label{def:acyl_bundles}
Given a manifold $\Sigma$, consider the projection
$\pi:\Sigma\times\R\rightarrow \Sigma$. A vector bundle $E_{\infty}$ on
$\Sigma\times\R$ is \textit{translation-invariant} if it is of the form
$\pi^*E'$, for some vector bundle $E'$ over $\Sigma$. We define the notion of
translation-invariant metrics and connections analogously. 

Let $P_\infty:C^\infty(E_\infty)\rightarrow C^\infty(F_\infty)$ be a
differential
operator between translation-invariant vector bundles. We say that $P_\infty$ is
\textit{translation-invariant} if it commutes with the action of $\R$ on
$\Sigma\times\R$ determined by translations; equivalently, writing
$P_\infty=\sum
A_k^\infty\cdot\nabla^k$ with respect to a translation-invariant $\nabla$, if
the
coefficient tensors $A_k^\infty$ are independent of $z$.

Let $(L,h)$ be an A.Cyl. manifold with link $\Sigma=\amalg\Sigma_i$. Let $E$,
$F$ be vector bundles
over $L$. Assume there exist translation-invariant vector
bundles $E_\infty$, $F_\infty$ over
$\Sigma\times\R$ such that, using
the notation of Definition \ref{def:metrics_endsbis}, $\phi_i^*(E_{|S_i})$
(respectively, $\phi_i^*(F_{|S_i})$)
coincides with the restriction to $\Sigma_i\times (R',\infty)$ of $E_\infty$
(respectively, $F_\infty$). 
Let
$P_\infty=\sum A_k^\infty\cdot\nabla^k:C^\infty(E_\infty)\rightarrow
C^\infty(F_\infty)$ be a translation-invariant linear differential operator of
order $n$. Consider a linear operator $P:C^\infty(E)\rightarrow C^\infty(F)$. We
say that $P$ is
\textit{asymptotic} to $P_\infty$ if on each end there exists $\nu_i<0$ such
that, writing $P=\sum
A_k\cdot\nabla^k$ (up to identifications) and as $z\rightarrow \infty$,
$$|\nabla^j(A_k-A_k^\infty)|=O(e^{\nu_i z}),$$
where $|\cdot|$ is defined by the translation-invariant metrics. We call $\nu_i$
the \textit{convergence rates} of the operator $P$.

In what follows, to define the spaces $W^p_{k,\boldsymbol{\beta}}(E)$,
we will assume that $E$ is endowed with a metric and a metric connection which
are asymptotic to the translation-invariant data on $E_\infty$, in the
appropriate sense.
\end{definition}

Assume $P$ is a linear operator of order $n$ with bounded coefficients $A_k$. It
follows from Definition \ref{def:acyl_sectionspaces} that, for all $p>1$, $k\geq
0$ and $\boldsymbol{\beta}$, $P$ extends to a continuous map 
\begin{equation}\label{eq:contextension}
P:W^p_{k+n,\boldsymbol{\beta}}(E)\rightarrow W^p_{k,\boldsymbol{\beta}}(F).
\end{equation}

\begin{remark}\label{rem:Pbeta}
It will sometimes be useful to denote by $P_{\boldsymbol{\beta}}$ the extended
operator of Equation \ref{eq:contextension}, so as to emphasize the particular
weight being used.
\end{remark}

Now assume $P$ is asymptotic to a translation-invariant operator $P_\infty$.
Then Equation \ref{eq:contextension} holds also for the operator
$e^{-\boldsymbol{\nu}\zeta}(P-P_\infty)$, where $\boldsymbol{\nu}<0$ denotes the
convergence rates of $P$ as in Definition \ref{def:acyl_bundles}. This implies
that the operator $P-P_\infty$ extends to a continuous map
\begin{equation}\label{eq:differenceoperator}
P-P_\infty:W^p_{k+n,\boldsymbol{\beta}}(E)\rightarrow
W^p_{k,\boldsymbol{\beta}+\boldsymbol{\nu}}(F).
\end{equation}
Notice that if $\boldsymbol{\beta}<\boldsymbol{\beta'}$ then
$W^p_{k+n,\boldsymbol{\beta}}(E)\subset W^p_{k+n,\boldsymbol{\beta'}}(E)$ and
that the operator $P_{\boldsymbol{\beta'}}$ extends the operator
$P_{\boldsymbol{\beta}}$. Notice also that $C_c^\infty(E)\subset
W^p_{k,\boldsymbol{\beta}}(E)$ as a dense subset. Dualizing this relation allows
us to identify the dual space $(W^p_{k,\boldsymbol{\beta}}(E))^*$ with a
subspace of the space of distributions $(C^\infty_c(E))^*$. It is customary to
denote this space $W^{p'}_{-k,-\boldsymbol{\beta}}(E)$. Endowed with the
appropriate norm, it again contains $C^\infty_c(E)$ as a dense subset. The
duality
map $W^{p'}_{-k,-\boldsymbol{\beta}}(E)\times
W^p_{k,\boldsymbol{\beta}}(E)\rightarrow\R$, restricted to this subset,
coincides with the map
\begin{equation}\label{eq:duality}
C^\infty_c(E)\times W^p_{k,\boldsymbol{\beta}}(E)\rightarrow\R,\ \
<\sigma,\sigma'>:=\int_L(\sigma,\sigma')_E\,\mbox{vol}_h.
\end{equation}
This map extends by continuity to a map defined on
$W^{p'}_{l,-\boldsymbol{\beta}}(E)\times W^p_{k,\boldsymbol{\beta}}(E)$ for all
$l\geq 0$, showing that $W^{p'}_{-k,-\boldsymbol{\beta}}(E)$ also contains all
spaces $W^{p'}_{l,-\boldsymbol{\beta}}(E)$.
It can be shown that $P$ admits continuous extensions as in Equation
\ref{eq:contextension} for any $k\in\Z$. 

\begin{lemma}\label{l:*=*}
Let $P:C^\infty(E)\rightarrow C^\infty(F)$ be a linear differential operator of
order $n$, asymptotic to a translation-invariant operator $P_\infty$. Let
$P^*:C^\infty(F)\rightarrow C^\infty(E)$ denote its formal adjoint. Consider the
continuous extension of $P^*$ to the spaces
\begin{equation}
P^*:W^{p'}_{-k,-\boldsymbol{\beta}}(F)\rightarrow
W^{p'}_{-k-n,-\boldsymbol{\beta}}(E).
\end{equation}
Under the identification of Sobolev spaces of negative order with dual spaces,
this operator coincides with the operator dual to that of Equation
\ref{eq:contextension},
\begin{equation}
P^*:(W^p_{k,\boldsymbol{\beta}}(F))^*\rightarrow
(W^p_{k+n,\boldsymbol{\beta}}(E))^*.
\end{equation}
Furthermore if $E=F$ and $P$ is self-adjoint, \textit{i.e.} $P=P^*$ on smooth
compactly-supported sections, then $P=P^*$ on any space
$W^p_{k,\boldsymbol{\beta}}$.
\end{lemma}
\begin{proof}
The formal adjoint of $P$ is asymptotic to the formal adjoint of
$P_\infty$, so the extensions exist as specified. The statement of this lemma
can be clarified by adopting the notation of Remark \ref{rem:Pbeta}: the claim
is then that $(P^*)_{-\boldsymbol{\beta}}=(P_{\boldsymbol{\beta}})^*$, where on
the left the superscript $*$ denotes the formal adjoint and on the right it
denotes the dual map.

Since both maps are continuous, it is sufficient to show that they coincide on a
dense subset: in particular that
$(P^*)_{-\boldsymbol{\beta}}(\tau)=(P_{\boldsymbol{\beta}})^*(\tau)$, for all
$\tau\in C^\infty_c(F)$. Since we are identifying
$(P^*)_{-\boldsymbol{\beta}}(\tau)$ with an element of the dual space
$(W^p_{k+n,\boldsymbol{\beta}}(E))^*$, we can again invoke continuity to claim
that it is sufficient to prove that, for all $e\in C^\infty_c(E)$,
\begin{equation}
\langle (P^*)_{-\boldsymbol{\beta}}(\tau), e\rangle=\langle
(P_{\boldsymbol{\beta}})^*(\tau),e\rangle.
\end{equation}
This claim is now a direct consequence of the definitions and of Equation
\ref{eq:duality}.

The claim concerning self-adjoint operators is a simple consequence of
continuity.
\end{proof}

\begin{remark} \label{rem:nestlingtheker}
As already remarked, $\boldsymbol{\beta'}>\boldsymbol{\beta}$ implies
$P_{\boldsymbol{\beta'}}$ extends $P_{\boldsymbol{\beta}}$. This shows that the
spaces $\mbox{Ker}(P_{\boldsymbol{\beta}})$ grow with $\boldsymbol{\beta}$. On
the other hand, as a vector space, the cokernel of $P$ in Equation
\ref{eq:contextension} is not canonically a subspace of
$W^p_{k,\boldsymbol{\beta}}(F)$ so there is no canonical way of relating
cokernels corresponding to different weights. However, consider the following
construction, for which we assume $P$, $P^*$ are Fredholm. Pick $
\tau_1\in W^p_{k,\boldsymbol{\beta}}(F)$ such that $\langle
\sigma,\tau_1\rangle\neq 0$, for some $\sigma\in \mbox{Ker}(P^*)$. According to
Remark \ref{rem:characterizations} this implies that $\tau_1$ does not belong to
$\mbox{Im}(P)$. By density we can then find $\tilde{\tau_1}$ which is smooth and
compactly-supported and does not belong to $\mbox{Im}(P)$. Now choose $\tau_2$
satisfying $\langle \sigma,\tau_2\rangle\neq 0$ for some $\sigma\in
\mbox{Ker}(P^*)$ and which is linearly independent of $\tau_1$, \textit{etc}.
After a finite number of steps we will have found a vector space spanned by
$\tilde{\tau_1},\dots,\tilde{\tau_k}$ which defines a complement to
$\mbox{Im}(P)$ and thus is isomorphic to $\mbox{Coker}(P)$. Notice that by
construction $\tilde{\tau_i}$ belong to all spaces
$W^p_{k,\boldsymbol{\beta}}(F)$. On the other hand, as $\boldsymbol{\beta}$
decreases the dual weight $-\boldsymbol{\beta}$ increases, so $\mbox{Ker}(P^*)$
increases, so the $\tilde{\tau_i}$ chosen for the weight $\boldsymbol{\beta}$
can be used also for any weight $\boldsymbol{\beta'}<\boldsymbol{\beta}$. The
conclusion is that we can construct spaces representing the cokernel which grow
as $\boldsymbol{\beta}$ decreases, \textit{i.e.} as the function spaces become
smaller.
\end{remark}

Now assume $P$ is elliptic. We are interested in conditions ensuring that the
extended map of Equation \ref{eq:contextension} is Fredholm.

\begin{definition}\label{def:exceptional} Let $\Sigma$ be a compact oriented
Riemannian manifold with connected components $\Sigma_1,\dots,\Sigma_e$. Let
$P_\infty$ be a translation-invariant operator on $\Sigma\times\R$.
Consider the complexified operator $P_\infty:E_\infty\otimes \C\rightarrow
F_\infty\otimes \C$. Choose a connected component $\Sigma_j\times\R$ and fix
$\gamma+i\delta\in\C$. Let us restrict our attention to the space of sections of
$E_\infty\otimes\C$ of the form $e^{(\gamma+i\delta)z}\sigma(\theta)$. Consider
the subspace $V^j_{\gamma+i\delta}$ determined by the solutions to the problem
$P_\infty(e^{(\gamma+i\delta)z}\sigma(\theta))=0$ on $\Sigma_j\times\R$. We
define the space $\mathcal{C}^j_{P_\infty}\subseteq\C$ to be the space of all
$\gamma+i\delta$ such that $V^j_{\gamma+i\delta}\neq 0$. We then define the
space of \textit{exceptional weights} for $P_\infty$ on $\Sigma_j\times\R$ to be
the corresponding set of real values,
$\mathcal{D}_{P_\infty}^j:=\mbox{Re}(\mathcal{C}^j_{P_\infty})\subseteq\R$. 

Now fix a multi-index $\boldsymbol{\gamma}+i\boldsymbol{\delta}\in \C^e$. Let
$V_{\boldsymbol{\gamma}+i\boldsymbol{\delta}}:=\oplus_{j=1}^e
V^j_{\gamma_j+i\delta_j}$. We define the space of \textit{exceptional weights}
for $P_\infty$ on $\Sigma\times\R$, denoted $\mathcal{D}_{P_\infty}\subseteq
\R^e$, to be the set of multi-indices $\gamma=(\gamma_1,\dots,\gamma_e)$ such
that, for some $j$, $\gamma_j\in\mathcal{D}_{P_\infty}^j$.
\end{definition}
\begin{remark}\label{rem:exceptional}
Definition \ref{def:exceptional} introduces the exceptional weights via the
kernel of $P_\infty$ and the space of sections with exponential growth. Along
the lines of \cite{lockhartmcowen}, the exceptional weights can equivalently be
defined as follows. Separating the $\partial\theta$ derivatives from the
$\partial z$ derivatives and setting $Dz=-i\partial z$, we can write 
\begin{equation}
P_\infty=\sum A_k(\theta,\partial\theta)(\partial z)^k=\sum
A_k(\theta,\partial\theta)i^k(Dz)^k,
\end{equation}
where, to simplify the notation, $\partial\theta$ denotes any combination of
derivatives in the $\theta$ variables.
For any $\lambda\in\C$, set $P_\lambda:=\sum
A_k(\theta,\partial\theta)i^k\lambda^k$. Notice that
\begin{equation}
P_\infty(e^{i\lambda z}\sigma(\theta))=\sum
A_k(\theta,\partial\theta)(i\lambda)^k\sigma e^{i\lambda z}=(P_\lambda (\sigma))
e^{i\lambda z}
\end{equation}
so $P_\infty(e^{i\lambda z}\sigma(\theta))=0$ iff $P_\lambda(\sigma)=0$. We view
the latter as a \textit{generalized eigenvalue problem} on $\Sigma$ and say that
$\lambda$ is an \textit{eigenvalue} iff the corresponding generalized eigenvalue
problem admits non-trivial solutions. It follows from the above calculations
that a weight $\gamma\in\R$ is exceptional in the sense of Definition
\ref{def:exceptional} iff $-\gamma=\mbox{Im}(\lambda)$, for some eigenvalue
$\lambda$.
\end{remark}

For elliptic operators it turns out that the exceptional weights of $P_\infty$
determine the possible Fredholm extensions of any $P$ asymptotic to $P_\infty$.
\begin{theorem}\label{thm:acyl_fredholm}
Let $(L,h)$ be an A.Cyl. manifold with link $\Sigma=\amalg \Sigma_i$. 
Let
$P:C^\infty(E)\rightarrow C^\infty(F)$ be a linear elliptic operator of order
$n$, asymptotic to an elliptic operator
$P_\infty$.

Then each $\mathcal{D}^j_{P_\infty}$ is discrete in $\R$ so
$\mathcal{D}_{P_\infty}$ defines a discrete set of hyperplanes in $\R^e$.
Furthermore, for each $p>1$ and $k\geq 0$, the extended operator
$P_{\boldsymbol{\gamma}}:W^p_{k+n,\boldsymbol{\gamma}}(E)\rightarrow
W^p_{k,\boldsymbol{\gamma}}(F)$ is Fredholm iff $\boldsymbol{\gamma}\notin
\mathcal{D}_{P_\infty}$.
\end{theorem}

In a similar vein, we can compute how the index of $P$ depends on
$\boldsymbol{\gamma}$. 

\begin{definition}\label{def:indexchange}
Consider the complexified operator $P_\infty:E_\infty\otimes \C\rightarrow
F_\infty\otimes \C$. Choose a connected component $\Sigma_j\times\R$ of
$\Sigma\times\R$ and fix $\gamma+i\delta\in\mathcal{C}^j_{P_\infty}$. We denote
by $\widetilde{V}^j_{\gamma+i\delta}$ the space of solutions to the problem
$P_\infty(e^{(\gamma+i\delta)z}\sigma(\theta,z))=0$ on $\Sigma_j\times\R$, where
$\sigma(\theta,z)$ is polynomial in $z$. We can extend this definition to all
$\gamma+i\delta$ by setting $\widetilde{V}^j_{\gamma+i\delta}=\{0\}$ if
$\gamma+i\delta\notin \mathcal{C}^j_{P_\infty}$. Notice that
$V^j_{\gamma+i\delta}\leq\widetilde{V}^j_{\gamma+i\delta}$. Given any
$\gamma\in\R$ we now set
$\widetilde{V}^j_\gamma:=\bigoplus_{\delta\in\R}\widetilde{V}^j_{\gamma+i\delta}
$, then define the \textit{multiplicity} of $\gamma$ on $\Sigma_j\times\R$ by
$m^j_{P_\infty}(\gamma):=\mbox{dim}(\widetilde{V}^j_\gamma)$.

Now fix a multi-index $\boldsymbol{\gamma}\in\R^e$. We define the
\textit{multiplicity} of $\boldsymbol{\gamma}$ on $\Sigma\times\R$ to be 
$m_{P_\infty}(\boldsymbol{\gamma}):=\sum_{j=1}^e m^j_{P_\infty}(\gamma_j)$. 
\end{definition}

\begin{theorem} \label{thm:acyl_indexchange}
In the setting of Theorem \ref{thm:acyl_fredholm}, each multiplicity
$m_{P_\infty}(\boldsymbol{\gamma})$ is finite. Furthermore, choose
$\boldsymbol{\gamma}_1,\boldsymbol{\gamma}_2\in \R^e\setminus
\mathcal{D}_{P_\infty}$ with $\boldsymbol{\gamma}_1\leq\boldsymbol{\gamma}_2$.
Then 
$$i_{\boldsymbol{\gamma}_2}(P)-i_{\boldsymbol{\gamma}_1}(P)=\sum_{\boldsymbol{
\gamma}\in\mathcal{D}_{P_\infty},
\boldsymbol{\gamma}_1\leq\boldsymbol{\gamma}\leq\boldsymbol{\gamma}_2}
m_{P_\infty}(\boldsymbol{\gamma}).$$
\end{theorem}

\begin{remark} Assume we can compute the value of
$i_{\boldsymbol{\gamma}}(P)$ for a specific good choice of non-exceptional
$\boldsymbol{\gamma}$. Theorem \ref{thm:acyl_indexchange} then allows us to
compute $i_{\boldsymbol{\gamma}}(P)$ for all non-exceptional
$\boldsymbol{\gamma}$ in terms of data on the link.
\end{remark}

The following result is proved in \cite{lockhartmcowen} Section 7, cf. also
\cite{joycesalur}, as a consequence of the Sobolev Embedding and
change of index theorems.

\begin{prop}\label{prop:acylindep}
In the setting of Theorem \ref{thm:acyl_indexchange}, assume
$\boldsymbol{\gamma}$ and $\boldsymbol{\gamma}'$ belong to the same connected
component of $\R^e\setminus\mathcal{D}_{P_\infty}$. Then
$i_{\boldsymbol{\gamma}}(P)=i_{\boldsymbol{\gamma}'}(P)$ and
$\mbox{Ker}(P_{\boldsymbol{\gamma}})=\mbox{Ker}(P_{\boldsymbol{\gamma}'})$.
Furthermore, the index and kernel are independent of the choice of $p$ and $k$.
\end{prop}

\begin{example} \label{e:acyl_harmonic}
Assume $(L,h)$ is an A.Cyl. manifold with one end with link $(\Sigma,g')$. Let
$P:=\Delta_h$ denote the positive Laplace operator on functions. Then $P$ is
asymptotic to the Laplace operator $\Delta_{\tilde{h}}$ defined on the product
$(\Sigma\times\R, \tilde{h}:=dz^2+g')$. One can check that
$\Delta_{\tilde{h}}=-(\partial z)^2+\Delta_{g'}$ and that
$\Delta_{\tilde{h}}e^{(\gamma+i\delta)z}\sigma(\theta)=0$ iff $\delta=0$ and
$\Delta_{g'}\sigma=\gamma^2\sigma$. In other words, the harmonic functions on
the cylinder which have exponential growth are generated by the eigenvalues of
$\Delta_{g'}$. In particular, the exceptional weights for $\Delta_h$ are of the
form $\pm\sqrt{e_n}$, where $e_n$ are the eigenvalues of $\Delta_{g'}$.
\end{example}


\section{Weight-crossing}\label{s:weightcrossing}

Let $(L,h)$ be an A.Cyl. manifold. Let $P:C^\infty(E)\rightarrow C^\infty(F)$ be
a linear elliptic operator asymptotic to some $P_\infty$ as in Definition
\ref{def:acyl_bundles}. Consider the extension of $P$ to weighted Sobolev spaces
as in Equation \ref{eq:contextension}. When $\boldsymbol{\beta}$ changes value
crossing an exceptional weight the change of index formula given in Theorem
\ref{thm:acyl_indexchange} leads us to expect that the kernel and/or cokernel of
$P$ will change. Specifically, when $\boldsymbol{\beta}$ increases we expect the
kernel of $P$ to increase and the cokernel to decrease. The process by which
this occurs can be formalized  using the Fredholm and index results stated in
Section \ref{s:acyl_analysis}. The notation we rely on was introduced in
Definitions \ref{def:exceptional} and \ref{def:indexchange}. To simplify the
notation, throughout this section we forgo the distinction between bundles (or
operators) and their complexifications.
 
Literally speaking, given any index $\gamma\in\R$ and end $S_j$, the sections in
each $\widetilde{V}^j_{\gamma}$ are defined on $\Sigma_j\times\R$. Using the
identification $\phi_j$, we can alternatively think of them as being defined on
$S_j$. However, we can also think of them as being globally defined on $L$ by
first choosing a basis of sections $\sigma_i^j$ for each
$\widetilde{V}^j_{\gamma}$, then interpolating between them so as to get smooth
extensions $\sigma_i^j$ over $L$. In particular it may be useful to choose the
extension of each $\sigma_i^j$ so that it is identically zero on the other
ends. The construction implies that each
$P_\infty(\sigma_i^j)$ has compact support. By choosing the extensions
generically over $L\setminus S$ we can assume that all $P(\sigma_i^j)$ are
linearly independent. This implies that $P$ is injective on
$\widetilde{V}_{\boldsymbol{\gamma}}$.

Now assume $\boldsymbol{\gamma}\in\R^e$ is exceptional. Then, for any
$\boldsymbol{\nu}<0$ with $|\boldsymbol{\nu}|<<1$, 
\begin{equation}
P:W^p_{k+n,\boldsymbol{\gamma}+\boldsymbol{\nu}}(E)\rightarrow
W^p_{k,\boldsymbol{\gamma}+\boldsymbol{\nu}}(F)
\end{equation}
is Fredholm. In particular, let $\boldsymbol{\nu}<0$ be the convergence rates of
$P$ as in Definition \ref{def:acyl_bundles}. We will assume that
$|\boldsymbol{\nu}|<<1$ as above. Writing
$P(\sigma)=(P-P_\infty)(\sigma)+P_\infty (\sigma)$ and using Equation
\ref{eq:differenceoperator} then shows that
$P(\widetilde{V}_{\boldsymbol{\gamma}})\subset
W^p_{k,\boldsymbol{\gamma}+\boldsymbol{\nu}}(F)$. 
Since $P$ is injective on $\widetilde{V}_{\boldsymbol{\gamma}}$ we can define a
decomposition
\begin{equation}
\widetilde{V}_{\boldsymbol{\gamma}}=\widetilde{V}_{\boldsymbol{\gamma}}'\oplus
\widetilde{V}_{\boldsymbol{\gamma}}''
\end{equation}
by defining
$P(\widetilde{V}_{\boldsymbol{\gamma}}'):=P(\widetilde{V}_{\boldsymbol{\gamma}}
)\cap\mbox{Im}(P_{\boldsymbol{\gamma}+\boldsymbol{\nu}})$ and choosing any
complement $\widetilde{V}_{\boldsymbol{\gamma}}''$. By definition,
$P(\widetilde{V}_{\boldsymbol{\gamma}}'')\cap\mbox{Im}(P_{\boldsymbol{\gamma}
+\boldsymbol{\nu}})=0$. In other words, we can think of
$P(\widetilde{V}_{\boldsymbol{\gamma}}'')$ as belonging to the cokernel of
$P_{\boldsymbol{\gamma}+\boldsymbol{\nu}}$. On the other hand,
$P(\widetilde{V}_{\boldsymbol{\gamma}}'')$ belongs to the image of
$P_{\boldsymbol{\gamma}-\boldsymbol{\nu}}$ because
$\widetilde{V}_{\boldsymbol{\gamma}}\subset
W^p_{k+n,\boldsymbol{\gamma}-\boldsymbol{\nu}}(E)$ . Roughly speaking,
$P(\widetilde{V}_{\boldsymbol{\gamma}}'')$ thus describes the portion of the
cokernel
of $P$ which ``disappears" when crossing the exceptional weight
$\boldsymbol{\gamma}$.

By construction, for any $\sigma\in \widetilde{V}_{\boldsymbol{\gamma}}'$ there
exists $u_\sigma\in W^p_{k+n,\boldsymbol{\gamma}+\boldsymbol{\nu}}(E)$ such that
$P(\sigma)=P (u_\sigma)$. Notice that $u_\sigma$ is not necessarily uniquely
defined. However it is sufficient to fix a choice of $u_{\sigma}$ for each
element of a basis of $\widetilde{V}_{\boldsymbol{\gamma}}'$ to obtain a unique
choice of $u_\sigma$ for any $\sigma\in \widetilde{V}_{\boldsymbol{\gamma}}'$.
Notice also that $\sigma-u_\sigma\in
W^p_{k+n,\boldsymbol{\gamma}-\boldsymbol{\nu}}(E)$. We have thus defined a map
\begin{equation}\label{eq:newkernel}
\widetilde{V}_{\boldsymbol{\gamma}}'\rightarrow\mbox{Ker}(P_{\boldsymbol{\gamma}
-\boldsymbol{\nu}}), \ \
\sigma\mapsto\sigma-u_\sigma\notin
W^p_{k+n,\boldsymbol{\gamma}+\boldsymbol{\nu}}(E).
\end{equation}
The image of the map of Equation \ref{eq:newkernel} thus defines a space of
``new" elements in $\mbox{Ker}(P)$, generated by crossing the exceptional weight
$\boldsymbol{\gamma}$. Notice that $u_\sigma$ is of strictly lower order of
growth compared to $\sigma$. This shows that the map of Equation
\ref{eq:newkernel} is injective and that the elements in its image admit an
asymptotic expansion of the form $e^{\boldsymbol{\gamma}\zeta}+\mbox{lower
order}$. The following result shows that every new element in $\mbox{Ker}(P)$
arises this way.

\begin{lemma} 
Let us identify $\widetilde{V}_{\boldsymbol{\gamma}}'$ with its image under the
map of Equation \ref{eq:newkernel}. Then
$$\mbox{Ker}(P_{\boldsymbol{\gamma}-\boldsymbol{\nu}})=\mbox{Ker}(P_{\boldsymbol
{\gamma}+\boldsymbol{\nu}})\oplus \widetilde{V}_{\boldsymbol{\gamma}}'.$$
\end{lemma}
\proof{}By injectivity, the inequality $\supseteq$ is clear. To prove the lemma
it is thus sufficient to prove that the inverse inequality holds on the
corresponding dimensions. Choose any $\sigma\in
\widetilde{V}_{\boldsymbol{\gamma}}''$. According to Remark
\ref{rem:characterizations},
\begin{eqnarray*}
P(\sigma)\in \mbox{Im}(P_{\boldsymbol{\gamma}-\boldsymbol{\nu}})\Leftrightarrow
\langle\tau,P(\sigma)\rangle=0,\ \forall \tau\in
\mbox{Ker}(P^*_{-\boldsymbol{\gamma}+\boldsymbol{\nu}}),\\
P(\sigma)\in \mbox{Im}(P_{\boldsymbol{\gamma}+\boldsymbol{\nu}})\Leftrightarrow
\langle\tau,P(\sigma)\rangle=0,\ \forall \tau\in
\mbox{Ker}(P^*_{-\boldsymbol{\gamma}-\boldsymbol{\nu}}).
\end{eqnarray*}
From the definition of $\widetilde{V}_{\boldsymbol{\gamma}}''$ we know that
$P(\sigma)\in \mbox{Im}(P_{\boldsymbol{\gamma}-\boldsymbol{\nu}})$ and that
$P(\sigma)\notin\mbox{Im}(P_{\boldsymbol{\gamma}+\boldsymbol{\nu}})$ unless
$\sigma=0$. Notice also that
$\mbox{Ker}(P^*_{-\boldsymbol{\gamma}+\boldsymbol{\nu}})\subseteq\mbox{Ker}(P^*_
{-\boldsymbol{\gamma}-\boldsymbol{\nu}})$. We conclude that the following map is
well-defined:
\begin{equation}
\frac{\mbox{Ker}(P^*_{-\boldsymbol{\gamma}-\boldsymbol{\nu}})}{\mbox{Ker}(P^*_{
-\boldsymbol{\gamma}+\boldsymbol{\nu}})}\times
\widetilde{V}_{\boldsymbol{\gamma}}'', \ \ ([\tau],\sigma)\mapsto
\langle\tau,P(\sigma)\rangle,
\end{equation}
and that the corresponding map
\begin{equation}
\widetilde{V}_{\boldsymbol{\gamma}}''\rightarrow
\left(\frac{\mbox{Ker}(P^*_{-\boldsymbol{\gamma}-\boldsymbol{\nu}})}{\mbox{Ker}
(P^*_{-\boldsymbol{\gamma}+\boldsymbol{\nu}})}\right)^*
\end{equation}
is injective. This proves that 
\begin{equation}\label{eq:dimcalc}
\mbox{dim}(\widetilde{V}_{\boldsymbol{\gamma}}'')\leq
\mbox{dim}(\mbox{Ker}(P^*_{-\boldsymbol{\gamma}-\boldsymbol{\nu}}))-\mbox{dim}
(\mbox{Ker}(P^*_{-\boldsymbol{\gamma}+\boldsymbol{\nu}})). 
\end{equation}
On the other hand, the change of index formula shows that
\begin{align}\label{eq:indexcalc}
\mbox{dim}(\widetilde{V}_{\boldsymbol{\gamma}}')+\mbox{dim}(\widetilde{V}_{
\boldsymbol{\gamma}}'')&=\mbox{dim}(\mbox{Ker}(P_{\boldsymbol{\gamma}
-\boldsymbol{\nu}}))-\mbox{dim}(\mbox{Ker}(P^*_{-\boldsymbol{\gamma}+\boldsymbol
{\nu}}))\\
&\ \
-\mbox{dim}(\mbox{Ker}(P_{\boldsymbol{\gamma}+\boldsymbol{\nu}}))+\mbox{dim}
(\mbox{Ker}(P^*_{-\boldsymbol{\gamma}-\boldsymbol{\nu}})).\nonumber
\end{align}
Subtracting Equation \ref{eq:dimcalc} from Equation \ref{eq:indexcalc} proves
the desired inequality. 
\endproof


\section{Fredholm results for elliptic operators on
conifolds}\label{s:accs_analysis}

We now want to see how to achieve analogous results for certain elliptic
operators on conifolds. In parallel with Section \ref{s:acyl_analysis} it
is possible to develop an abstract definition and theory of
\textit{asymptotically conical} operators, analogous to that of asymptotically
translation-invariant operators on A.Cyl. manifolds. For simplicity, however, we
will limit ourselves to the special case of the Laplace
operator acting on functions. This already 
contains the main ideas of the general theory.

Let $(L,g)$ be a conifold. Consider the weighted spaces introduced in
Definition \ref{def:csac_sectionspaces}. As in Section \ref{s:acyl_analysis} we
denote the dual space $(W^p_{k,\boldsymbol{\beta}})^*$ by
$W^{p'}_{-k,-\boldsymbol{\beta}-m}$. This choice of weights is compatible with
the identifications of Remark \ref{rem:spacescoincide}, and the properties of
these dual spaces are analogous to those seen in Section \ref{s:acyl_analysis}.
It follows directly from the definitions that
\begin{equation*}
\nabla:W^p_{k,\boldsymbol{\beta}}\rightarrow W^p_{k-1,\boldsymbol{\beta}-1}
\end{equation*}
is a continuous operator. Equation \ref{eq:laplace} then implies that $\Delta_g$
extends to a continuous map 
\begin{equation}\label{eq:extendedlaplacian}
\Delta_{\boldsymbol{\beta}}:W^p_{k,\boldsymbol{\beta}}\rightarrow
W^p_{k-2,\boldsymbol{\beta}-2}.
\end{equation}
The following result is closely related to Lemma \ref{l:*=*} and uses the fact
that $\Delta_g$ is formally self-adjoint.

\begin{lemma} \label{l:intbyparts}
Let $(L,g)$ be a conifold. Choose $u\in W^p_{k,\boldsymbol{\beta}}$, $v\in
W^{p'}_{2-k, 2-\boldsymbol{\beta}-m}$. Then
\begin{equation}\label{eq:intbyparts}
\langle v,\Delta_g u\rangle=\langle dv,du\rangle=\langle \Delta_g
v,u\rangle.\end{equation}
\end{lemma}
\begin{proof}
Using the appropriate dualities, each expression in Equation \ref{eq:intbyparts}
defines by composition a continuous bilinear map $(u,v)\in W^p_{k,\beta}\times
W^{p'}_{2-k, 2-\beta-m}\rightarrow\R$. Since $\Delta_g=d^*d$ the equalities hold
on the dense subsets $C^\infty_c\times C^\infty_c$. By continuity the equalities
thus continue to hold on the full Sobolev spaces. 
\end{proof}

We now want to investigate the Fredholm properties of
$\Delta_{\boldsymbol{\beta}}$. It is initially useful to distinguish between the
AC and CS case. To begin, let $(L,g)$ be an AC manifold with ends $S_j$ and
links $\Sigma_j$. The starting point for the Fredholm theory is then the
following observation. 

\begin{lemma} \label{l:conformalac=acyl}
Let $(\Sigma,g')$ be a Riemannian manifold. Let the corresponding cone
$C:=\Sigma\times (0,\infty)$ have the conical metric $\tg:=dr^2+r^2g'$. Let
$\Delta_{\tg}$ denote the corresponding Laplace operator on functions. Then,
under the substitution $r=e^z$, the operator $r^2\Delta_{\tg}$ coincides with
the translation-invariant operator 
\begin{equation}\label{eq:rescaledlaplacian}
P_\infty:=-(\partial z)^2+(2-m)\partial z+\Delta_\Sigma
\end{equation}
on the cylinder $\Sigma\times\R$.
\end{lemma}
\proof{}
Recall that in any local coordinate system the Laplace operator on functions is
given by the formula 
\begin{equation}\label{eq:locallaplacian}
\Delta_g=-\frac{1}{\sqrt{g}}\partial_j(\sqrt{g}g^{ij}\partial_i).
\end{equation}
Let $U$ be a local chart on $\Sigma$ so that $U\times (0,\infty)$ is a local
chart on $C$. Equation \ref{eq:locallaplacian} then shows that 
\begin{equation}
\Delta_{\tg}=-(\partial r)^2-\frac{m-1}{r}\partial r+r^{-2}\Delta_\Sigma.
\end{equation}
The substitution $r=e^z$ implies $r\partial r=\partial z$. The claim is then a
simple calculation.
\endproof

Lemma \ref{l:conformalac=acyl} allows us to study the Fredholm properties of
$\Delta_g$ by building an equivalent problem for an A.Cyl. manifold, as follows.
We use the notation of Section \ref{s:acyl_analysis}.

Multiplication by $\rho^2$ defines an isometry
$W^p_{k-2,\boldsymbol{\beta}-2}\simeq W^p_{k-2,\boldsymbol{\beta}}$. Thus
$\Delta_{\boldsymbol{\beta}}$ in Equation \ref{eq:extendedlaplacian} is Fredholm
iff the operator 
\begin{equation}\label{eq:extendedrescaledlaplacian}
\rho^2\Delta_{\boldsymbol{\beta}}:W^p_{k,\boldsymbol{\beta}}\rightarrow
W^p_{k-2,\boldsymbol{\beta}}
\end{equation}
is Fredholm. Now consider the A.Cyl. manifold $(L,h)$, where $h=\rho^{-2}g$. It
follows from Equation \ref{eq:laplace} and Lemma \ref{l:conformalac=acyl} that
the operator $P:=\rho^2\Delta_g$ is asymptotic in the sense of Definition
\ref{def:acyl_bundles} to the translation-invariant operator $P_\infty$ of
Equation \ref{eq:rescaledlaplacian}. One can check that the convergence rate
$\boldsymbol{\nu}$ of $P$ coincides with the convergence rate $\boldsymbol{\nu}$
of the AC manifold, cf. Definition \ref{def:metrics_ends}. 

It is simple to verify that the equation
$P_\infty(e^{(\gamma+i\delta)z}\sigma(\theta))=0$ is equivalent to the following
eigenvalue problem on the link:
\begin{equation}\label{eq:ac_harmonicbis}
\Delta_{\Sigma_j}\sigma=[(\gamma+i\delta)^2+(m-2)(\gamma+i\delta)]\sigma.
\end{equation}
Using the fact that the eigenvalues $e_n^j$ of $\Delta_{\Sigma_j}$ are real and
non-negative, it follows that $\delta=0$ and that $\gamma$ satisfies
$\gamma^2+(m-2)\gamma=e_n^j$ for some $n$, \textit{i.e.} 
\begin{equation}\label{eq:exceptionalforlaplacian}
\gamma=\frac{(2-m)\pm\sqrt{(2-m)^2+4e_n^j}}{2}.
\end{equation}

This shows that, for this particular operator,
$\mathcal{C}^j_{P_\infty}=\mathcal{D}^j_{P_\infty}$. It also follows from Lemma
\ref{l:conformalac=acyl} that the equation $P_\infty(e^{\gamma
z}\sigma(\theta))=0$ is equivalent to $\Delta_{\tg}(r^{\gamma}\sigma)=0$. Thus
\begin{equation}\label{eq:ac_harmonicter}
V^j_{\gamma}=\{r^\gamma\sigma(\theta): \Delta_{\tilde{g}}(r^{\gamma}\sigma)=0\},
\end{equation}
\textit{i.e.} $V^j_{\gamma}$ coincides with the space of homogeneous harmonic
functions of degree $\gamma$ on the cone $\Sigma_j\times (0,\infty)$. 

Varying the choice of eigenvalue $e_n^j$ gives the set of exceptional weights
for $P_\infty$ on the end $S_j$. Repeating this for each end defines the set
$\mathcal{D}_{P_\infty}\subset\R^e$. According to Theorem
\ref{thm:acyl_fredholm} these are the weights for which the operator $P$ is not
Fredholm with respect to the Sobolev spaces of $(L,h)$. However, recall from
Remark \ref{rem:spacescoincide} that the Sobolev spaces of $(L,g)$ and $(L,h)$
coincide. Thus $\mathcal{D}_{P_\infty}\subset\R^e$ are also the weights for
which the operators of Equations \ref{eq:extendedrescaledlaplacian},
\ref{eq:extendedlaplacian} are not Fredholm. 

\begin{remark}
Notice that in this particular case (and in the analogous case presented in
Example \ref{e:acyl_harmonic}) the generalized eigenvalue problem introduced in
Remark \ref{rem:exceptional} has reduced to an eigenvalue problem in the usual
sense.
\end{remark}

It is also fairly straight-forward to verify that, for this operator $P_\infty$,
the spaces $\widetilde{V}^j_{\gamma+i\delta}$ and $V^j_{\gamma+i\delta}$
coincide, cf. Joyce \cite{joyce:I} Proposition 2.4 for details. This allows us
to simplify the definition of the multiplicity $m(\boldsymbol{\gamma})$.

\ 

The situation for CS manifolds is similar. The change of variables $r=e^{-z}$
introduces a change of sign in Equation \ref{eq:rescaledlaplacian}. This sign is
later cancelled by a change of sign in the identification of Sobolev spaces of
$(L,g)$ and $(L,h)$. The final result is thus identical to the AC case.
Combining these results leads to the following conclusion.

\begin{corollary}\label{cor:laplaceresults}
Let $(L,g)$ be a conifold with $e$ ends. For each end $S_j$ with
link $\Sigma_j$ let $e^j_n$ denote the eigenvalues of the positive Laplace
operator $\Delta_{\Sigma_j}$ and define the set of ``exceptional weights"
$\mathcal{D}^j:=\{\gamma_j\}\subseteq\R$ as in Equation
\ref{eq:exceptionalforlaplacian}. Given any weight $\gamma\in\R$ define
$V^j_\gamma$ as in Equation \ref{eq:ac_harmonicter} and let $m^j(\gamma)$ denote
its dimension. Given any weight $\boldsymbol{\gamma}\in \R^e$ set
$m(\boldsymbol{\gamma}):=\sum_{j=1}^e m^j(\gamma_j)$. Let
$\mathcal{D}\subseteq\R^e$ denote the set of weights $\boldsymbol{\gamma}$ for
which $m(\boldsymbol{\gamma})>0$. Then each multiplicity
$m(\boldsymbol{\gamma})$ is finite and the Laplace operator 
\begin{equation}
\Delta_g:W^p_{k,\boldsymbol{\beta}}\rightarrow W^p_{k-2,\boldsymbol{\beta}-2}
\end{equation}
is Fredholm iff $\boldsymbol{\beta}\notin \mathcal{D}$.

The analogue of Theorem \ref{thm:acyl_indexchange} also holds. For example,
assume $L$ is a CS/AC manifold and write
$\boldsymbol{\beta}=(\boldsymbol{\mu},\boldsymbol{\lambda})$. Choose
$(\boldsymbol{\mu}_1,\boldsymbol{\lambda}_1),
(\boldsymbol{\mu}_2,\boldsymbol{\lambda}_2) \in \R^e\setminus \mathcal{D}$ with
$\boldsymbol{\mu}_1\geq\boldsymbol{\mu}_2$,
$\boldsymbol{\lambda}_1\leq\boldsymbol{\lambda}_2$. Then 
$$i_{\boldsymbol{\mu}_2,\boldsymbol{\lambda}_2}(\Delta_g)-i_{\boldsymbol{\mu}_1,
\boldsymbol{\lambda}_1}(\Delta_g)=\sum
m(\boldsymbol{\mu},\boldsymbol{\lambda}),$$
where the sum is taken over all
$(\boldsymbol{\mu},\boldsymbol{\lambda})\in\mathcal{D}$ such that
$\boldsymbol{\mu}_1\geq\boldsymbol{\mu}\geq\boldsymbol{\mu}_2$,
$\boldsymbol{\lambda}_1\leq\boldsymbol{\lambda}\leq\boldsymbol{\lambda}_2$.
\end{corollary}

In the same way one can also prove the analogue of Proposition
\ref{prop:acylindep}.


\section{Application: harmonic functions on conifolds}\label{s:accs_harmonic}
We can use the results of Sections \ref{s:weightcrossing} and
\ref{s:accs_analysis} to reach a good understanding of the properties of the
Laplace operator acting on functions on conifolds. Specifically, we will
be interested in the kernel and cokernel of $\Delta_g$. 
\subsubsection*{Smooth compact manifolds} Let $(L,g)$ be a smooth compact
Riemannian manifold. Let $\Delta_g$ denote the positive Laplace operator on
functions. Consider the map
\begin{equation}\label{eq:cptextendedlaplacian}
\Delta_g:W^p_k(L)\rightarrow W^p_{k-2}(L).
\end{equation}
For all $p>1$ and $k\in\Z$, standard elliptic regularity shows that any
$f\in\mbox{Ker}(\Delta_g)$ is smooth. The maximum principle then proves that $f$
is constant. Thus $\mbox{Ker}(\Delta_g)=\R$, independently of the choice of
$p,k$.
 
As seen in Section \ref{s:prelim}, $f\in\mbox{Im}(\Delta_g)$ iff $<u,f>=0$, for
all $u\in \mbox{Ker}(\Delta_g^*)$, where $\Delta_g^*$ is the operator dual to that of
Equation \ref{eq:cptextendedlaplacian}. As in Lemma \ref{l:*=*} we can identify
this with the formal adjoint operator. However, $\Delta_g$ is formally
self-adjoint, \textit{i.e.} the operators $\Delta_g$ and $\Delta_g^*$ coincide
on smooth functions. By continuity they continue to coincide when extended to
any Sobolev space. Thus $\mbox{Ker}(\Delta_g^*)=\mbox{Ker}(\Delta_g)=\R$. As in
Equation \ref{eq:duality} we find $<u,f>=\int_Luf\,\mbox{vol}_g$. It follows
that $\mbox{Im}(\Delta_g)=\{f\in W^p_{k-2}(L):\int_Lf\,\mbox{vol}_g=0\}$. In
particular, $\Delta_g$ has index zero.
\subsubsection*{AC manifolds} Let $(L,g)$ be a AC manifold with convergence
rate $\boldsymbol{\nu}<0$ as in Definition \ref{def:metrics_ends}. Let
$\Delta_g$ denote the positive Laplace operator on weighted Sobolev spaces of
functions, as in Equation \ref{eq:extendedlaplacian}. For simplicity, we will
restrict our attention to the case of $L$ with 2 ends. 

Each end defines exceptional weights, plotted as points on the horizontal and
vertical axes of Figure \ref{fig:1}. Each exceptional weight gives rise to an
exceptional hyperplane, plotted as a vertical or horizontal line. The Laplacian
is Fredholm for weights $\boldsymbol{\beta}=(\beta_1,\beta_2)$ which are
non-exceptional, \textit{i.e.} which do not lie on these lines. The arrow
indicates the direction in which the corresponding Sobolev spaces, thus the
kernel of $\Delta_g$, become bigger.

Choose $\boldsymbol{\beta}$ non-exceptional. For all $p>1$ and $k\in\Z$,
standard elliptic regularity proves that any $f\in\mbox{Ker}(\Delta_g)$ is
smooth. Furthermore, since $\mbox{Ker}(\Delta_g)$ is independent of $p$ and $k$,
the Sobolev Embedding Theorems show that $f$ has growth of the order
$O(r^{\boldsymbol{\beta}})$. If $\boldsymbol{\beta}<0$ we can thus apply the
maximum principle to conclude that $f\equiv 0$. In other words, $\Delta_g$ is
injective throughout the quadrant defined by the lower shaded region. Since
$\Delta_g$ is formally self-adjoint, the same holds for $\Delta_g^*$. Recall
from Section \ref{s:accs_analysis} how weights on AC manifolds change under
duality. We conclude, following Section \ref{s:prelim}, that
$\mbox{Coker}(\Delta_g)=0$ for $\boldsymbol{\beta}>2-m$. In other words,
$\Delta_g$ is surjective throughout the quadrant defined by the upper shaded
region. In particular, the map of Equation \ref{eq:extendedlaplacian} is an
isomorphism and has index zero for $2-m<\boldsymbol{\beta}<0$, \textit{i.e.} in
the region marked by A.

When $\boldsymbol{\beta}>2-m$ the cokernel is independent of the weight. Thus,
any
change of index corresponds entirely to a change of kernel. Furthermore,
$\mbox{Ker}(\Delta_g)=\mbox{Ker}(\rho^2\Delta_g)$. We can thus use the results
of Section \ref{s:weightcrossing} 
to study how the kernel changes as $\boldsymbol{\beta}$
increases. For example, assume we are interested in harmonic functions for some
(thus any) $\boldsymbol{\beta}$ in the region B. We can reach this region by
keeping $\beta_2$ fixed and repeatedly increasing $\beta_1$, starting from the
region A. Each time we cross an exceptional line $x=\gamma$, new harmonic
functions on $(L,g)$ are generated by elements $r^\gamma\sigma(\theta)\in
V^1_\gamma$. Specifically, these new harmonic functions will be asymptotic to
$r^\gamma\sigma$ on the first end and to zero on the second end. Using the ideas
of Section \ref{s:weightcrossing} we can further show that the lower-order terms
will have rate $O(r^{\gamma+\nu_1})$ on the first end and $O(r^{\nu_2})$ on the
second. Analogous results hold for harmonic functions for $\boldsymbol{\beta}$
in the region C. The construction shows that the harmonic functions in the
regions B and C are linearly independent. We can thus apply the change of index
formula to show that harmonic functions in the generic region D are generated by
linear combinations of harmonic functions in the regions B, C. 

It may be good to emphasize that the above constructions depend on the specific
$(L,g)$ only in terms of the specific exceptional weights, but are otherwise
completely independent of $(L,g)$. However, these constructions fail if D is
chosen outside the region where $\Delta_g$ is surjective. 

\begin{figure}
\includegraphics[width=90mm,height=90mm]{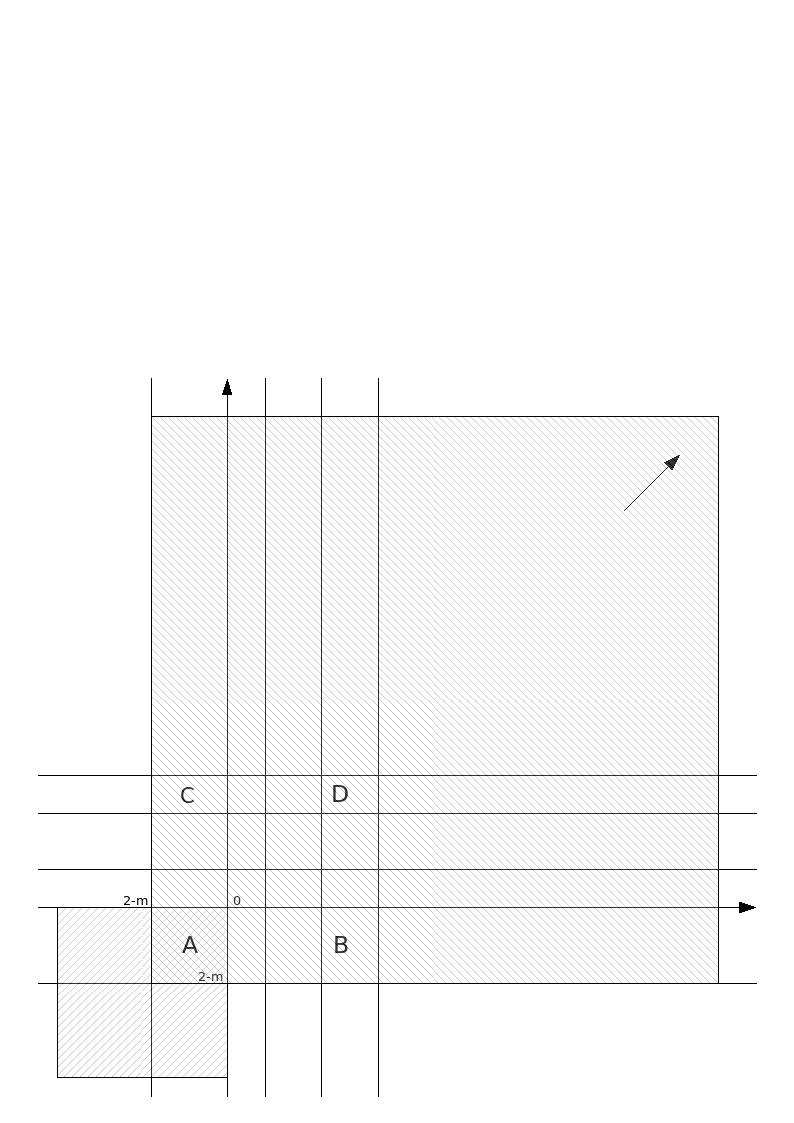}
\caption{Harmonic functions on AC manifolds}\label{fig:1}
\end{figure}
\subsubsection*{CS manifolds} Let $(L,g)$ be a CS manifold with convergence
rate $\boldsymbol{\nu}>0$ as in Definition \ref{def:metrics_ends}. As before,
let $\Delta_g$ denote the positive Laplace operator on weighted Sobolev spaces
of functions, as in Equation \ref{eq:extendedlaplacian}. We again restrict our
attention to the case of $L$ with 2 ends.

Figure \ref{fig:2} plots the exceptional weights and lines in this case. Once
again the arrow indicates the direction in which the corresponding Sobolev
spaces, thus the kernel of $\Delta_g$, become bigger. Choose
$\boldsymbol{\beta}$ non-exceptional. As before, any $f\in \mbox{Ker}(\Delta_g)$
is smooth with growth of order $O(r^{\boldsymbol{\beta}})$. If
$\boldsymbol{\beta}>0$ the maximum principle shows that $f=0$. Now assume
$\boldsymbol{\beta}=\frac{2-m}{2}$. In this case
$(W^2_{k-2,\boldsymbol{\beta}-2})^*=W^2_{2-k,\boldsymbol{\beta}}$. Choose $f\in
W^2_{k,\boldsymbol{\beta}}$ and assume $\Delta_gf=0$. Then, choosing $u=v=f$ in
Lemma \ref{l:intbyparts} and using regularity, we can conclude $df=0$ so $f$ is
constant. This shows that, for any weight in the region A,
$\mbox{Ker}(\Delta_g)=\R$.
As before we also find that, in this region, $\mbox{Im}(\Delta_g)=\{f\in
W^p_{k-2,{\beta}-2}:\int_Lf\,\mbox{vol}_g=0\}$. In particular, the index of
$\Delta_g$ is zero. 

Now assume $(\beta_1,\beta_2)>(0,\frac{2-m}{2})$. Then
$W^p_{k,\boldsymbol{\beta}}\subset W^p_{k,(\frac{2-m}{2},\frac{2-m}{2})}$ so our
integration by parts argument remains valid. On the other hand the only constant
function in $W^p_{k,\boldsymbol{\beta}}$ is zero so in this case we find that
$\Delta_g$ is injective. The same holds for
$(\beta_1,\beta_2)>(\frac{2-m}{2},0)$. Thus $\Delta_g$ is injective in the upper
shaded region. By duality we deduce that $\Delta_g$ is surjective in the lower
shaded region.

Now assume $\boldsymbol{\beta}$ crosses from A to B. In this particular case the
method used above for AC manifolds fails, because it would require $\Delta_g$ to
be surjective in the region A. We can however bypass this problem as follows:
the change of index formula shows that the index increases by one and we know
that the Laplacian is surjective in B, so $\mbox{Ker}(\Delta_g)=\R$ in B. The
same is true for the region C. We can use Section \ref{s:weightcrossing} to
study the harmonic functions in the lower shaded region. For example, the
harmonic functions in D will be generated by functions which are of the form
$r^\gamma\sigma+O(r^{\gamma+\nu_1})$ on the first end and of the form
$O(r^{\nu_2})$ on the second end. Notice a difference with respect to AC
manifolds: harmonic functions in B and C (more generally, in D and E) are not
necessarily linearly independent. Thus we cannot write harmonic functions in F
as the direct sum of harmonic functions in D and E, as in the AC case. Once
again, harmonic functions elsewhere will be heavily dependent on the specific
$(L,g)$.

We may also be interested in the cokernel of $\Delta_g$. The change of index
formula shows that the dimension of the cokernel increases with $\beta$. For
example, the index is -1 in the regions G,H. Since $\Delta_g$ is injective here
this implies that the cokernel has dimension 1. More generally, the change of
index formula allows us to compute the dimension of the cokernel wherever
$\Delta_g$ is injective. We can also use the ideas of Remark
\ref{rem:nestlingtheker} to build complements of $\mbox{Im}(\Delta_g)$ which
grow with $\boldsymbol{\beta}$.

\begin{figure}
\includegraphics[width=90mm,height=90mm]{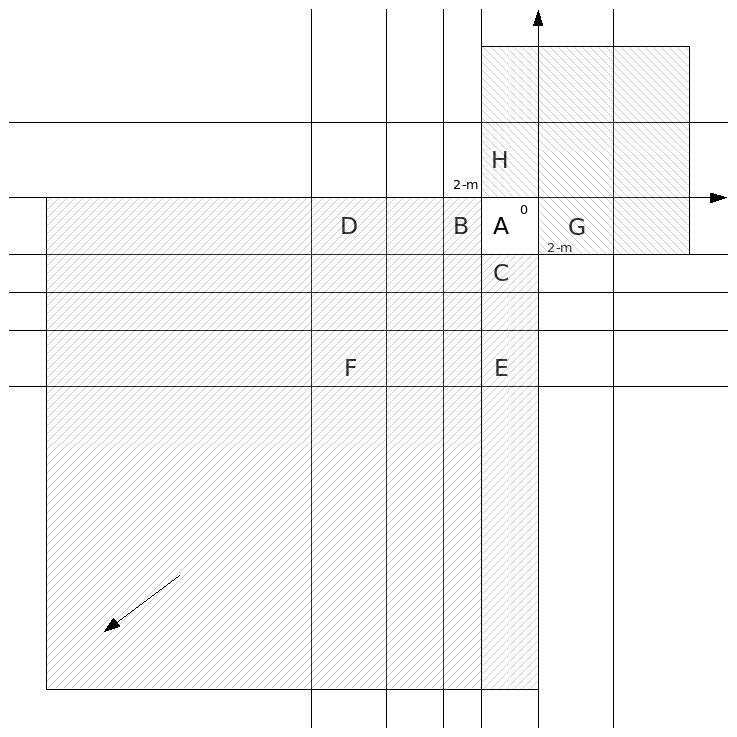}
\caption{Harmonic functions on CS manifolds}\label{fig:2}
\end{figure} 
\subsubsection*{CS/AC manifolds} Let $(L,g)$ be a CS/AC manifold with
convergence rate $\boldsymbol{\nu}$. Following the same conventions as before,
we now turn to Figure \ref{fig:3}. Here, the horizontal axis corresponds to the
CS end with weight $\mu$ and the vertical axis corresponds to the AC end with
weight $\lambda$. 

When $\lambda<0$ and $\mu>2-m$, the maximum principle and integration by parts
show that $\Delta_g$ is injective. Dually, when $\lambda>2-m$ and $\mu<0$,
$\Delta_g$ is surjective. In the region A, $\Delta_g$ is an isomorphism with
index zero. Harmonic functions in the region B are of the form
$r^\gamma\sigma+O(r^{\gamma+\nu_2})$ on the AC end and of the form
$O(r^{\nu_1})$ on the CS end. Harmonic functions in the region C are of the form
$r^\gamma\sigma+O(r^{\gamma+\nu_1})$ on the CS end and of the form
$O(r^{\nu_2})$ on the AC end. Since these functions are linearly independent,
their linear combinations give the harmonic functions in the region D.

\begin{figure}
\includegraphics[width=90mm,height=90mm]{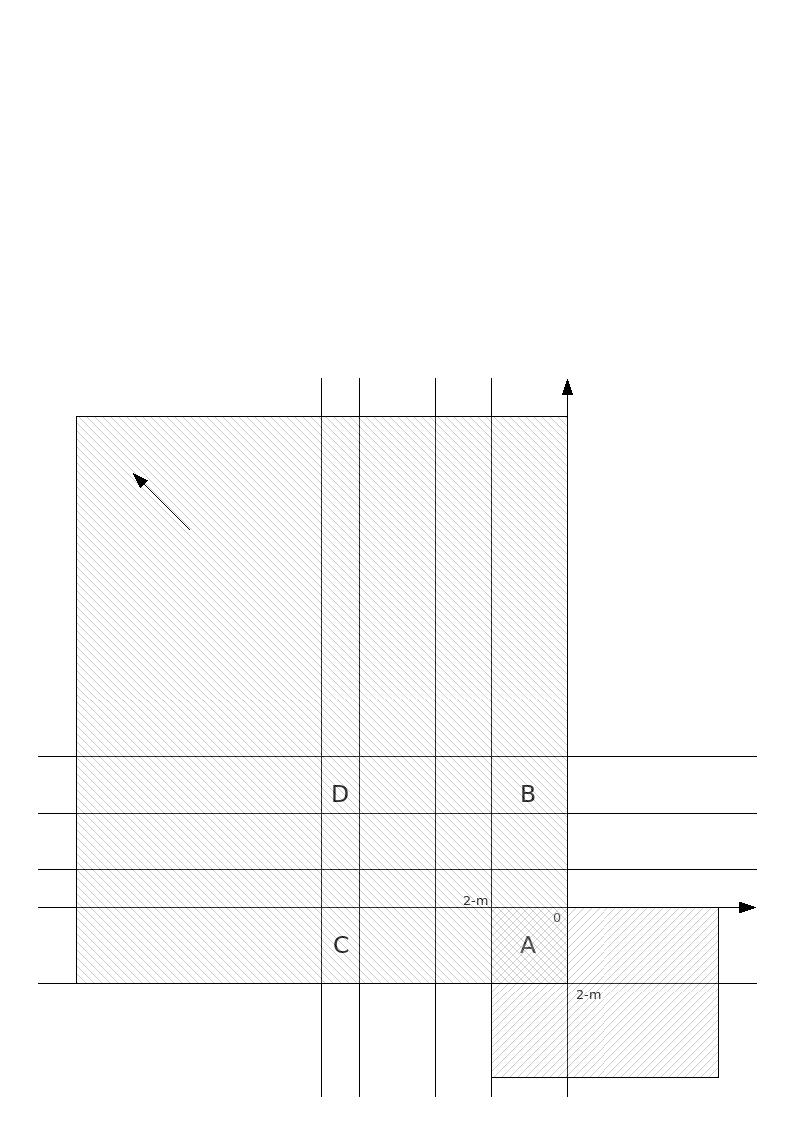}
\caption{Harmonic functions on CS/AC manifolds}\label{fig:3}
\end{figure}

\begin{example} 
$\R^m$ with its standard metric can be viewed as a CS/AC manifold, the CS end
being a neighbourhood of the origin. In this case all harmonic functions can be
written explicitly, so in this case we have exact information on their
asymptotics.
\end{example}


\part{Conifold connect sums and uniform estimates}

Ths is the main part of the paper. Our goal is to introduce a certain ``parametric connect sum''
construction between conifolds; as mentioned in the Introduction, this is the
abstract analogue of
certain desingularization procedures used in Differential Geometry, in which an
isolated conical singularity is
replaced by something smooth or perhaps by a new collection of AC or CS ends. We will show that careful choices of parameters and weights lead to uniform estimates concerning both Sobolev Embedding Theorems and the Laplace operator. These estimates are at the heart of the paper \cite{pacini:slgluing}. Readers interested in specific applications of these estimates can thus refer there for details.  

\section{Conifold connect sums}\label{s:sums_sobolev}

The goal of this section is to define the ``parametric connect sum''
construction and prove that the scaled and weighted Sobolev
constants
are independent of the parameter $\boldsymbol{t}$. For simplicity we start with
the
non-parametric version.

\begin{definition}\label{def:marking}
Let $(L,g)$ be a conifold, not necessarily connected. Let $S$ denote
the union of its ends. A subset $S^*$ of $S$ defines a \textit{marking} on $L$.
We can then write $S=S^*\amalg S^{**}$, where $S^{**}$ is simply the complement
of $S^*$. We say $S^*$ is a \textit{CS-marking} if all ends in $S^*$ are CS; it
is an \textit{AC-marking} if all ends in $S^*$ are AC. We will denote by $d$
the number of ends in $S^*$.

If $L$ is weighted via $\boldsymbol{\beta}$ we require that
$\beta_i=\beta_j$ if $S_i$ and $S_j$ are marked ends belonging to the same
connected component of $L$.
\end{definition}

\begin{definition}\label{def:compatible}
 Let $(L,g,S^*)$ be a CS-marked conifold. Let
$\Sigma^*$, $C^*$ denote the
links and cones corresponding to $S^*$, as in Definition
\ref{def:metrics_ends}. Given any end $S_i\subseteq S^*$ let
$\phi_i:\Sigma_i\times (0,\epsilon]\rightarrow \overline{S_i}$ be the
diffeomorphism of
Definition \ref{def:metrics_ends}. 

Let $(\hat{L},\hat{g},\hat{S}^*)$ be an AC-marked
conifold. Let $\hat{\Sigma}^*$, $\hat{C}^*$,
$\hat{\phi}_i:\hat{\Sigma}_i\times
[\hat{R},\infty)\rightarrow \overline{\hat{S}_i}$ denote the corresponding
links, cones and diffeomorphisms, as above. 

We say that $L$ and $\hat{L}$ are \textit{compatible} if they satisfy the
following assumptions:
\begin{enumerate}
\item $C^*=\hat{C}^*$. Up to relabelling the ends, we may assume that $C_i^*=\hat{C}_i^*$.
\item $\hat{R}<\epsilon$.  We can then identify
appropriate
subsets of $S^*$ and $\hat{S}^*$ via the maps $\hat{\phi}_i\circ\phi_i^{-1}$. 
\item On each marked AC end,
the metrics $\hat{\phi}_i^*\hat{g}$ and $\tg_i$ are scaled-equivalent in the
sense of Definition \ref{def:equivalentscaledmetrics}. Analogously, on each
marked CS end,
the metrics $\phi_i^*g$ and $\tg_i$ are scaled-equivalent in the
sense of Definition \ref{def:equivalentscaledmetrics}. 
\end{enumerate}
If
$L$ is weighted via $\boldsymbol{\beta}$ and $\hat{L}$ is weighted via
$\hat{\boldsymbol{\beta}}$ we further require that, on the marked ends, the
corresponding constants satisfy
$\beta_i=\hat{\beta}_i$ and that $\hat{\beta}_i=\hat{\beta}_j$ if $\hat{S}_i$ and $\hat{S}_j$ are marked ends in the same connected component of $\hat{L}$.
\end{definition}

\begin{remark}\label{rem:compatible}
The condition $\hat{R}<\epsilon$ may seem rather strong. However, let
$(L,g,S^*)$ be CS-marked, $(\hat{L},\hat{g},\hat{S}^*)$ be AC-marked
and $C^*=\hat{C}^*$. As seen in Remark \ref{rem:metrics_ends}, by making
$\hat{R}$ larger if necessary it is possible to assume that the
metrics $\hat{\phi}_i^*\hat{g}$, $\tg_i$ on $\Sigma_i\times [\hat{R},\infty)$
are scaled-equivalent in the
sense of Definition \ref{def:equivalentscaledmetrics}. Lemma
\ref{l:rescaledconifold} then shows that the metrics
$\hat{\phi}_{t,i}^*(t^2\hat{g})$, $\tg_i$ on
$\Sigma_i\times [t\hat{R},\infty)$ are also scaled-equivalent, with the same
bounds. Analogously, by making $\epsilon$ smaller if necessary, we can assume 
that the metrics $\phi_i^*g$, $\tg_i$ on $\Sigma_i\times
(0,\epsilon]$ are scaled-equivalent. By first making
$\hat{R}$ large and $\epsilon$ small and then rescaling to satisfy the
condition $\hat{R}<\epsilon$ we thus obtain 
compatible conifolds in the sense of Definition \ref{def:compatible}.
\end{remark}

\begin{definition}\label{def:connectsum}
 Let $(L,g,S^*)$, $(\hat{L},\hat{g},\hat{S}^*)$ be
compatible marked conifolds. We define the \textit{connect sum}
of $L$ and $\hat{L}$ as follows. We set
\begin{equation}
\hat{L}\#L:=(\hat{L}\setminus\hat{S}^*)\cup
(\Sigma^*\times[\hat{R},\epsilon])\cup (L\setminus S^*), 
\end{equation}
where the boundary of $\hat{L}\setminus\hat{S}^*$ is identified with
$\Sigma^*\times\{\hat{R}\}$ via the maps $\hat{\phi}_i$ and the boundary of
$L\setminus S^*$ is identified with $\Sigma^*\times\{\epsilon\}$ via
the maps $\phi_i$. We can endow this manifold with any metric $\hat{g}\#g$ which
restricts to $\hat{g}$ on $\hat{L}\setminus\hat{S}^*$ and to $g$ on
$L\setminus{S}^*$. Then $\hat{L}\#L$ is a conifold. Its ends are
$\hat{S}^{**}\amalg S^{**}$. We
call $\Sigma^*\times[\hat{R},\epsilon]$ the \textit{neck region} of
$\hat{L}\#L$.

Given radius functions $\rho$ on $L$ and $\hat{\rho}$ on $\hat{L}$ we can endow
$\hat{L}\#L$ with the radius function
\begin{equation*}
 \hat{\rho}\#\rho:=\left\{
\begin{array}{ll}
 \hat{\rho} & \mbox{ on } \hat{L}\setminus\hat{S}^*\\
r & \mbox{ on } \Sigma^*\times[\hat{R},\epsilon]\\
\rho & \mbox{ on }L\setminus S^*.
\end{array}\right.
\end{equation*}
If $L$, $\hat{L}$ are weighted via $\boldsymbol{\beta}$,
$\hat{\boldsymbol{\beta}}$ then $\hat{L}\#L$ is weighted via the function
\begin{equation*}
 \hat{\boldsymbol{\beta}}\#\boldsymbol{\beta}:=\left\{
\begin{array}{ll}
 \hat{\boldsymbol{\beta}} & \mbox{ on } \hat{L}\setminus\hat{S}^*\\
\boldsymbol{\beta}_{|S^*} & \mbox{ on } \Sigma^*\times[\hat{R},\epsilon]\\
\boldsymbol{\beta} & \mbox{ on }L\setminus S^*.
\end{array}\right.
\end{equation*}
\end{definition}

\begin{example}\label{e:connectsum}
Let $\overline{L}$ be a smooth $m$-dimensional submanifold of
$\R^n$, endowed with the induced metric. Assume that it is either compact or
that it has AC ends: \textit{e.g.}, it could be a collection of $m$-planes in
$\R^n$. Now assume it has transverse self-intersection
points $x_1,\dots,x_k\in \R^n$. For each $x_i$ choose a ball $B(x_i,\epsilon)$
in $\R^n$. Then $L:=\overline{L}\setminus \{x_1,\dots,x_k\}$ is a conifold with
$s$ CS ends defined by the connected components of $\left(
B(x_1,\epsilon)\cup\dots\cup B(x_k,\epsilon)\right)\cap L$. The corresponding
cones are copies of $\R^m$. Choose a pair $S_1$, $S_2$ of
connected components of $B(x_1,\epsilon)\cap L$ and an appropriately rescaled
$m$-dimensional hyperboloid $\hat{L}\subseteq \R^n$ asymptotic to the
corresponding cones $C_1$,
$C_2$. Then $L$, $\hat{L}$ are compatible and $\hat{L}\#L$ is an abstract
Riemannian manifold which we can think of as a desingularization of
$\overline{L}$. Our hypothesis in Definition \ref{def:marking} that $L$,
$\hat{L}$ are not necessarily connected allows us to extend this construction to
intersection points of distinct submanifolds and to desingularize all points
simultaneously.
\end{example}

Since $\hat{L}\#L$ is again a conifold it is clear that all versions of the
Sobolev Embedding Theorems continue to hold for it. Notice that
$\hat{S}^{**}\cup S^{**}$ might also be empty: in this case $\hat{L}\#L$ is
a smooth compact manifold. We now consider the parametric version of this
construction.
 
\begin{definition}\label{def:tconnectsum}
 Let $(L,g,S^*)$, $(\hat{L},\hat{g},\hat{S}^*)$ be
compatible marked conifolds with $d$ marked ends. Let
$(\rho,\boldsymbol{\beta})$, respectively
$(\hat{\rho},\hat{\boldsymbol{\beta}})$, be corresponding radius functions and
weights. Choose parameters $\boldsymbol{t}=(t_1,\dots,t_d)>0$ sufficiently
small. We assume that $\boldsymbol{t}$ is compatible with the
decomposition of $\hat{L}$ into its connected components: specifically, that
$t_i=t_j$ if $\hat{S}_i$ and $\hat{S}_j$ belong to the same connected component
of $\hat{L}$. We then define the
\textit{parametric connect sum}
of $L$ and $\hat{L}$ as follows. We set
\begin{equation*}
L_{\boldsymbol{t}}:=(\hat{L}\setminus\hat{S}^*)\cup
(\cup_{\Sigma_i\subseteq\Sigma^*}\Sigma_i\times[t_i\hat{R},\epsilon])\cup
(L\setminus S^*), 
\end{equation*}
where the components of the boundary of $\hat{L}\setminus\hat{S}^*$ are
identified with
the $\Sigma_i\times\{t_i\hat{R}\}$ via maps $\hat{\phi}_{t_i,i}$ defined as in
Lemma
\ref{l:rescaledconifold} and the components of the boundary of $L\setminus S^*$
are identified
with the $\Sigma_i\times\{\epsilon\}$ via the maps $\phi_i$. Choose
$\tau\in (0,1)$. If the $t_i$ are sufficiently small, we find
$t_i\hat{R}<t_i^\tau<2t_i^\tau<\epsilon$. Choose any
metric $g_{\boldsymbol{t}}$ on $L_{\boldsymbol{t}}$ such that, for each
$\Sigma_i\subseteq \Sigma^*$,
\begin{equation*}
 g_{\boldsymbol{t}}:=\left\{
\begin{array}{llll}
 t_i^2\hat{g} & \mbox{ on the corresponding component of }
\hat{L}\setminus\hat{S}^*\\
\hat{\phi}_{t_i,i}^*(t_i^2\hat{g}) & \mbox{ on
}\Sigma_i\times[t_i\hat{R},t_i^\tau]\\
\phi_i^*g & \mbox{ on } \Sigma_i\times[2t_i^\tau,\epsilon]\\
g & \mbox{ on } L\setminus S^*
\end{array}\right.
\end{equation*}
and such that, for all $j\geq 0$ and as
$\boldsymbol{t}\rightarrow 0$,
$$\sup_{\Sigma_i\times[t_i^\tau,2t_i^\tau]}
|\tnabla^j(g_{\boldsymbol{t}}-\tg_i)|_{
r^ { -2 } \tg_i\otimes\tg_i} \rightarrow 0.$$
 We endow $L_{\boldsymbol{t}}$ with the radius function
\begin{equation*}
 \rho_{\boldsymbol{t}}:=\left\{
\begin{array}{ll}
 t_i\hat{\rho} & \mbox{ on the corresponding component of }
\hat{L}\setminus\hat{S}^*\\
r & \mbox{ on } \Sigma_i\times[t_i\hat{R},\epsilon]\\
\rho & \mbox{ on }L\setminus S^*
\end{array}\right.
\end{equation*}
and the weight
\begin{equation*}
 \boldsymbol{\beta}_{\boldsymbol{t}}:=\left\{
\begin{array}{ll}
 \hat{\boldsymbol{\beta}} & \mbox{ on } \hat{L}\setminus\hat{S}^*\\
\beta_i & \mbox{ on } \Sigma_i\times[t_i\hat{R},\epsilon]\\
\boldsymbol{\beta} & \mbox{ on }L\setminus S^*.
\end{array}\right.
\end{equation*}
We now need to define the weight function $w_{\boldsymbol{t}}$. As in Corollary \ref{cor:rescaledconifold}, the simplest
case is when $\hat{\boldsymbol{\beta}}$ is constant on each connected
component of $\hat{L}$. We then define
\begin{equation*}
 w_{\boldsymbol{t}}:=\rho_t^{-\boldsymbol{\beta}_t}=\left\{
\begin{array}{ll}
 (t_i\hat{\rho})^{-\hat{\beta}_i} & \mbox{ on the corresponding
component of }
\hat{L}\setminus\hat{S}^*\\
r^{-\beta_i} & \mbox{ on } \Sigma_i\times[t_i\hat{R},\epsilon]\\
\rho^{-\boldsymbol{\beta}} & \mbox{ on }L\setminus S^*.
\end{array}\right.
\end{equation*}
For general weights $\hat{\boldsymbol{\beta}}$ we need to modify the
weight function. As in Corollary \ref{cor:rescaledconifold}, on the $i$-th component of
$\hat{L}$ consider the constant ``reference'' weight
$\hat{\beta}_i$. We then define
\begin{equation*}
 w_{\boldsymbol{t}}:=\left\{
\begin{array}{ll}
 (t_i^\frac{\hat{\beta}_i-\hat{\boldsymbol{\beta}}}{\hat{\boldsymbol{\beta}}}t_i\hat{\rho})^{-\hat{
\boldsymbol {\beta}} }  & \mbox{ on the
corresponding component of }\hat{L}\setminus\hat{S}^*\\
r^{-\beta_i} & \mbox{ on } \Sigma_i\times[t_i\hat{R},\epsilon]\\
\rho^{-\boldsymbol{\beta}} & \mbox{ on }L\setminus S^*.
\end{array}\right.
\end{equation*}
We may equivalently write this as
\begin{equation*}
w_{\boldsymbol{t}}:=\left\{
\begin{array}{ll}
t_i^{\hat{\boldsymbol{\beta}}-\hat{\beta}_i}\rho_{\boldsymbol{t}}^{-\boldsymbol{\beta}_{\boldsymbol{t}}}&\mbox{ on }\hat{L}\setminus \hat{S}^*\\
\rho_{\boldsymbol{t}}^{-\boldsymbol{\beta}_{\boldsymbol{t}}}&\mbox{ elsewhere}.
\end{array}\right.
\end{equation*}
Using this data we now define weighted Sobolev spaces
$W^p_{k,\boldsymbol{\beta}_{\boldsymbol{t}}}$ on $L_{\boldsymbol{t}}$ as in
Section \ref{s:weighted}. We call $\Sigma_i\times[t_i\hat{R},\epsilon]$ the
\textit{neck regions} of $L_{\boldsymbol{t}}$.
\end{definition}

\begin{theorem}\label{thm:normstequivalent}
 Let $(L,g,S^*)$, $(\hat{L},\hat{g},\hat{S}^*)$ be
compatible weighted marked conifolds. Define $L_{\boldsymbol{t}}$, $g_{\boldsymbol{t}}$,
$\rho_{\boldsymbol{t}}$ and
$\boldsymbol{\beta}_{\boldsymbol{t}}$
as in Definition \ref{def:tconnectsum}. Then all
forms of the weighted Sobolev Embedding Theorems hold uniformly in
${\boldsymbol{t}}$,
\textit{i.e.} the corresponding Sobolev constants are independent of
${\boldsymbol{t}}$.
\end{theorem}
\begin{proof}
The proof is similar to that of Corollary \ref{cor:embedding}. Let us for the
moment pretend that the metrics $g$, $\hat{g}$ are exactly conical on all ends
of $L$, $\hat{L}$. This allows us to assume that the metrics
$g_{\boldsymbol{t}}$ are exactly
conical on
all ends and neck regions of $L_{\boldsymbol{t}}$ so the assumptions of
Theorem \ref{thm:weighted_ok} are satisfied in these regions. On
$\hat{L}\setminus\hat{S}^*$ we are using rescaled metrics, radius
functions and weights as in Corollary \ref{cor:rescaledconifold}. As seen in
Remark \ref{rem:tweighted_ok}, the assumptions of Theorem \ref{thm:weighted_ok}
are $\boldsymbol{t}$-independent so they are verified here.
These assumptions are also
verified on $L\setminus S^*$ and on the neck regions. We conclude that all forms
of the weighted Sobolev
Embedding Theorems hold for these metrics, with
$\boldsymbol{t}$-independent Sobolev constants.

Let us now go back to the metric $g_{\boldsymbol{t}}$. Recall from Lemma
\ref{l:rescaledconifold} that we can assume that, on each end of
$L_{\boldsymbol{t}}$, $g_{\boldsymbol{t}}$ is a $\boldsymbol{t}$-uniformly
small perturbation of the conical metric. The same is true also on the neck
regions. Specifically, on $\Sigma_i\times[t_i\hat{R},t_i^\tau]$ Lemma
\ref{l:rescaledconifold} shows
that 
\begin{equation*}
 \sup |\phi_{t,i}^*(t_i^2\hat{g})-\tg_i|\leq C_0\hat{R}^{\hat{\nu}_i}.
\end{equation*}
On $\Sigma_i\times[t_i^\tau,2t_i^\tau]$ our hypotheses imply 
\begin{equation*}
\sup |g_{\boldsymbol{t}}-\tg_i|_{
r^ { -2 } \tg_i\otimes\tg_i} \leq C_0.
\end{equation*}
The analogue is true also on $\Sigma_i\times[2t_i^\tau,\epsilon]$, using the
estimates provided by Definition \ref{def:metrics_ends}. 

These perturbations are all $\boldsymbol{t}$-independent so according to
Theorem \ref{thm:weighted_ok} the weighted Sobolev Embedding Theorems 
hold also for $g_{\boldsymbol{t}}$, with $\boldsymbol{t}$-independent Sobolev
constants.
\end{proof}

\begin{remark} Notice that Theorem
\ref{thm:normstequivalent} actually requires only $\boldsymbol{t}$-uniform
$C^0$-bounds
over the metrics $g_{\boldsymbol{t}}$. In Definition \ref{def:tconnectsum} we
include control
over the higher derivatives and the assumption that the quantities in
question tend to zero for use in later sections. The same is also true for
various other results, \textit{e.g.} Corollary \ref{cor:embedding}.
\end{remark}

We conclude with the following result which serves to highlight certain 
properties
of $g_{\boldsymbol{t}}$ as $\boldsymbol{t}\rightarrow 0$. This is
important for Section \ref{s:sums_laplace}.

\begin{lemma}\label{l:neckestimate}
Consider $g_{\boldsymbol{t}}$ as in Definition \ref{def:tconnectsum}. Choose
a neck region in $L_{\boldsymbol{t}}$ and $b\in (0,\tau)$
so that $t_i\hat{R}<t_i^\tau<2t_i^\tau<t_i^b<\epsilon$.
 Then, on $\Sigma_i\times [t_i\hat{R},t_i^b]$, the metric
$g_{\boldsymbol{t}}$ converges to the rescaled metric $t_i^2\hat{\phi}_{t,i}^*\hat{g}$ in the
following sense: for all $j\geq 0$ and as
$t\rightarrow 0$,
\begin{equation*}
 \sup|r^j\hat{\nabla}^j(g_{\boldsymbol{t}}-t_i^2\hat{\phi}_{t,i}^*\hat{g})|_{
t_i^2\hat{\phi}_{t,i}^*\hat{g}\otimes t_i^2\hat{\phi}_{t,i}^*\hat{g}
}\rightarrow 0,
\end{equation*}
where $\hat{\nabla}$ denotes the Levi-Civita connection defined by
$\hat{\phi}_{t_i,i}^*\hat{g}$ on $\Sigma_i\times[t_i\hat{R},t_i^b]$.
\end{lemma}
\begin{proof}
Consider the map 
\begin{equation*}
 \delta_{t_i}:\Sigma_i\times[\hat{R},t_i^{b-1}]\rightarrow \Sigma_i\times
[t_i\hat{R},t_i^b], \ \ (\theta,r)\mapsto(\theta,t_ir).
\end{equation*}
We can use this map to pull the estimate back to $\Sigma_i\times[\hat{R},t_i^{b-1}]$. We can then write it as follows: for all $j\geq 0$ and as
$t\rightarrow 0$,
\begin{equation}\label{eq:neckestimate}
\sup
|\hat{\nabla}^j(\delta_{t_i}^*(t_i^{-2}g_{\boldsymbol{t}})-\hat{\phi}_i^*\hat{g}
)|_ { r^ { -2 }
\hat{\phi}_i^*\hat{g}\otimes\hat{\phi}_i^*\hat{g} } \rightarrow 0,
\end{equation}
where $\hat{\nabla}$ denotes the Levi-Civita connection defined by
$\hat{\phi}_i^*\hat{g}$ on $\Sigma_i\times[\hat{R},t^{b-1}]$.

We choose to prove this form of the estimate. 

On $\Sigma_i\times [\hat{R},t_i^{\tau-1}]$ it
follows from Definition
\ref{def:tconnectsum} that
$\delta_{t_i}^*(t_i^{-2}g_{\boldsymbol{t}})=\hat{\phi}_i^*\hat{g}$
so the equation is trivially true.

On $\Sigma_i\times[t_i^{\tau-1},2t_i^{\tau-1}]$, 
\begin{align*}
|\tnabla^j(\delta_{t_i}^*(t_i^{-2}g_{\boldsymbol{t}})-\tg_i)|_{r^{-2}
\tg_i\otimes\tg_i }
&=|\tnabla^j(\delta_{t_i}^*(t_i^{-2}g_{\boldsymbol{t}})-\delta_{t_i}^*(t_i^{-2}
\tg_i))|_
{\delta_{t_i}^*(r/{t_i})^{-2}\delta_{t_i}^*({t_i}^{-2}\tg_i)\otimes\delta_{
t_i}^*(t_i^ { -2 } \tg_i) } \\
&=\delta_{t_i}^*
\left(|\tnabla^j(t_i^{-2}g_{\boldsymbol{t}}-t_i^{-2}\tg_i)|_{(r/{t_i})^{-2}t_i^{
-2 } \tg_i\otimes
t_i^{-2}\tg_i}\right)\\
&=\delta_{t_i}^*\left(|\tnabla^j(g_{\boldsymbol{t}}-\tg_i)|_{r^{-2}
\tg_i\otimes\tg_i }
\right)\rightarrow 0,
\end{align*}
where the last statement follows from Definition \ref{def:tconnectsum}.
Furthermore, it follows from Definition \ref{def:metrics_ends} that
\begin{equation*}
 |\tnabla^j(\hat{\phi}_i^*\hat{g}-\tg_i)|_{r^{-2}\tg_i\otimes\tg_i}\leq
C_j{t_i}^{(\tau-1)\hat{\nu}_i}\rightarrow 0,
\end{equation*}
using $(\tau-1)\hat{\nu}_i>0$. We have thus found that both metrics of interest
converge to the same metric $\tg_i$. The conclusion is a simple computation.

On $\Sigma_i\times[2t_i^{\tau-1},t_i^{b-1}]$, as above and using
$g_{\boldsymbol{t}}=\phi_i^*g$,
\begin{equation*}
|\tnabla^j(\delta_{t_i}^*(t_i^{-2}g_{\boldsymbol{t}})-\tg_i)|_{r^{-2}
\tg_i\otimes\tg_i }
=\delta_{t_i}^*\left(|\tnabla^j(\phi_i^*g-\tg_i)|_{r^{-2}\tg_i\otimes\tg_i}
\right)\leq
C_jt_i^{b\nu_i}\rightarrow 0 ,
\end{equation*}
using $b\nu_i>0$. Furthermore, 
\begin{equation*}
 |\tnabla^j(\hat{\phi}_i^*\hat{g}-\tg_i)|_{r^{-2}\tg_i\otimes\tg_i}\leq
C_j(2t_i^{\tau-1})^{\hat{\nu}_i}\rightarrow 0.
\end{equation*}
Again, combining these estimates implies the claim.
\end{proof}


\section{The Laplacian on conifold connect sums}\label{s:sums_laplace}

Let $(L,g,\rho, S^*)$,
$(\hat{L},\hat{g},\hat{\rho},\hat{S}^*)$ be compatible marked
conifolds. As seen in Section \ref{s:sums_sobolev}, we can define their connect sum
$(\hat{L}\#L,\hat{g}\#g,\hat{\rho}\#\rho)$. This is a new conifold so we can
study the properties of its Laplace operator as in Section
\ref{s:accs_harmonic}. 

We start with the case in which $\hat{S}^{**}\cup S^{**}\neq\emptyset$, \textit{i.e.} the set of ends is non-empty. This case actually turns out to be easier than the alternative situation, where $\hat{L}\#L$ is smooth and compact, because we can use weights to force injectivity of the Laplacian.
\subsubsection*{Non-compact conifolds} Assume the set $\hat{S}^{**}\cup S^{**}$ of ends of $\hat{L}\#L$ is non-empty. If weights $\boldsymbol{\beta}$,
$\hat{\boldsymbol{\beta}}$ are non-exceptional for $\Delta_g$,
$\Delta_{\hat{g}}$ then the weight
$\hat{\boldsymbol{\beta}}\#\boldsymbol{\beta}$ is non-exceptional for
$\Delta_{\hat{g}\#g}$ so 
$$\Delta_{\hat{g}\#g}:W^p_{k,\hat{\boldsymbol{\beta}}\#\boldsymbol{\beta}}
\rightarrow W^p_{k-2,\hat{\boldsymbol{\beta}}\#\boldsymbol{\beta}-2} $$ 
is Fredholm. The same holds
for the parametric connect sums
$(L_{\boldsymbol{t}},g_{\boldsymbol{t}},\rho_{\boldsymbol{t}},\boldsymbol{\beta}
_{\boldsymbol{t}})$. 

We want to study the invertibility of the Laplace operator. The following result is obvious.

\begin{lemma}\label{l:sum_injective}
Let $(\hat{L},\hat{g},\hat{\rho},\hat{\boldsymbol{\beta}},\hat{S}^*)$ be a weighted
AC-marked conifold. Assume $\hat{\boldsymbol{\beta}}$ satisfies the
conditions
\begin{equation*}
\left\{
\begin{array}{l}
\hat{\beta}_i<0\mbox{ for all AC ends
}\hat{S}_i\in \hat{S}\\
\hat{\beta}_i>2-m\mbox{
for all CS ends }\hat{S}_i\in \hat{S}
\end{array}
\right.
\end{equation*}
so that $\Delta_{\hat{g}}$ is injective.

Let $(L,g,\rho,\boldsymbol{\beta}, S^*)$ be a weighted CS-marked conifold. Assume
$\boldsymbol{\beta}$ satisfies the conditions
\begin{equation*}
\left\{
\begin{array}{l}
\beta_i<0\mbox{ for all AC
ends }S_i\in S\\
\beta_i>2-m\mbox{ for all CS ends
}S_i\in S.
\end{array}
\right.
\end{equation*}
This is not yet sufficient to conclude that $\Delta_g$ is
injective because the set of AC ends might be empty. To obtain injectivity we
must furthermore assume that each connected component of $L$ has at least one end, \textit{e.g.}
$S'$, satisfying the condition
\begin{equation*}
\left\{
\begin{array}{l}
\beta'<0 \mbox{ if $S'$ is AC}\\
\beta'>0 \mbox{ if $S'$ is CS}.
\end{array}
\right.
\end{equation*}

Now assume that $L$, $\hat{L}$ are compatible. Then, for all ends $S_i\in S^*$,
$2-m<\beta_i<0$. This implies that $S'\in S^{**}$ so
$\hat{L}\#L$ has at least one end. Furthermore,
$\hat{\boldsymbol{\beta}}\#\boldsymbol{\beta}$ satisfies the conditions
\begin{equation*}
\left\{
\begin{array}{l}
\hat{\boldsymbol{\beta}}\#\boldsymbol{\beta}_{|S_i}<0\mbox{ for all AC ends
}S_i\in \hat{S}^{**}\cup S^{**}\\
\hat{\boldsymbol{\beta}}\#\boldsymbol{\beta}_{|S_i}>2-m\mbox{
for all CS ends }S_i\in \hat{S}^{**}\cup S^{**}.
\end{array}
\right.
\end{equation*}
Together with the condition on $S'$, this implies that
$\Delta_{\hat{g}\#g}$ is injective. 

If furthermore $\boldsymbol{\beta}$, $\hat{\boldsymbol{\beta}}$ are
non-exceptional for $\Delta_g$,
$\Delta_{\hat{g}}$ then 
$$\Delta_{\hat{g}\#g}:W^p_{k,\hat{\boldsymbol{\beta}}\#\boldsymbol{\beta}}
\rightarrow W^p_{k-2,\hat{\boldsymbol{\beta}}\#\boldsymbol{\beta}-2} $$
is a topological isomorphism onto its image so there exists $C>0$ such that, for
all $f\in W^p_{k,\hat{\boldsymbol{\beta}}\#\boldsymbol{\beta}}$,
$$\|f\|_{W^p_{k,\hat{\boldsymbol{\beta}}\#\boldsymbol{\beta}}}\leq C\|\Delta
f\|_{W^p_{k-2,\hat{\boldsymbol{\beta}}\#\boldsymbol{\beta}-2}}.$$
\end{lemma}
For the constant $C$ in Lemma \ref{l:sum_injective} one can choose the norm of
the inverse map $(\Delta_{\hat{g}\#g})^{-1}$,
as in Equation \ref{eq:normofinverse}. The analogous result holds also for
parametric connect sums. We now want to show that, in this case, the
invertibility constant $C$ can be chosen to be $\boldsymbol{t}$-independent. In
other words,
there exists a $\boldsymbol{t}$-uniform upper bound on the norms of the inverse maps
$(\Delta_{g_{\bt}})^{-1}$.

\begin{theorem}\label{thm:sum_injective}
Let $(L,g,\rho,\boldsymbol{\beta}, S^*)$,
$(\hat{L},\hat{g},\hat{\rho},\hat{\boldsymbol{\beta}},\hat{S}^*)$ be marked
compatible conifolds satisfying all the conditions of Lemma
\ref{l:sum_injective}. Define $(L_{\boldsymbol{t}},g_{\boldsymbol{t}},\rho_{\boldsymbol{t}},\boldsymbol{\beta}_{\boldsymbol{t}})$ as in Definition
\ref{def:tconnectsum}. Then there exists $C>0$ such that, for all $f\in
W^p_{k,\boldsymbol{\beta}_t}(L_{\boldsymbol{t}})$, 
$$\|f\|_{W^p_{k,\boldsymbol{\beta}_{\boldsymbol{t}}}}\leq
C\|\Delta_{g_{\boldsymbol{t}}}
f\|_{W^p_{k-2,\boldsymbol{\beta}_{\boldsymbol{t}}-2}}.$$
\end{theorem}
\begin{proof}
To simplify the notation let us assume that all $t_i$ coincide: we can then
work with a unique parameter $t$. The general case is analogous.

 Let $C_g$ denote an invertibility constant for $\Delta_g$ on $L$,
\textit{i.e.} for all $f\in W^p_{k,\boldsymbol{\beta}}(L)$,
$$\|f\|_{W^p_{k,\boldsymbol{\beta}}}\leq C_g\|\Delta_g
f\|_{W^p_{k-2,\boldsymbol{\beta}-2}}.$$

Let $C_{\hat{g}}$ denote an analogous constant for $\Delta_{\hat{g}}$ on
$\hat{L}$.

Choose constants $a$, $b$ satisfying $0<b<a<\tau$ and a smooth
decreasing function $\eta:\R\rightarrow [0,1]$ such that $\eta(s)=1$ for
$s\leq b$ and $\eta(s)=0$ for $s\geq a$. Then the function $\eta_t(r):=\eta
(\log r/\log t):(0,\infty)\rightarrow [0,1]$ has the following properties:
\begin{enumerate}
\item $\eta_t$ is smooth increasing, $\eta_t(r)=0$ for $r\leq t^a$,
$\eta_t(r)=1$ for $r\geq t^b$.
\item For all $k\geq 1$ there exists $C_k>0$ such that
$$\left|r^k\frac{\partial^k\eta_t}{(\partial r)^k}(r)\right|\leq
\frac{C_k}{|\log t|}\rightarrow 0\ \ \mbox{ as
$t\rightarrow 0$}.$$
We set $\eta_t'(r):=\frac{\partial\eta_t}{\partial r}(r)$, 
$\eta_t''(r):=\frac{\partial^2\eta_t}{(\partial r)^2}(r)$. 
\end{enumerate}
Using the diffeomorphisms $\hat{\phi}_{t,i}$ and $\phi_i$ we now extend
$\eta_t$ to a smooth function on $L_{\boldsymbol{t}}$ by setting $\eta_t\equiv
0$ on $(\hat{L}\setminus \hat{S}^*)\cup(\Sigma^*\times[t\hat{R},t^a])$ and
$\eta_t\equiv 1$ on $(L\setminus S^*)\cup(\Sigma^*\times[t^b,\epsilon])$.

For any $f\in W^p_{k,\boldsymbol{\beta}_t}$,
$$\|f\|_{W^p_{k,\boldsymbol{\beta}_t}}\leq
\|\eta_tf\|_{W^p_{k,\boldsymbol{\beta}_t}}
+\|(1-\eta_t)f\|_ { W^p_{k,\boldsymbol{\beta}_t}}.$$
Notice that $\eta_tf$ has support in $(\Sigma^*\times
[t^a,\epsilon])\cup(L\setminus S^*)$, where, up to
identifications via the diffeomorphisms $\phi_i$,
$(g_{\boldsymbol{t}},\rho_{\boldsymbol{t}})=(g,\rho)$,
$\boldsymbol{\beta}_{\boldsymbol{t}}=\boldsymbol{\beta}$. Thus 
\begin{align*}
 \|\eta_tf\|_{W^p_{k,\boldsymbol{\beta}_{\boldsymbol{t}}}(g_{\boldsymbol{t}})}
&=\|\eta_tf\|_{W^p_{k,\boldsymbol{\beta}}(g)}\\
&\leq C_g\|\Delta_g(\eta_tf)\|_{W^p_{k-2,\boldsymbol{\beta}-2}(g)}\\
&=C_g\|\Delta_{g_{\boldsymbol{t}}}(\eta_tf)\|_{W^p_{k-2,\boldsymbol{\beta}_{
\boldsymbol{t}}-2} (g_{\boldsymbol{t}})}\\
&\leq
C_g\left(\|\eta_t\Delta_{g_{\boldsymbol{t}}}f\|_{W^p_{k-2,\boldsymbol{\beta}_{
\boldsymbol{t}}-2 }}+
\|\eta_t'\nabla f\|_{W^p_{k-2,\boldsymbol{\beta}_{\boldsymbol{t}}-2}}+
\|\eta_t''f\|_{W^p_{k-2,\boldsymbol{\beta}_{\boldsymbol{t}}-2}}\right),
\end{align*}
where we drop unnecessary constants. 
Applying the Leibniz rule to expressions of the form
$\nabla^j(\eta_t\Delta_{g_{\boldsymbol{t}}}f)$ we find (again up to constants)
\begin{align*}
 \|\eta_t\Delta_{g_{\boldsymbol{t}}}f\|^p_{W^p_{k-2,\boldsymbol{\beta}_{
\boldsymbol{t}}-2 }}&\leq
\sum_{j=0}^{k-2}\sum_{l=0}^{j}\int|\rho^l\nabla^l\eta_t|^p_{g_{\boldsymbol{t}}}
|\rho^{2-\boldsymbol{ \beta_{
\boldsymbol{t}} } +j-l } \nabla^ { j
-l}\Delta_{g_{\boldsymbol{t}}}f|^p_{g_{\boldsymbol{t}}}\rho^{-m}\mbox{vol}_{g_{
\boldsymbol{t}} }\\
&\leq \left(1+\left(\frac{C}{|\log
t|}\right)^p\right)\|\Delta_{g_{\boldsymbol{t}}}f\|^p_{W^p_{k-2,\boldsymbol{
\beta } _ {
\boldsymbol{t}}-2}}.
\end{align*}
We conclude that 
\begin{equation*}
 \|\eta_t\Delta_{g_{\boldsymbol{t}}}f\|_{W^p_{k-2,\boldsymbol{\beta}_{
\boldsymbol{t}}-2 }}\leq \|\Delta_{g_{\boldsymbol{t}}}f\|_{W^p_{k-2,\boldsymbol{\beta}_{
\boldsymbol{t}}-2 }}+\frac{C}{|\log t|}\|f\|_{W^p_{k,\boldsymbol{\beta}_{
\boldsymbol{t}}}}.
\end{equation*}
Analogously,
\begin{align*}
\|\eta_t'\nabla f\|^p_{W^p_{k-2,\boldsymbol{\beta}_{\boldsymbol{t}}-2}}&\leq
\sum_{j=0}^{k-2}\sum_{l=0}^{j}\int|\rho^{1+l}\nabla^l\eta_t'|^p_{g_{\boldsymbol{t}}}
|\rho^{1-\boldsymbol{ \beta_{
\boldsymbol{t}} } +j-l } \nabla^ { j
-l}\nabla f|^p_{g_{\boldsymbol{t}}}\rho^{-m}\mbox{vol}_{g_{
\boldsymbol{t}} }\\
&\leq \left(\frac{C}{|\log
t|}\right)^p
\|f\|^p_{W^p_{k,\boldsymbol{\beta}_{\boldsymbol{t}}}}.
\end{align*}
Similar calculations apply to $\|\eta_t''f\|$, ultimately showing that
\begin{equation*}
\|\eta_t'\nabla f\|_{W^p_{k-2,\boldsymbol{\beta}_{\boldsymbol{t}}-2}}\leq
\frac{C}{|\log
t|}
\|f\|_{W^p_{k,\boldsymbol{\beta}_{\boldsymbol{t}}}},\
\ 
\|\eta_t''f\|_{W^p_{k-2,\boldsymbol{\beta}_{\boldsymbol{t}}-2}}\leq
\frac{C}{|\log
t|}
\|f\|_{W^p_{k,\boldsymbol{\beta}_{\boldsymbol{t}}}}.
\end{equation*}

The function $(1-\eta_t)f$ has support in
$(\hat{L}\setminus\hat{S}^*)\cup(\Sigma^*\times[t\hat{R},t^b])$. On this space
Definition \ref{def:tconnectsum} shows that
$\boldsymbol{\beta}_{\boldsymbol{t}}=\hat{\boldsymbol{\beta}}$. Furthermore,
on the $i$-th component $\Sigma_i\times [t\hat{R},t^b]$ and up to identifications via the diffeomorphisms $\hat{\phi}_{t,i}$, Lemma \ref{l:neckestimate} shows that $g_{\boldsymbol{t}}$ is scaled-equivalent to $t^2\hat{g}$ and $\rho_{\boldsymbol{t}}=t\hat{\rho}$.

Using Corollary \ref{cor:rescaledconifold} we thus find
\begin{align*}
 \|(1-\eta_t)f\|_{W^p_{k,\boldsymbol{\beta}_{\boldsymbol{t}}}(g_{\boldsymbol{t}},\rho_{\boldsymbol
{t}})}
&\simeq
\|(1-\eta_t)f\|_ { W^p_{k,\hat{\boldsymbol{\beta}}}(t^2\hat{g},t\hat{\rho})}\\
&=t^{-\beta_i}\|(1-\eta_t)f\|_ { W^p_{k,\hat{\boldsymbol{\beta}}}(\hat{g},\hat{\rho})}\\
&\leq t^{-\beta_i}C_{\hat{g}}
\|\Delta_{\hat{g}}((1-\eta_t)f)\|_ {
W^p_{k-2,\hat{\boldsymbol{\beta}}-2}(\hat{g},\hat{\rho})}\\
&=t^{2-\beta_i}C_{\hat{g}} \|\Delta_{t^2\hat{g}}((1-\eta_t)f)\|_ {
W^p_{k-2,\hat{\boldsymbol{\beta}}-2}(\hat{g},\hat{
\rho}) }\\
&=C_{\hat{g}}\|\Delta_{t^2\hat{g}}((1-\eta_t)f)\|_ {
W^p_{k-2,\hat{\boldsymbol{\beta}}-2}(t^2\hat{g},t\hat{\rho})}\\
&\simeq C_{\hat{g}}\|\Delta_{g_{\boldsymbol{t}}}((1-\eta_t)f)\|_ {
W^p_{k-2,\boldsymbol{\beta}_{\boldsymbol{t}}-2}(g_{\boldsymbol{t}},\rho_{\boldsymbol{t}})},
\end{align*}
where $\simeq$ replaces multiplicative constants. We now continue as above.
Combining the above results leads to an inequality of the form
$$\|f\|_{W^p_{k,\boldsymbol{\beta}_{\boldsymbol{t}}}}\leq (C_g+C_{\hat{g}})
\left(\|\Delta_{g_{\boldsymbol{t}}}f\|_{W^p_{k-2,\boldsymbol{\beta}_{\boldsymbol
{t}}-2}}+\frac{C}{|\log
t|}
\|f\|_{W^p_{k,\boldsymbol{\beta}_{\boldsymbol{t}}}}\right).$$
For $t$ sufficiently small we can absorb
the second term on the
right hand side into the left hand side, proving the claim.
\end{proof}
\subsubsection*{Smooth compact manifolds} Assume the set $\hat{S}^{**}\cup S^{**}$ is empty, so that $\hat{L}\#L$ is smooth and compact. In this case the Laplace operator, acting on functions, always has kernel: the space of constants $\R$. We can thus achieve injectivity only by restricting ourselves to a subspace transverse to constants. Furthermore, if we want the invertibility constant to be independent of $\boldsymbol{t}$ we must allow the subspace to depend on $\boldsymbol{t}$, as follows.
\begin{theorem}\label{thm:cpt_sum_injective}
Let $(L,g,\rho, S^*)$,
$(\hat{L},\hat{g},\hat{\rho},\hat{S}^*)$ be marked
compatible conifolds such that the parametric connect sums $(L_{\boldsymbol{t}},g_{\boldsymbol{t}},\rho_{\boldsymbol{t}})$ are smooth and compact. Choose constant weights $\boldsymbol{\beta}=\hat{\boldsymbol{\beta}}\in (2-m,0)$ and define $\boldsymbol{\beta}_{\bt}$ as usual.  
\begin{enumerate}
\item Assume $L$ has only one connected component.
Then there exists a constant $C>0$ and, for each $\bt$ sufficiently small, a subspace $E_{\bt}\subset W^p_{k,\bt}(L_{\bt})$ such that 
\begin{equation}\label{eq:decomp_Et}
 W^p_{k,\bt}(L_{\bt})=E_{\bt}\oplus\R
\end{equation}
and, for all $f\in E_{\bt}$,
$$\|f\|_{W^p_{k,\boldsymbol{\beta}_{\boldsymbol{t}}}}\leq
C\|\Delta_{g_{\boldsymbol{t}}}
f\|_{W^p_{k-2,\boldsymbol{\beta}_{\boldsymbol{t}}-2}}.$$
Furthermore, the image of the restricted operator $\Delta_{g_{\bt}|E_{\bt}}$ coincides with the image of the full operator $\Delta_{g_{\bt}}$.
\item Assume $L$ has $k>1$ connected components. Then there exists a constant $C>0$ and, for each $\bt$ sufficiently small, a codimension $k$ subspace $E_{\bt}\subset W^p_{k,\bt}(L_{\bt})$ transverse to constants such that, for all $f\in E_{\bt}$,
$$\|f\|_{W^p_{k,\boldsymbol{\beta}_{\boldsymbol{t}}}}\leq
C\|\Delta_{g_{\boldsymbol{t}}}
f\|_{W^p_{k-2,\boldsymbol{\beta}_{\boldsymbol{t}}-2}}.$$
\end{enumerate}
\end{theorem}
\begin{proof}
Assume $L$ has one connected component. Choose any closed subspace $E\subset W^p_{k,\boldsymbol{\beta}}(L)$ such that 
\begin{equation*}
 W^p_{k,\boldsymbol{\beta}}(L)=E\oplus\R.
\end{equation*}
Define $\eta_{\bt}$ as in the proof of Theorem \ref{thm:sum_injective}. Extending it to zero on the CS ends of $L$, we can think of it as an element of $W^p_{k,\boldsymbol{\beta}}(L)$. One can check that $\eta_{\bt}\rightarrow 1$ in the $W^p_{k,\boldsymbol{\beta}}$ norm as $\bt\rightarrow 0$ so, for small $\bt$, $\eta_{\bt}\notin E$.
The multiplication map
\begin{equation*}
 P_{\bt}:W^p_{k,\boldsymbol{\beta}_{\bt}}(L_{\bt})\rightarrow W^p_{k,\boldsymbol{\beta}}(L),\ \ f\mapsto \eta_{\bt}f,
\end{equation*}
is linear and uniformly continuous with respect to the parameter $\bt$, so $E_{\bt}:=P_{\bt}^{-1}(E)$ is linear and closed. Since $\eta_{\bt}$ does not belong to $E$, constants do not belong to $E_{\bt}$. To confirm that $E_{\bt}$ has codimension 1, choose any linear function $Q:W^p_{k,\boldsymbol{\beta}}(L)\rightarrow\R$ such that $E=\mbox{Ker}(Q)$. Then $E_{\bt}=\mbox{Ker}(Q\circ P_{\bt})$, so it is defined by one linear condition. This proves Decomposition \ref{eq:decomp_Et}.

Consider $\Delta_{g_{\bt}}$ restricted to $E_{\bt}$. It is clearly injective. One can check that it is uniformly injective exactly as in Theorem \ref{thm:sum_injective}.

Now assume $L$ has multiple components $L_1,\dots,L_k$. For each $L_i$, choose a closed subspace $E_i\subset W^p_{k,\boldsymbol{\beta}}(L_i)$ as above. The multiplication map
\begin{equation*}
W^p_{k,\boldsymbol{\beta}_{\bt}}(L_{\bt})\rightarrow \bigoplus W^p_{k,\boldsymbol{\beta}_{\bt}}(L_i),\ \ f\mapsto \eta_{\bt}f,
\end{equation*}
is again linear and uniformly continuous, so we can define $E_{\bt}$ as the inverse of $E_1\oplus\dots,\oplus E_k$. One can again check that it has codimension $k$ and that, restricted to this space, $\Delta_{g_{\bt}}$ is uniformly injective.
\end{proof}
\begin{remark}
 Notice that, even though $L_{\bt}$ is smooth and compact, the proof of Theorem \ref{thm:cpt_sum_injective} requires the use of radius functions and weights on the necks.
\end{remark}


\section{Further Sobolev-type inequalities on conifold connect
sums}\label{s:improved_sob}
Given a conifold $(L,g)$, we can also apply the theory of Section
\ref{s:accs_analysis} to the elliptic operator
\begin{equation}\label{eq:D_g}
D_g=d\oplus d^*_g:W^p_{k,\boldsymbol{\beta}}(\Lambda^{even})\rightarrow
W^p_{k-1,\boldsymbol{\beta}-1}(\Lambda^{odd}),
\end{equation}
defined from the bundle of all even-dimensional forms on $L$ to the bundle of
all odd-dimensional forms. As for the Laplacian, it is possible to define and
study the exceptional weights for this operator. For any non-exceptional
weight $\boldsymbol{\beta}$, the operator $D_g$ of Equation \ref{eq:D_g} is
Fredholm. This implies that 
\begin{equation*}
 D_g:W^p_{k,\boldsymbol{\beta}}(\Lambda^{even})/\mbox{Ker}(D_g)\rightarrow
W^p_{k-1,\boldsymbol{\beta}-1}(\Lambda^{odd})
\end{equation*}
is a topological isomorphism onto its image. Notice that $W^p_{k,\boldsymbol{\beta}}(L)/\mbox{Ker}(D_g)$ is closed in $W^p_{k,\boldsymbol{\beta}}(\Lambda^{even})/\mbox{Ker}(D_g)$. It follows that
$d(W^p_{k,\boldsymbol{\beta}}(L))=D_g(W^p_{k,\boldsymbol{\beta}}(L))=D_g(W^p_{k,\boldsymbol{\beta}}(L)/\mbox{Ker}(D_g))$
is closed in $\mbox{Im}(D_g)$, thus in $W^p_{k-1,\boldsymbol{\beta}-1}(\Lambda^{odd})$. We can conclude that the restricted
operator
\begin{equation}\label{eq:d}
 d:W^p_{k,\boldsymbol{\beta}}(L)\rightarrow
W^p_{k-1,\boldsymbol{\beta}-1}(\Lambda^1)
\end{equation}
has closed image. Notice that $\mbox{Ker}(d)$ can only contain constants. If
the choice of weights is such that constants do not belong to the space
$W^p_{k,\boldsymbol{\beta}}(L)$, the operator $d$ of Equation \ref{eq:d} is a
topological isomorphism onto its image and can be inverted. In particular
there exists $C>0$ such that, for any $f\in W^p_{k,\boldsymbol{\beta}}(L)$,
\begin{equation*}
 \|f\|_{W^p_{k,\boldsymbol{\beta}}}\leq
C\|df\|_{W^p_{k-1,\boldsymbol{\beta}-1}}.
\end{equation*}
We now want to show that, on conifolds obtained as parametric connect sums, 
such $C$ can chosen independently of $\boldsymbol{t}$. For brevity, we restrict our attention to the non-compact case.

\begin{theorem}\label{thm:d_invertible}
Let $(\hat{L},\hat{g},\hat{\rho},\hat{\boldsymbol{\beta}},\hat{S}^*)$ be a weighted AC-marked conifold. Assume that  $\hat{\boldsymbol{\beta}}$ is non-exceptional for the operator
\begin{equation*}
 D_{\hat{g}}:W^p_{k,\hat{\boldsymbol{\beta}}}(\Lambda^{even})\rightarrow
W^p_{k-1,\hat{\boldsymbol{\beta}}-1}(\Lambda^{odd})
\end{equation*}
defined on the manifold $\hat{L}$ and that $\hat{\beta}_i<0$ for all ends $\hat{S}_i\in \hat{S}^*$.

Let $(L,g,\rho,\boldsymbol{\beta}, S^*)$ be a weighted CS-marked conifold. Assume $\boldsymbol{\beta}$ is non-exceptional for the operator
\begin{equation*}
D_g:W^p_{k,\boldsymbol{\beta}}(\Lambda^{even})\rightarrow
W^p_{k-1,\boldsymbol{\beta}-1}(\Lambda^{odd})
\end{equation*}
defined on the manifold $L$ and that each connected component of $L$ has at least one end, \textit{e.g.}
$S'$, satisfying the condition
\begin{equation*}
\left\{
\begin{array}{l}
\beta'<0 \mbox{ if $S'$ is AC}\\
\beta'>0 \mbox{ if $S'$ is CS}.
\end{array}
\right.
\end{equation*}
Now assume that $L,\hat{L}$ are compatible. Then, for all ends $S_i\in S^*$, 
 $\beta_i=\hat{\beta_i}<0$. This implies that $S'\in S^{**}$ so each connect sum $L_{\bt}$ has at least one end.

There exists $C>0$ such that, for all $f\in
W^p_{k,\boldsymbol{\beta}_t}(L_{\boldsymbol{t}})$, 
\begin{equation}\label{eq:d_invertible}
\|f\|_{W^p_{k,\boldsymbol{\beta}_{\boldsymbol{t}}}}\leq
C\|df\|_{W^p_{k-1,\boldsymbol{\beta}_{\boldsymbol{t}}-1}}.
\end{equation}
\end{theorem}
\begin{proof}
As seen above, the assumptions prove that the operator $d$ is a topological
isomorphism (onto its image) between Sobolev spaces on both manifolds $L$,
$\hat{L}$. This means that there exist constants $C_g$, $C_{\hat{g}}$ satisfying
the analogue of Equation \ref{eq:d_invertible} on both manifolds separately. We
can use $C_g$, $C_{\hat{g}}$ to build $C$ satisfying Equation
\ref{eq:d_invertible} on $L_{\boldsymbol{t}}$ using the same ideas introduced
in the proof of Theorem \ref{thm:sum_injective}. There is only one difference,
as follows. In the proof of Theorem \ref{thm:sum_injective} we use the equality
\begin{equation*}
 t^{-\beta_i}C_{\hat{g}}
\|\Delta_{\hat{g}}((1-\eta_t)f)\|_{W^p_{k-2,\hat{\boldsymbol{\beta}}-2}(\hat{g},\hat{\rho})}=t^{2-\beta_i}C_{\hat{g}} \|\Delta_{t^2\hat{g}}((1-\eta_t)f)\|_ {
W^p_{k-2,\hat{\boldsymbol{\beta}}-2}(\hat{g},\hat{\rho}) }.
\end{equation*}
The factor $t^{2-\beta_i}$ is then cancelled by rescaling. In
particular, the above equality uses the fact that the
Laplacian depends on the metric and rescales in a specific
way.

In the case at hand the operator $d$ does not depend on the metric. However,
notice that it takes functions into 1-forms: it is this property that allows us
to conclude. Specifically, setting $\alpha_t=d((1-\eta_t)f)$ and assuming $\hat{\boldsymbol{\beta}}$ is constant to simplify the
notation, we find:
\begin{align*}
\|\alpha_t\|^p_{W^p_{k-1,\hat{\boldsymbol{\beta}}-1}(\hat{g},
\hat{\rho})}
&=\sum_j\int_{\hat{L}}|\hat{\rho}^{1-\hat{\boldsymbol{\beta}}+j}
\nabla^j\alpha_t|^p_{\hat{g}\otimes\hat{g}}\hat{\rho}^{-m}\mbox{vol}_{\hat{g}}\\
&=t^{p\hat{\boldsymbol{\beta}}}\sum_j\int_{\hat{L}}|(t\hat{\rho})^{1-\hat{\boldsymbol{\beta}}+j}
\nabla^j\alpha_t|^p_{t^2\hat{g}\otimes t^2\hat{g}}(t\hat{\rho})^{-m}\mbox{vol}_{t^2\hat{g}}\\
&=t^{p\hat{\boldsymbol{\beta}}}\|\alpha_t\|^p_{W^p_{k-1,\hat{\boldsymbol{\beta}}-1}(t^2\hat{g} ,
t\hat{\rho})}.
\end{align*}
The proof can now continue as for Theorem \ref{thm:sum_injective}.
\end{proof}

Combining Theorems \ref{thm:normstequivalent} and \ref{thm:d_invertible} we
obtain the following improvement of the weighted Sobolev Embedding Theorems,
Part 1, for parametric connect sums. 
\begin{corollary}\label{cor:improved_sob}
Let $(L,g,\rho,\boldsymbol{\beta}, S^*)$,
$(\hat{L},\hat{g},\hat{\rho},\hat{\boldsymbol{\beta}},\hat{S}^*)$ be marked
compatible conifolds as in Theorem \ref{thm:d_invertible}. Define
$L_{\boldsymbol{t}}$ as in Definition
\ref{def:tconnectsum}. Then there exists
$C>0$ such that, for all $1\leq p<m$, $\boldsymbol{t}$ and $f\in
W^p_{1,\boldsymbol{\beta}_{\boldsymbol{t}}}(L_{\boldsymbol{t
}} )$ ,
\begin{equation*}
\|f\|_{L^{p^*}_{\boldsymbol{\beta}_{\boldsymbol{t}}}}\leq C
\|df\|_{L^p_{\boldsymbol{\beta}_{\boldsymbol{t}}-1}}.
\end{equation*}
\end{corollary}
\begin{remark}
 Following standard terminology in the
literature we can refer to Equation \ref{eq:d_invertible} as a ``uniform weighted Poincar\'{e} inequality'' and to Corollary \ref{cor:improved_sob} as a ``uniform
weighted
Gagliardo-Nirenberg-Sobolev inequality''. Alternatively, following \cite{hebey}
Chapter 8, the latter is a ``uniform weighted Euclidean-type Sobolev
inequality''.
\end{remark}


\bibliographystyle{amsplain}
\bibliography{accssldefs}
\end{document}